\documentclass[11pt]{article}
\usepackage{epsfig}
\usepackage{color}
\usepackage{amsmath}
\usepackage{amssymb}  %
\usepackage{comment}
\usepackage{pgfplots}
\usepackage{mathdots}
\usepackage{tikz}
\usepackage{url}
\usepackage{multirow}

\newfont{\Bb}{msbm10 scaled\magstep0}

\newtheorem{theorem}{Theorem}[section]
\newtheorem{proposition}{Proposition}[section]
\newtheorem{remark}{Remark}[section]

\usepackage{algorithm}
\usepackage{algorithmic}

\newenvironment{proof}[1][\textit{Proof}]{\begin{trivlist}
\item[\hskip \labelsep {\bfseries #1}]}{\end{trivlist}}

\newcommand{\qed}{\nobreak \ifvmode \relax \else
      \ifdim\lastskip<1.5em \hskip-\lastskip
\vskip-1.5em
   \hskip38em plus0em minus0.5em \fi \nobreak
      \vrule height0.4em width0.3em depth0.25em\fi}

\title{On Error Estimation for Reduced-order Modeling of Linear Non-parametric and Parametric Systems}\author{Lihong Feng \thanks{Max Planck Institute for Dynamics of Complex Technical Systems, Sandtorstrasse 1, D-39106 Magdeburg, Germany~{\tt feng@mpi-magdeburg.mpg.de}}
\and Peter Benner\thanks{Max Planck Institute for Dynamics of Complex Technical Systems, Magdeburg, Germany and Fakultät für Mathematik, Otto-von-Guericke Universität Magdeburg, Germany.~{\tt benner@mpi-magdeburg.mpg.de}}
}

\begin{document}



\maketitle

\begin{abstract}
Motivated by a recently proposed error estimator for the transfer function of the 
reduced-order model of a given linear dynamical system, we further develop more theoretical results in this work. Furthermore, 
we propose several variants of the error estimator, and compare those variants with the existing ones both theoretically and numerically. It has been shown that some of the proposed error estimators perform better than or equally well as the existing ones. All the error estimators considered can be easily extended to estimate 
output error of reduced-order modeling for steady linear parametric systems. 
 
\end{abstract}

\section{Introduction}

Many model order reduction (MOR) methods have been proposed during the last decades. 
For many problems, especially parametric time-dependent problems, efficient error estimation of the reduced-order model (ROM) is still critical. 



It is well-known that many a-posteriori error bounds/estimators~\cite{morRozHP08, morYan14, morZhaFLetal15, morHesGB16, morFenAB17} need compute the inf-sup constant, which appears as the denominator of the error estimator. In the numerically discretized space, the inf-sup constant corresponds to the smallest singular value of a large matrix. 
For many models from, e.g., circuit simulation, MEMS simulation, the smallest singular value can be zero at some samples of the parameter due to resonances~\cite{morHesGB15}, making the error bound unavailable at those samples. Besides, computing  the smallest singular value at many samples of the parameter is time-consuming for large-scale problems. Although some algorithms are proposed to compute a lower bound of the inf-sup constant~\cite{ morHuyetal07}, they are found to be inefficient for many problems~\cite{morSemZP18}. The error bound often overestimates the true error, especially for those systems whose smallest singular values are close to zero at many parameter values. 

In recent work~\cite{morFenB19b}, new estimators for the transfer function error, as well 
as for the output error of MOR for steady parametric systems, is proposed. The proposed error estimator avoids computing the singular values of any matrix, and depends mainly on the ROM. It is applicable to any system whose ROMs  are computed using a projection based MOR method. It is illustrated by the numerical results that the error estimator is much sharper than the error bound in~\cite{morFenAB17} for those systems with small inf-sup constants. Using the proposed error estimation, the adaptive greedy algorithm converges much faster than using the error bound from~\cite{morFenAB17}. 

Error estimation based on randomized residual for parametric steady systems, is proposed in~\cite{morSemZP18} . The output error estimation proposed there is also free of computing the inf-sup constant and can be used to estimate the transfer function error in frequency domain. We will show in this work that the error estimator in~\cite{morSemZP18} more likely underestimate the true error as compared with the error estimators in~\cite{morFenB19b} and the proposed error estimators.

Another error estimation which is independent of the inf-sup constant is proposed in~\cite{morHaiOetal18}. This error estimation is used to estimate the error of the state (solution vector). It simply uses the error between two approximate solutions computed from two ROMs divided by a saturation constant as the error estimator. As for estimation of the transfer function error or output error, trivially multiplying the output matrix norm $\|C\|$ with the error estimator could also estimate the output error, but may lead to slow error decay if $\|C\|$ is large. Moreover, a saturation constant needs to be estimated for the error estimator in~\cite{morHaiOetal18}, which needs extra computations and may cause inefficiency of the error estimator if computed without sufficient accuracy. 

The error bound in~\cite{morSchWH18} is proposed for nonlinear systems and also requires the computation of the inf-sup constant or its lower bound. Numerical issues concerning computing these quantities remain. Moreover, some assumptions on the magnitude of the inf-sup constant is needed in order to derive the error estimator. For problems whose inf-sup constants are close to zero, e.g. $O(10^{-12})$, as for the examples presented in this work, the error bound might not be tight anymore. From Lemma 2 in~\cite{morSchWH18}, it is not difficult to check that for linear problems, the error bound in~\cite{morSchWH18} is an upper bound of the error estimator $\Delta_1^{pr}$ proposed in this work when the output matrix satisfies $C(\mu)=I$, the identity matrix. The residual system needed for computing $\Delta_1^{pr}$ is called {\it error equation} in ~\cite{morSchWH18}, where a ROM of the error equation needs also to be constructed. This ROM is constructed by running a {\it separate} greedy algorithm at {\it each} iteration of the main greedy algorithm. In contrast, we simultaneously construct the ROM of the residual system (error equation) and that of the original system in one greedy algorithm.

In this work, we further explore the property of the error estimator in~\cite{morFenB19b} and propose some variants of it. Sensitivity analyses are presented to show that the proposed  error estimators may behave as error bounds when influenced with a small constant. The proposed variants are theoretically and numerically analyzed, and compared with the existing ones. Furthermore, the more general MOR framework based on Petrov-Galerkin projection is used to analyze the error estimators and to explore the corresponding theoretical and numerical properties.
In the next section, we first review the error estimator in~\cite{morFenB19b} and develop more theoretical results. Variants of the error estimator and corresponding theoretical analyses are provided in Section~\ref{sec:variants}. Section~\ref{sec:comp} theoretically compares the new error estimators with the existing ones. Section~\ref{sec:alg} provides greedy algorithms for constructing the ROMs based on the error estimators. Numerical results of all the error estimators for various problems are presented and compared in Section~\ref{sec:numerics}. Conclusions are drawn in the end.  

\section{Preliminaries and Notation}

Consider linear systems
\begin{equation}
\label{eq:FOM}
\begin{array}{rcl}
E(\mu) \frac{d}{dt}x(t, \mu)&=&A( \mu) x(t, \mu)+B( \mu) u(t), \\
y(t,\mu)&=&C(\mu)x(t, \mu)
\end{array}
\end{equation}
with or without parameters.
Here $x(t, \mu) \in \mathbb R^n$ is the state vector, $n$ is often referred to as the {\it order} of the system. The vector $ \mu:=(\mu_1,\ldots, \mu_m) \in \mathbb R^{1\times m}$ includes all of the geometrical and physical parameters. The system matrices $E( \mu), A(\mu) \in \mathbb R^{n\times n}$, and $B( \mu) \in \mathbb R^{n \times n_I}$, $C(\mu) \in \mathbb R^{n_O \times n}$ may depend on the parameters. 

The reduced-order model (ROM) of the original system can be written as
\begin{equation}
\label{eq:ROM}
\begin{array}{rcl}
\hat{E}(\mu) \frac{d}{dt}z(t,\mu)&=&\hat{A}(\mu)z(t,\mu)+\hat{B}(\mu)u(t), \\
\hat{y}(t, \mu)& = &\hat{C}( \mu)z(t,\mu),
\end{array}
\end{equation}
where $\hat{E}(\mu)=W^TE(\mu)V\in \mathbb R^{r \times r}$, $\hat {A}(\mu) =W^TA(\mu)V\in  \mathbb R^{r \times r}$, $\hat{B}( \mu)=W^TB(\mu)\in  \mathbb R^{r \times n_I}$, $\hat{C}(\mu)=C(\mu)V\in  \mathbb R^{n_O \times r}$, and $z(t,\mu) \in \mathbb R^r$ with $r \ll n$. 
Then $x(t,\mu)$ can be recovered by $x(t, \mu) \approx Vz(t,\mu)$. 

The transfer function of the original system is defined as
\begin{eqnarray}
\label{transfun}
H(\tilde \mu)=C(\mu)Q(\tilde \mu)^{-1}B(\mu),
\end{eqnarray}
where $Q(\tilde \mu)=sE(\tilde \mu)-A(\mu)$.
Similarly, the transfer function of the ROM is
$$\hat H(\tilde \mu)=\hat C(\mu)\hat Q(\tilde \mu)^{-1}\hat B(\mu),$$
where  $\hat Q(\tilde \mu)=s\hat E(\tilde \mu)-\hat A(\mu)$.
Here, $s\in \mathbb C$ is the Laplace variable in frequency domain, and $\tilde \mu:=(\mu, s)$. 
In the following, we define a primal system and a dual system, whose solutions depict the right part  $Q(\tilde \mu)^{-1}B(\mu)$ and the left part $C(\mu)Q(\tilde \mu)^{-1}$ of the transfer function $H(\tilde \mu)$, respectively.
A primal system in frequency domain is defined as
\begin{equation}
\label{eq:primal}
\begin{array}{rcl}
Q(\tilde \mu)x_{pr}(\tilde \mu)&=& B(\mu).
\end{array}
\end{equation}
The reduced primal system is then defined as
\begin{equation}
\label{eq:primal_redu}
\begin{array}{rcl}
\hat Q(\tilde \mu) z_{pr}(\tilde \mu)&=& \hat B(\mu).
\end{array}
\end{equation}
Define a dual system
\begin{equation}
\label{eq:dual}
\begin{array}{rcl}
Q^{T}(\tilde \mu) x_{du}(\tilde \mu)=C^T(\mu),
\end{array}
\end{equation}
where $x_{du}(\tilde \mu)$ solves the dual system. 
The ROM of the dual system is 
\begin{equation}
\label{eq:dual_redu}
\hat Q_{du}(\tilde \mu) z_{du}(\tilde \mu)=\hat C_{du}(\mu),
\end{equation}
where $\hat Q_{du}(\tilde \mu)=W_{du}^TQ^{T}(\tilde \mu) V_{du}$, $\hat C_{du}(\mu)=W_{du}^TC^T(\mu)$, such that $\hat x_{du}(\tilde \mu):=V_{du}z_{du}(\tilde \mu)$ well approximates $x_{du}(\tilde \mu)$. 
The ROMs of the primal and the dual systems introduce two residuals, respectively, i.e. the {\it primal residual}
\begin{equation}
\label{eq:prim_resi}
r_{pr}(\tilde \mu)=B(\mu)-Q(\tilde \mu)\hat x_{pr}(\tilde \mu)
\end{equation}
and the {\it dual residual}
\begin{equation}
\label{eq:dual_resi}
r_{du}(\tilde \mu)=C^T(\mu)-Q^T(\tilde \mu)\hat x_{du}(\tilde \mu).
\end{equation}
In the following, we first review the error estimator in~\cite{morFenB19b}, then develop more theoretical results. Several variants of the error estimator and corresponding theoretical analyses are proposed afterwards. We only consider single-input single-output (SISO) systems. Extension of the error estimator to MIMO systems as well as to output error estimation for steady linear parametric systems is detailed in~\cite{morFenB19b} and will not be repeated in this work. $|\cdot|$ denotes the absolute value of a scalar.

\section{Error estimator in~\cite{morFenB19b} and extensions}
\label{sec:review}

It is not difficult to obtain the following proposition.
\begin{proposition}
\label{prop:est1}
The error of the reduced transfer function  $\hat H(\tilde \mu)$ satisfies
$$|H(\tilde \mu)-\hat H(\tilde \mu)|=|x^T_{du}(\tilde \mu)r_{pr}(\tilde \mu)|,$$
where $x_{du}(\tilde \mu)$ solves the dual system in~(\ref{eq:dual}) and $r_{pr}(\tilde \mu)$ is defined in~(\ref{eq:prim_resi}).
\end{proposition}

\begin{proof}
\begin{equation}
\label{error}
\begin{array}{rcl}
|H(\tilde \mu)-\hat H(\tilde \mu)|&=&|C(\mu)(Q^{-1}(\tilde \mu)B(\mu)-V\hat Q^{-1}(\tilde \mu)\hat B(\mu))| \\
&=& |C(\mu)Q^{-1}(\tilde \mu)(B(\mu)-\underbrace{Q(\tilde \mu)V \hat Q^{-1}(\tilde \mu)\hat B(\mu))}_{\hat x(\tilde \mu):=Vz_{pr}(\tilde \mu)}| \\
&=&|C(\mu)Q^{-1}(\tilde \mu)r_{pr}(\tilde \mu)| \\
&=&|x^T_{du}(\tilde \mu)r_{pr}(\tilde \mu)|. 
\end{array}
\end{equation}
\qed
\end{proof}
Note that computing $x^T_{du}(\tilde \mu)$ in the last equality of~(\ref{error}) needs to solve the dual system of original large scale $n$. If we solve the ROM of the dual system instead, then $x^T_{du}(\tilde \mu)$ can be  approximated by $x_{du}(\tilde \mu) \approx \hat x_{du}(\tilde \mu)=V_{du}z_{du}(\tilde \mu)$ .
Consequently, the error of $\hat H(\tilde \mu)$ can be estimated as
\begin{equation}
\label{error_est1}
|H(\tilde \mu)-\hat H(\tilde \mu)|\approx \Delta_1(\tilde \mu):=|\hat x^T_{du}(\tilde \mu)r_{pr}(\tilde \mu)|.
\end{equation}
Clearly, the error estimator $\Delta_1(\tilde \mu)$ might underestimate the true error. To
reduce the probability of underestimation, a more robust error estimator is proposed in~\cite{morFenB19b}, which is based on the following error bound. 
\begin{theorem}~\cite{morFenB19b}
\label{theo:bound}
The error of the reduced transfer function $\hat H(\tilde \mu)$ can be bounded as
\begin{equation}
\label{eq:bound}
|H(\tilde \mu)-\hat H(\tilde \mu)| \leq \Delta_1(\tilde \mu)+|x^T_{r_{du}}(\tilde \mu)r_{pr}(\tilde \mu)|,
\end{equation}
where $x_{r_{du}}(\tilde \mu)$ is the solution to the dual-residual system defined as
\begin{equation}
\label{eq:dual_resisys}
Q^{T}(\tilde \mu) x_{r_{du}}(\tilde \mu)=r_{du}(\tilde \mu).
\end{equation} 
\end{theorem}
\begin{proof}
See~\cite{morFenB19b}. \qed
\end{proof}
Again, computing $x_{r_{du}}(\tilde \mu)$ in~(\ref{eq:bound}) requires solving a large system in~(\ref{eq:dual_resisys}). Instead, we compute the ROM of~(\ref{eq:dual_resisys}),
\begin{equation}
\label{eq:redu_dual_resi}
\tilde Q(\tilde \mu) z_{r_{du}}(\tilde \mu)=\tilde r_{du}(\tilde \mu),
\end{equation}
where $\tilde Q(\tilde \mu) =W_{r_{du}}^T Q^T(\tilde \mu)V_{r_{du}}$, $\tilde r_{du}(\tilde \mu)=W_{r_{du}}^T r_{du}(\tilde \mu)$. Then $x_{r_{du}}(\tilde \mu)\approx\hat x_{r_{du}}(\tilde \mu):=V_{r_{du}}z_{r_{du}}(\tilde \mu)$.
Finally we replace $x_{r_{du}}(\tilde \mu)$ in the error bound with $\hat x_{r_{du}}(\tilde \mu)$, and get the error estimator: 
$$|H(\tilde \mu)-\hat H(\tilde \mu)| \lesssim \Delta_1(\tilde \mu)+|\hat x^T_{r_{du}}(\tilde \mu)r_{pr}(\tilde \mu)|=:\Delta_2(\tilde \mu).$$
\begin{theorem}
\label{theo:Ubound1}
The error for the reduced transfer function $\hat H(\tilde \mu)$ can be bounded as
\begin{equation}
\label{eq:bound1}
\Delta_1(\tilde \mu)-\varepsilon_1\leq |H(\tilde \mu)-\hat H(\tilde \mu)| \leq \Delta_1(\tilde \mu)+\varepsilon_1,
\end{equation}
where $\varepsilon_1:=|(x_{du}(\tilde \mu)-\hat x_{du}(\tilde \mu))^Tr_{pr}(\tilde \mu)|\geq 0$. 
\end{theorem}

\begin{proof}
On the one hand, by Proposition~\ref{prop:est1}
\begin{equation}
\begin{array}{rcl}
|H(\tilde \mu)-\hat H(\tilde \mu)|&=&|x_{du}^T(\tilde \mu)r_{pr}(\tilde \mu)|+|\hat x_{du}^T(\tilde \mu)r_{pr}(\tilde \mu)|-|\hat x_{du}^T(\tilde \mu)r_{pr}(\tilde \mu)|\\
&=&\Delta_1\tilde \mu)+|x_{du}^T(\tilde \mu)r_{pr}(\tilde \mu)|-|\hat x_{du}^T(\tilde \mu)r_{pr}(\tilde \mu)|\\
&\leq&\Delta_1(\tilde \mu)+\varepsilon_1.
\end{array}
\end{equation}
On the other hand,
\begin{equation}
\begin{array}{rcl}
\Delta_1(\tilde \mu)&=&|\hat x_{du}^T(\tilde \mu)r_{pr}(\tilde \mu)|+|x_{du}^T(\tilde \mu)r_{pr}(\tilde \mu)|-|x_{du}^T(\tilde \mu)r_{pr}(\tilde \mu)|\\
&=&|H(\tilde \mu)-\hat H(\tilde \mu)|+|\hat x_{du}^T(\tilde \mu)r_{pr}(\tilde \mu)|-|x_{du}^T(\tilde \mu)r_{pr}(\tilde \mu)|\\
&\leq&|H(\tilde \mu)-\hat H(\tilde \mu)|+\varepsilon_1.
\end{array}
\end{equation}\qed
\end{proof}
Theorem~\ref{theo:Ubound1} shows that the true error is both lower bounded and upper bounded by $\Delta_1(\tilde \mu)$ with the influence of a small-valued $\varepsilon_1$. 
\ \\
\begin{theorem}
\label{theo:Ubound2}
The error of the reduced transfer function $\hat H(\tilde \mu)$ can be bounded as
\begin{equation}
\label{eq:bound2}
\Delta_2(\tilde \mu)-\delta_2-\varepsilon_1 \leq|H(\tilde \mu)-\hat H(\tilde \mu)| \leq \Delta_2(\tilde \mu)+\varepsilon_2
\end{equation}
where $\varepsilon_2:=|(x_{r_{du}}(\tilde \mu)-\hat x_{r_{du}}(\tilde \mu))^Tr_{pr}(\tilde \mu)|\geq 0$, $\delta_2:=|\hat x_{r_{du}}(\tilde \mu))^Tr_{pr}(\tilde \mu)|$. 
\end{theorem}

\begin{proof}
From~(\ref{eq:bound}), 
\begin{equation}
\label{eq:Ubound}
\begin{array}{rcl}
|H(\tilde \mu)-\hat H(\tilde \mu)| &\leq& \Delta_1(\tilde \mu)+|x^T_{r_{du}}(\tilde \mu)r_{pr}(\tilde \mu)|\\
&=& \Delta_1(\tilde \mu)+|\hat x^T_{r_{du}}(\tilde \mu)r_{pr}(\tilde \mu)|-|\hat x^T_{r_{du}}(\tilde \mu)r_{pr}(\tilde \mu)|+|x^T_{r_{du}}(\tilde \mu)r_{pr}(\tilde \mu)|\\
&=& \Delta_2(\tilde \mu)+|x^T_{r_{du}}(\tilde \mu)r_{pr}(\tilde \mu)|-|\hat x^T_{r_{du}}(\tilde \mu)r_{pr}(\tilde \mu)|\\
&\leq&\Delta_2(\tilde \mu)+\varepsilon_2.
\end{array}
\end{equation}
The proof of the lower bound is a direct result from the lower bound of Theorem~\ref{theo:Ubound1} and the relation between $\Delta_1(\tilde \mu)$ and $\Delta_2(\tilde \mu)$.\qed
\end{proof}
Theorem~\ref{theo:Ubound2} shows that the error estimator $\Delta_2(\tilde \mu)$ cannot underestimate the true error too much, since $\Delta_2(\tilde \mu)\geq |H(\tilde \mu)-\hat H(\tilde \mu)|-\varepsilon_2$ and $\varepsilon_2$ can be made very small by letting $\hat x_{r_{du}}(\tilde \mu)$ approximate $x_{r_{du}}(\tilde \mu)$ well. On the other hand, when $\varepsilon_2$ is small, Theorem~\ref{theo:Ubound2} implicates that $\Delta_2(\tilde \mu)$ is a tight error estimator. Furthermore, Theorem~\ref{theo:Ubound2} also provides a lower bound for the true error using $\Delta_2(\tilde \mu)$ and two small valued variables $\varepsilon_2$
and $\delta_2$. Here, $\delta_2$ cannot be large when both $r_{pr}(\tilde \mu)$ and $r_{du}(\tilde \mu)$ become small. Note that $r_{du}(\tilde \mu)$ appears on the right-hand side of the reduced dual-residual system~(\ref{eq:redu_dual_resi}) from which $\hat x^T_{r_{du}}(\tilde \mu)$ in $\delta_2$ is computed. 
\section{Error estimator variants}
\label{sec:variants}

In the following, we derive some error estimators, which can be seen as variants of the error estimators $\Delta_1(\tilde \mu)$ and $\Delta_2(\tilde \mu)$, respectively.
\subsection {Variant 1}

From the error bound in~(\ref{eq:bound}) and~(\ref{eq:dual_resisys}), we get 
\begin{equation}
\label{eq:bound_pr}
|H(\tilde \mu)-\hat H(\tilde \mu)|\leq \Delta_1(\tilde \mu)+|r_{du}^T(\tilde \mu)Q^{-1}(\tilde \mu)r_{pr}(\tilde \mu)|
\end{equation}
We see that instead of solving the dual-residual system~(\ref{eq:dual_resisys}), one can also solve the primal-residual system as below,
\begin{equation}
\label{eq:prim_resisys}
Q(\tilde \mu) x_{r_{pr}}(\tilde \mu)=r_{pr}(\tilde \mu).
\end{equation} 
Replacing $Q^{-1}(\tilde \mu)r_{pr}(\tilde \mu)$ in~(\ref{eq:bound_pr}) with $x_{r_{pr}}(\tilde \mu)$ in~(\ref{eq:prim_resisys}), we obtain
\begin{equation}
\label{error_distance_pr}
|H(\tilde \mu)-\hat H(\tilde \mu)|\leq \Delta_1(\tilde \mu)+|r_{du}^T(\tilde \mu) x_{r_{pr}}(\tilde \mu)|.
\end{equation}
If we construct the ROM of the primal-residual system in~(\ref{eq:prim_resisys}), i.e.
\begin{equation}
\label{eq:redu_prim_resi}
W_{r_{pr}}^TQ(\tilde \mu)V_{r_{pr}}z_{r_{pr}}(\tilde \mu)=W_{r_{pr}}^Tr_{pr}(\tilde \mu),
\end{equation} 
then we obtain a variant of $\Delta_2(\tilde \mu)$,
$$|H(\tilde \mu)-\hat H(\tilde \mu)| \lesssim \Delta_1(\tilde \mu)+|r_{du}^T(\tilde \mu) \hat x_{r_{pr}}(\tilde \mu)|=:\Delta_2^{pr}(\tilde \mu),$$
where $\hat x_{r_{pr}}(\tilde \mu)=V_{r_{pr}}z_{r_{pr}}$ is computed from~(\ref{eq:redu_prim_resi}), the ROM of the primal-residual system and approximates
the state vector $x_{r_{pr}}(\tilde \mu)$ of the primal-residual system.
We obtain a similar sensitivity analysis for $\Delta_2^{pr}(\tilde \mu)$ presented in Theorem~\ref{theo:Ubound2pr}.
\begin{theorem}
\label{theo:Ubound2pr}
The error of the reduced transfer function $\hat H(\tilde \mu)$ can be bounded as
\begin{equation}
\label{eq:bound2pr}
\Delta_2^{pr}(\tilde \mu)-\delta_2^{pr}-\varepsilon_1\leq |H(\tilde \mu)-\hat H(\tilde \mu)| \leq \Delta_2^{pr}(\tilde \mu)+\varepsilon_2^{pr}
\end{equation}
where $\varepsilon_2^{pr}:=|r_{du}^T(\tilde \mu)(x_{r_{pr}}(\tilde \mu)-\hat x_{r_{pr}}(\tilde \mu))|\geq 0$ and $\delta_2^{pr}:=|r_{du}^T(\tilde \mu) \hat x_{r_{pr}}(\tilde \mu)|$.
\end{theorem}

\begin{proof}
The result can be obtained by using~(\ref{error_distance_pr}) and following similar steps as in the proof of Theorem~\ref{theo:Ubound2}.\qed
\end{proof}
Note that $\varepsilon_2^{pr}$ will be of small value once the reduced solution $\hat x_{r_{pr}}(\tilde \mu)$ approximates $x_{r_{pr}}(\tilde \mu)$, the solution to the  primal-residual system~(\ref{eq:prim_resisys}), well.
\subsection {Variant 2}

From~(\ref{error}), we know
$$
|H(\tilde \mu)-\hat H(\tilde \mu)|=|C(\mu)Q^{-1}(\tilde \mu)r_{pr}(\tilde \mu)|. 
$$
Similarly, if we use the solution to the  primal-residual system~(\ref{eq:prim_resisys}) to replace 
$Q^{-1}(\tilde \mu)r_{pr}(\tilde \mu)$,
then we get
\begin{equation}
\label{error_variant2}
|H(\tilde \mu)-\hat H(\tilde \mu)|=|C(\mu)x_{r_{pr}}(\tilde \mu)|. 
\end{equation}
If further using the ROM~(\ref{eq:redu_prim_resi}) to compute an approximate state, then $x_{r_{pr}}(\tilde \mu)$ in~(\ref{error_variant2}) can be approximated by $\hat x_{r_{pr}}(\tilde \mu)$. We obtain the following error estimation
$$|H(\tilde \mu)-\hat H(\tilde \mu)|\approx|C(\mu) \hat x_{r_{pr}}(\tilde \mu)|=:\Delta_1^{pr}(\tilde \mu),$$
which can be considered as a variant of $\Delta_1(\tilde \mu)$.
\begin{theorem}
\label{theo:Ubound1pr}
The error of the reduced transfer function $\hat H(\tilde \mu)$ can be bounded as
\begin{equation}
\label{eq:bound1pr}
\Delta_1^{pr}(\tilde \mu)-\varepsilon_1^{pr}\leq |H(\tilde \mu)-\hat H(\tilde \mu)| \leq \Delta_1^{pr}(\tilde \mu)+\varepsilon_1^{pr},
\end{equation}
where $\varepsilon_1^{pr}:=|C(\mu)(x_{r_{pr}}(\tilde \mu)-\hat x_{r_{pr}}(\tilde \mu))|\geq 0$. 
\end{theorem}

\begin{proof}
The proof is similar to that of Theorem~\ref{theo:Ubound1} and therefore not be repeated here. \qed
\end{proof}

\subsection {Variant 3}

The next theorem presents an error bound based on $\Delta_1^{pr}(\tilde \mu)$, from which we get another variant of $\Delta_2(\tilde \mu)$.
\begin{theorem}
\label{theo:bound2}
The error of the reduced transfer function $\hat H(\tilde \mu)$ can be bounded as
$$|H(\tilde \mu)-\hat H(\tilde \mu)|\leq \Delta_1^{pr}(\tilde \mu)+|x_{du}^T(\tilde \mu)r_{r_{pr}}(\tilde \mu)|,$$
where $r_{r_{pr}}$ is the residual of the approximate solution $\hat x_{r_{pr}}(\tilde \mu)$ to the primal-residual system in~(\ref{eq:prim_resisys}), i.e. $r_{r_{pr}}=r_{pr}(\tilde \mu) -Q\hat x_{r_{pr}}(\tilde \mu)$.
\end{theorem}

\begin{proof}
From~(\ref{error_variant2}), the true error can be presented as
\begin{equation}
\label{error_variant3}
|H(\tilde \mu)-\hat H(\tilde \mu)|=|C(\mu) x_{r_{pr}}(\tilde \mu)|. 
\end{equation}
We check the distance between the true error $|C(\mu) x_{r_{pr}}(\tilde \mu)|$ and its estimator $\Delta_1^{pr}(\tilde \mu)=|C(\mu) \hat x_{r_{pr}}(\tilde \mu)|$,
\begin{equation}
\label{error_var_distance}
\begin{array}{rcl}
|C(\mu) x_{r_{pr}}(\tilde \mu)| -|C(\mu) \hat x_{r_{pr}}(\tilde \mu)|
&\leq&|C(\mu) Q^{-1}r_{pr}(\tilde \mu) -C(\mu) \hat x_{r_{pr}}(\tilde \mu)|\\
&=&|C(\mu) Q^{-1}[\underbrace{r_{pr}(\tilde \mu) -Q\hat x_{r_{pr}}(\tilde \mu)}_{=:r_{r_{pr}}(\tilde \mu)}]|.\\
\end{array}
\end{equation}
Combining~(\ref{error_variant3}) and~(\ref{error_var_distance}), we get
\begin{equation}
\label{error_esti3}
\begin{array}{rcl}
|H(\tilde \mu)-\hat H(\tilde \mu)|&\leq& 
|C(\mu) \hat x_{r_{pr}}(\tilde \mu)|+|C(\mu) Q^{-1}r_{r_{pr}}(\tilde \mu)| \\
&=&\Delta_1^{pr}(\tilde \mu)+|x_{du}^T(\tilde \mu)r_{r_{pr}}(\tilde \mu)|.
\end{array}
\end{equation}
\qed
\end{proof}
Similarly, we get the following error estimator by approximating $x_{du}(\tilde\mu)$ with $\hat x_{du}(\tilde\mu)$.
$$|H(\tilde \mu)-\hat H(\tilde \mu)|\lesssim \Delta_1^{pr}(\tilde \mu)+|\hat x^T_{du}(\tilde \mu)r_{r_{pr}}(\tilde \mu)|=:\Delta_3(\tilde \mu).$$
\begin{theorem}
\label{theo:Ubound3}
The error of the reduced transfer function $\hat H(\tilde \mu)$ can be bounded as
\begin{equation}
\label{eq:bound3}
 \Delta_3(\tilde \mu)-\delta_3-\varepsilon_1^{pr}\leq|H(\tilde \mu)-\hat H(\tilde \mu)| \leq \Delta_3(\tilde \mu)+\varepsilon_3
\end{equation}
where $\varepsilon_3:=|(x_{du}(\tilde \mu)-\hat x_{du}(\tilde \mu))^Tr_{pr}(\tilde \mu)|\geq0$ and $\delta_3=|\hat x_{du}^T(\tilde \mu)r_{r_{pr}}(\tilde \mu)|$. 
\end{theorem}

\begin{proof}
The result can be obtained by using~(\ref{error_esti3}), the relation between $\Delta_3(\tilde \mu)$ and $\Delta_1^{pr}(\tilde \mu)$, and the lower bound of Theorem~\ref{theo:Ubound1pr}, then following similar steps as in the proof of Theorem~\ref{theo:Ubound2}.\qed
\end{proof}
Analogously, $\varepsilon_3$ is also a small number, since $\hat x_{du}(\tilde \mu)$ is close enough to $x_{du}(\tilde \mu)$ if it is a good approximation computed from the ROM of the dual system. 
\subsection {Variant 4}

In~(\ref{error_esti3}), if we consider $Q^{-1}r_{r_{pr}}$ and seek the solution to
the primal-residual-residual system,
\begin{equation}
\label{eq:prim_resi_resi}
Q(\tilde \mu)x_{r_{rpr}}(\tilde \mu)=r_{r_{pr}}(\tilde \mu),
\end{equation}
then the error bound in~(\ref{error_esti3}) becomes
\begin{equation}
\label{error_esti3_pr}
\begin{array}{rcl}
|H(\tilde \mu)-\hat H(\tilde \mu)|&\leq& 
\Delta_1^{pr}(\tilde \mu)+|C(\mu) x_{r_{rpr}}(\tilde \mu)|.
\end{array}
\end{equation}
Certainly, we can compute the ROM of~(\ref{eq:prim_resi_resi}),
\begin{equation}
\label{eq:prim_resi_resi_redu}
W_{r_{rpr}}^TQ(\tilde \mu)V_{r_{rpr}} z_{r_{rpr}}(\tilde \mu)=W_{r_{rpr}}^Tr_{r_{pr}}(\tilde \mu),
\end{equation}
and replace  
$x_{r_{rpr}}(\tilde \mu)$ in~(\ref{error_esti3_pr}) with its approximation $\hat  x_{r_{rpr}}(\tilde \mu)=V_{r_{rpr}}z_{r_{rpr}}(\tilde \mu)$ computed from the ROM. Finally, we get the error estimator as below,
 $$|H(\tilde \mu)-\hat H(\tilde \mu)|\lesssim |\Delta_1^{pr}(\tilde \mu)|+|C(\mu)\hat x_{r_{rpr}}(\tilde \mu)|=:\Delta_3^{pr}(\tilde \mu).$$
From~(\ref{error_esti3_pr}), we can get the following lower and upper bound using the error estimator $\Delta_3^{pr}(\tilde \mu)$.
\begin{theorem}
\label{theo:Ubound4}
The error of the reduced transfer function $\hat H(\tilde \mu)$ can be bounded as
\begin{equation}
\label{eq:bound4}
\Delta_3^{pr}(\tilde \mu)-\delta_3^{pr}-\varepsilon_1^{pr}\leq|H(\tilde \mu)-\hat H(\tilde \mu)| \leq \Delta_3^{pr}(\tilde \mu)+\varepsilon_3^{pr}
\end{equation}
where $\varepsilon_3^{pr}:=|C(\mu)(x_{r_{rpr}}(\tilde \mu)-\hat x_{r_{rpr}}(\tilde \mu))|\geq 0$ and $\delta_3^{pr}:=|C(\mu) \hat x_{r_{rpr}}(\tilde \mu)|$. 
\end{theorem}

\begin{proof}
The result can be obtained by using~(\ref{error_esti3_pr}) and following similar steps as in the proof of Theorem~\ref{theo:Ubound3}.\qed
\end{proof}

\subsection{Relations among the error estimators}
\label{sec:relation}

In this section we explore relations among the error estimators discussed in the previous two sections and present the following propositions.

\begin{proposition}
\label{prop:VduV}
If $W_{du}=V$, and $V_{du}=W$, then $\Delta_1(\tilde \mu)=0$.
\end{proposition}
\begin{proof}
\begin{equation}
\label{eq:est1_analy}
\begin{array}{rcl}
\Delta_1(\tilde \mu)&=&|\hat x^T_{du}(\tilde \mu)r_{pr}(\tilde \mu)| \\
&=&|\hat x^T_{du}(\tilde \mu)(B(\mu)-Q(\tilde \mu)V(W^TQ(\tilde \mu)V)^{-1}W^TB(\mu)| \\
&=&|\hat x^T_{du}(\tilde \mu)B(\mu)-\hat x^T_{du}(\tilde \mu)Q(\tilde \mu)V(W^TQ(\tilde \mu)V)^{-1}W^TB(\mu)|.
\end{array}
\end{equation}
The first part of the last equation in~(\ref{eq:est1_analy}) is
\begin{equation}
\label{eq:est1_analy1}
\begin{array}{rcl}
\hat x^T_{du}(\tilde \mu)B(\mu)&=& [V_{du}(W_{du}^TQ^T(\tilde \mu)V_{du})^{-1}W_{du}^TC^T(\mu)]^TB(\mu)\\
&=& C(\mu)V(W^TQ(\tilde \mu)V)^{-1}W^TB(\mu)  \ (\textrm{if} \ W_{du}=V  \ \textrm{and} \ V_{du}=W). \\
\end{array}
\end{equation}
If $W_{du}=V$ and $V_{du}=W$, the second part of the last equation in~(\ref{eq:est1_analy}) becomes
\begin{equation}
\label{eq:est1_analy2}
\begin{array}{rcl}
\hat x^T_{du}(\tilde \mu)V(W^TQ(\tilde \mu)V)^{-1}W^TB(\mu)&=& [V_{du}(W_{du}^TQ^T(\tilde \mu)V_{du})^{-1}W_{du}^TC^T(\mu)]^TQ(\tilde \mu)V(W^TQ(\tilde \mu)V)^{-1}W^TB(\mu)\\
&=& C(\mu)V(W^TQ(\tilde \mu)V)^{-1}W^TQ(\tilde \mu)V(W^TQ(\tilde \mu)V)^{-1}W^TB(\mu) \\
&=&C(\mu)V(W^TQ(\tilde \mu)V)^{-1}W^TB(\mu). \\
\end{array}
\end{equation}
Comparing~(\ref{eq:est1_analy1}) and (\ref{eq:est1_analy2}), we get the conclusion. 
\qed
\end{proof}

\begin{remark}
\label{rem:VduV}
Proposition~\ref{prop:VduV} points out that if $W_{du}=V$ and $V_{du}=W$, then $\Delta_1(\tilde \mu)$ is always zero, and cannot be a good error estimator. This is not the case for 
most problems. However, if the system is symmetric, i.e., $Q(\tilde \mu)=Q^T(\tilde \mu)$, and $B(\mu)=C^T(\mu)$, this will likely happen, since in this case, the primal system and the dual system are identical. We will show later that for systems which are almost symmetric, i.e. $Q(\tilde \mu)\approx Q^T(\tilde \mu)$ and/or $B(\mu) \approx C^T(\mu)$, $\Delta_1(\tilde \mu)$ also behaves badly.  One possibility of avoiding $\Delta_1(\tilde \mu)$ being zero or improving the performance of $\Delta_1(\tilde \mu)$ is to construct $(W_{du}, V_{du})$ and $(W,V)$ from different subspaces of the solution (state) manifold.  More specifically, when using time domain methods, different snapshots should be chosen for $(W_{du}, V_{du})$ and $(W,V)$, respectively; or different expansion points should be taken if using frequency domain methods, e.g., moment-matching. 
\end{remark}
\begin{remark}
Using Galerkin projection, i.e. $W=V$, $W_{du}=V_{du}$, then $V_{du}=V$ leads to $\Delta_1(\tilde \mu)=0$. 
\end{remark}
\begin{proposition}
\label{prop:VdurVdu}
If $W_{r_{du}}=W_{du}$, then the second part of $\Delta_2(\tilde \mu)$ is always zero, i.e. $|\hat x^T_{r_{du}}(\tilde \mu)r_{pr}(\tilde \mu)|=0$.
\end{proposition}
\begin{proof}
\begin{equation}
\label{eq:delta2_e2}
\begin{array}{rcl}
\hat x^T_{r_{du}}(\tilde \mu)r_{pr}(\tilde \mu)&=&[V_{r_{du}}\tilde Q^{-1}(\tilde \mu)(W_{r_{du}}^Tr_{du}(\tilde \mu))]^Tr_{pr}(\tilde \mu)\\
&=&r_{du}^T(\tilde \mu)W_{du}\tilde Q^{-T}(\tilde \mu) V_{r_{du}}^Tr_{pr}(\tilde \mu) \  (\textrm{if} \  W_{r_{du}}=W_{du}).
\end{array}
\end{equation}
Considering the first two terms in the last equation, we get
\begin{equation}
\label{eq:delta2_e2_1}
\begin{array}{rcl}
(r_{du}^T(\tilde \mu)W_{du})^T&=&W_{du}^T(C^T(\mu)-Q^T(\tilde \mu)V_{du}z_{du}(\tilde \mu))\\
&=&0 \quad (\textrm{due to}\ (\ref{eq:dual_redu})).
\end{array}
\end{equation}
\qed
\end{proof}
\begin{remark}
Proposition~\ref{prop:VdurVdu} points out that if $W_{r_{du}}=W_{du}$, then $\Delta_2(\tilde \mu)$ reduces to $\Delta_1(\tilde \mu)$, and cannot be more robust than  $\Delta_1(\tilde \mu)$. Therefore, $W_{r_{du}}$ should be carefully constructed. In case of Galerkin projection, i.e. $W_{r_{du}}=V_{r_{du}}$, and $W_{du}=V_{du}$, then $V_{r_{du}}=V_{du}$ leads to the same result in Proposition~\ref{prop:VdurVdu}.
\end{remark}
\begin{proposition}
\label{prop:VrprV}
If $W_{r_{pr}}=W$, then $\hat x_{r_{pr}}(\tilde \mu)=0$.
\end{proposition}
\begin{proof}
From the ROM of the primal-residual system in~(\ref{eq:redu_prim_resi}),
\begin{equation}
\label{eq:hatxrpr}
\begin{array}{rcl}
\hat x_{r_{pr}}(\tilde \mu)&=&V_{r_{pr}}(W_{r_{pr}}^TQ(\tilde \mu)V_{r_{pr}})^{-1}(W_{r_{pr}}^Tr_{pr}(\tilde \mu))\\
&=&V_{r_{pr}}(W_{r_{pr}}^TQ(\tilde \mu)V_{r_{pr}})^{-1}(W^Tr_{pr}(\tilde \mu))  \quad (\textrm{if} \ W_{r_{pr}}=W)\\
&=&V_{r_{pr}}(W_{r_{pr}}^TQ(\tilde \mu)V_{r_{pr}})^{-1}W^T(B(\mu)-Q(\tilde \mu)Vz_{pr}(\tilde \mu))  \\
&=&V_{r_{pr}}(W_{r_{pr}}^TQ(\tilde \mu)V_{r_{pr}})^{-1}[W^TB(\mu)-W^TQ(\tilde \mu)Vz_{pr}(\tilde \mu)] \\
&=&0. \quad (\textrm{due to}\ (\ref{eq:primal_redu}) ).
\end{array}
\end{equation}
\qed
\end{proof}
\begin{remark}
Proposition~\ref{prop:VrprV} implicates that if $W_{r_{pr}}=W$, then the second part of $\Delta_2^{pr}(\tilde \mu)$ is always zero, i.e. $|r_{du}^T(\tilde \mu) \hat x_{r_{pr}}(\tilde \mu)|=0$ , and $\Delta_2^{pr}(\tilde \mu)$ equals to $\Delta_1(\tilde \mu)$. Also, $\hat x_{r_{pr}}(\tilde \mu)=0$ makes $\Delta_1^{pr}(\tilde \mu)$ zero, meaning the first part of $\Delta_3(\tilde\mu)$ and the first part of $\Delta_3^{pr}(\tilde \mu)$ are all zeros. Therefore, $W_{r_{pr}}$ should also be carefully constructed to avoid being equal to $W$. For Galerkin projection, i.e. $W_{r_{pr}}=V_{r_{pr}}$ and $W=V$, 
Proposition~\ref{prop:VrprV} reads: If $V_{r_{pr}}=V$, then $\hat x_{r_{pr}}(\tilde \mu)=0$.
\end{remark}

\begin{proposition}
\label{prop:VrrprV}
If $W_{r_{rpr}}=W_{r_{pr}}$, then $\hat x_{r_{rpr}}(\tilde \mu)=0$.
\end{proposition}
\begin{proof}
From the ROM of the primal-residual-residual system in~(\ref{eq:prim_resi_resi_redu}),
\begin{equation}
\label{eq:hatxrrpr}
\begin{array}{rcl}
\hat x_{r_{rpr}}(\tilde \mu)&=&V_{r_{rpr}}(W_{r_{rpr}}^TQ(\tilde \mu)V_{r_{rpr}})^{-1}(W_{r_{rpr}}^Tr_{r_{pr}}(\tilde \mu)).\\
\end{array}
\end{equation}
The last tow terms of the right-hand side of~(\ref{eq:hatxrrpr}) are
\begin{equation*}
\label{eq:hatxrrpr1}
\begin{array}{rcl}
W_{r_{rpr}}^Tr_{r_{pr}}(\tilde \mu)&=&W_{r_{pr}}^Tr_{r_{pr}}(\tilde \mu) \quad (\textrm{if} \ W_{r_{rpr}}=W_{r_{pr}}) \\
&=&W_{r_{pr}}^T(r_{pr}(\tilde \mu)-QV_{r_{pr}}z_{r_{pr}}(\tilde \mu))\\
&=&0 \quad (\textrm{due to}\ (\ref{eq:redu_prim_resi}) ).
\end{array}
\end{equation*}\qed
\end{proof}
\begin{remark}
From Proposition~\ref{prop:VrrprV}, we see that if $W_{r_{rpr}}=W_{r_{pr}}$, then the second part of $\Delta_3^{pr}(\tilde \mu)$ is always zero, i.e. $|C(\mu) \hat x_{r_{rpr}}(\tilde \mu)|=0$ , and is no better than $\Delta_1^{pr}(\tilde \mu)$ in underestimating the true error. Similarly, in case of Galerkin projection,  i.e. $W_{r_{pr}}=V_{r_{pr}}$ and $W_{r_{rpr}}=V_{r_{rpr}}$, 
Proposition~\ref{prop:VrrprV} reads: If $V_{r_{rpr}}=V_{r_{pr}}$, then $\hat x_{r_{rpr}}(\tilde \mu)=0$.
\end{remark}

\subsection{Constructing projection matrices for the ROMs}
\label{subsec:construct_ROMs}
The key components for computing the error estimators are the projection matrix pairs $(W,V)$, $(W_{du},V_{du})$, $(W_{r_{du}},V_{r_{du}})$ or $(W_{r_{pr}},V_{r_{pr}})$, $(W_{r_{rpr}},V_{r_{rpr}})$ which are used to construct the reduced systems in~(\ref{eq:primal_redu}), (\ref{eq:dual_redu}), (\ref{eq:redu_dual_resi}) or in~(\ref{eq:redu_prim_resi}), (\ref{eq:prim_resi_resi_redu}), respectively. For simplicity and clarity of analysis, we only use Galerkin projection for all the reduced systems, so that only one projection matrix $V, V_{du}$, $V_{r_{du}}$ or $V_{r_{pr}}$, $V_{r_{rpr}}$ needs to be computed for each reduced system. The analysis in this subsection can be extended to Petrov-Galerkin projection without many difficulties and could be addressed in a future work.

By definition of the reduced primal system~(\ref{eq:primal_redu}), $V$ is also the projection matrix for constructing the ROM of the original model. 
Since the proposed error estimator does not depend on the MOR method, $V$ can be computed either using time-domain MOR methods, such as the reduced basis (RB) method, the proper orthogonal decomposition (POD) method~\cite{morBenGW15, morBauBHetal17}, which use the snapshots in time domain (trajectories of the state vector $x$) to obtain $V$ or using frequency domain methods, such as multi-moment-matching~\cite{morFenB14}. 

The dual system~(\ref{eq:dual}), the dual-residual system~(\ref{eq:dual_resisys}), as well as the primal-residual system~(\ref{eq:prim_resisys}), the primal-residual-residual system~(\ref{eq:prim_resi_resi}) are parametric systems in frequency domain, with $\tilde \mu=s$ or $\tilde \mu=(\mu, s)$ being the vector of parameters. 
Similarly, we can compute the projection matrices for MOR of these systems either through snapshot based methods, or the multi-moment-matching method. The snapshots do not represent the trajectory of the solution in time domain, instead, they are the solution vectors at different samples of the parameter $\tilde \mu$. 

In order to be consistent with the previous work in~\cite{morFenAB17, morFenB19b}, and to be comparable with existing results, we apply the frequency domain method, i.e., the multi-moment-matching method~\cite{morFenB14} to derive the ROMs for all the systems contributing to the error estimator. To be self-contained, we also review the construction of $
V, V_{du}$ and $V_{r_{du}}$, though it is detailed in~\cite{morFenB19b}. It is illustrated in~\cite{morFenB19b} that the reduced basis method can be seen as a special case of the multi-moment-matching method for systems in frequency domain. 

\subsubsection{Constructing $V$ using the multi-moment-matching method~\cite{morFenB14}}

When using the multi-moment-matching method proposed in~\cite{morFenB14} to construct the ROM, then $V$ can be computed as follows. We first consider the state vector $x(t,\mu)$ in frequency domain, i.e., the state vector $x(\tilde \mu)$ of the 
primal system. Assume that $Q(\tilde \mu)$ has the following affine decomposition %
$$Q(\tilde \mu)=Q_0+h_1(\tilde \mu) Q_1+\ldots+h_p(\tilde \mu) Q_p,$$
where $h_j(\tilde \mu): \mathbb C^m \mapsto \mathbb C, j=1,\ldots,p$ are scalar functions of $\tilde \mu$. From the series expansion of $x(\tilde \mu)$,
\begin{equation}
\label{eq:x_expan}
\begin{array}{rcl}
 x(\tilde \mu)&=&[Q(\tilde \mu)]B(\mu)\\
&=&[Q_0+h_1(\tilde \mu)Q_1+\ldots+h_p(\tilde \mu)Q_p]^{-1}B(\mu)\\
&=&[I-(\sigma_1M_1+\ldots +\sigma_pM_p)]^{-1}B_M\\
 &=&\sum\limits_{k=0}^{\infty}(\sigma_1M_1+\ldots+\sigma_pM_p)^kB_M,
\end{array}
\end{equation}
where $\sigma_j=h_j(\tilde \mu)-h_j(\tilde \mu^i)$, $B_M=[Q(\tilde \mu^i)]^{-1}B(\mu)$, $M_j=-[Q(\tilde \mu^i)]^{-1}Q_j$, $j=1,2,\ldots,p$; $h(\tilde \mu^i):=(h_1(\tilde \mu^i), \ldots, h_p(\tilde \mu^i))$ is the expansion point at which the above power series of $x(\tilde \mu)$ is derived. Since $h(\tilde \mu^i)$ is uniquely determined by $\tilde \mu^i$, we call $\tilde \mu^i$ the expansion point in the following text, for simplicity. There exist recursions between the coefficients of the series expansion as below,
\begin{equation}
\label{moments}
\begin{array}{rcl}
R_0&=&\tilde B_M,\\ \ R_1&=&[M_1R_0,\ldots, M_pR_0], \\
R_2&=&[M_1R_1,\ldots, M_pR_1], \\
\vdots \\
R_q&=&[M_1R_{q-1},\ldots, M_pR_{q-1}],\\
\vdots
\end{array}
\end{equation}
Here, $\tilde B_M=B_M$, if $B(\mu)$ does not depend on $\mu$, i.e. $B(\mu)=B$. Otherwise, $\tilde B_M=[B_{M_1}, \ldots, B_{M_p}]$, $B_{M_j}=[Q(\tilde \mu^i)]^{-1}B_j$, $j=1,\ldots,p$, if $B(\mu)$ can be written in an affine form, e.g., $B(\mu)=B_1 \alpha_1(\mu) +\ldots + B_p \alpha_p(\mu)$, $\alpha_i(\mu): \mathbb C^m \mapsto \mathbb C$.
Then $V_{\tilde \mu^i}$ is computed as
\begin{equation}
\label{eq:Vi}
\mathop{\mathrm{range}}(V_{\tilde \mu^i})=\mathop{\mathrm{span}}\{R_0, R_1,\ldots, R_q \}_{\tilde \mu^i},
\end{equation}
where usually we require $q\leq 1$ to avoid exponential increase of column dimension. The matrix $V_{\tilde \mu^i}$
depends on the expansion point $\tilde \mu^i$. 
Finally, $V$
can be constructed as
\begin{equation}
\label{eq:V}
V=
\textrm{orth}\{V_{\tilde \mu^1},\ldots, V_{\tilde \mu^l}\}.
\end{equation}

\subsubsection{Constructing $V_{du}$ using multi-moment-matching}
If using the multi-moment-matching method, $V_{du}$ can also be constructed similarly as $V$. Considering the dual system in~(\ref{eq:dual}), $x_{du}(\tilde \mu)$ can be written as
\begin{equation}
\label{eq:xdu_expan}
\begin{array}{rcl}
 x_{du}(\tilde \mu)&=&[Q(\tilde \mu)]^{-T}C^T(\mu)\\
&=&[Q_0^T+h_1(\tilde \mu)Q_1^T+\ldots+h_p(\tilde \mu)Q_p^T]^{-1}C^T(\mu)\\
&=&[I-(\sigma_1\tilde M_1+\ldots +\sigma_p \tilde M_p)]^{-1}C_M\\
 &=&\sum\limits_{k=0}^{\infty}(\sigma_1 \tilde M_1+\ldots+\sigma_p \tilde M_p)^k C_M,
\end{array}
\end{equation}
where 
$C_M=[Q(\tilde \mu^i)]^{-T}C^T(\mu)$, $\tilde M_j=-[Q(\tilde \mu^i)]^{-T}Q_j^T$, $j=1,2,\ldots,p$. The recursions between the coefficients of the series expansion in~(\ref{eq:xdu_expan}) are
\begin{equation}
\label{moments}
\begin{array}{rcl}
\tilde R_0&=&\tilde C_M,\\ \ \tilde R_1&=&[\tilde M_1\tilde R_0,\ldots, \tilde M_p \tilde R_0], \\
\tilde R_2&=&[\tilde M_1 \tilde R_1,\ldots, \tilde M_p \tilde R_1], \\
\vdots \\
\tilde R_q&=&[\tilde M_1\tilde R_{q-1},\ldots, \tilde M_p \tilde R_{q-1}],\\
\vdots
\end{array}
\end{equation}
Here, $\tilde C_M=C_M$, if $C(\mu)$ does not depend on $\mu$, i.e. $C(\mu)=C$. Otherwise, $\tilde C_M=[C_{M_1}, \ldots, C_{M_p}]$, $C_{M_i}=[Q(\tilde \mu^i)]^{-1}C_j$, $j=1,\ldots,p$, if $C(\mu)$ can be written in an affine form, e.g., $C(\mu)=C_1 \beta_1(\mu) +\ldots + C_p \beta_p(\mu)$.
Then $V^{du}_{\tilde \mu^i}$ is computed as
\begin{equation}
\label{eq:Vidu}
\mathop{\mathrm{range}}(V^{du}_{\tilde \mu^i})=\mathop{\mathrm{span}}\{\tilde R_0, \tilde R_1,\ldots, \tilde R_q \}_{\tilde \mu^i}. 
\end{equation}
Finally, $V_{du}$
can be constructed as
\begin{equation}
\label{eq:Vdu}
V_{du}=
\textrm{orth}\{V^{du}_{\tilde \mu^1},\ldots, V^{du}_{\tilde \mu^l}\}.
\end{equation}

\subsubsection{Constructing $V_{r_{du}}$}
$V_{r_{du}}$ is used to construct the ROM of the dual-residual system and the error estimator $\Delta_2(\tilde \mu)$.
From the
state vector of the dual-residual system~(\ref{eq:dual_resisys}), we see that
\begin{equation}
\label{eq:xrdu}
\begin{array}{rl}
x_{r_{du}}(\tilde \mu)&=Q^{-T}(\tilde \mu)r_{du}(\tilde \mu)\\
&=Q^{-T}(\tilde \mu)C^T(\mu)-\hat x_{du}(\tilde \mu)\\
&= Q^{-T}(\tilde \mu)C^T(\mu)-V_{du}z_{du}(\tilde \mu),
\end{array}
\end{equation}
where $Q^{-T}(\tilde \mu)C^T( \mu)$ is nothing but the state vector $x_{du}(\tilde \mu)$ of the dual system. 

Considering the series expansion of $x_{du}(\tilde \mu)$ in~(\ref{eq:xdu_expan}), we see that taking the same expansion point as in~(\ref{eq:xdu_expan}), the series expansion leads to the subspace $\textrm{range}(V_{du})$. Finally, $Q^{-T}(\tilde \mu)C^T( \mu)$ in the last equality of~(\ref{eq:xrdu}) provides no new information than $V_{du}$, so that we can use $\textrm{range}(V_{du})$ as the subspace for approximating the trajectory space of $x_{r_{du}}(\tilde \mu)$, i.e. $V_{r_{du}}=V_{du}$.
However, from Proposition~\ref{prop:VdurVdu}, we know that $V_{r_{du}}$ should be different from $V_{du}$. 
Therefore, if we use expansion points different from those used for $V_{du}$ to obtain a second projection matrix $V_{r_{du}}^1$ which is different from $V_{du}$, then the projection matrix $V_{r_{du}}:=\textrm{orth}\{V_{r_{du}}^1, V_{du}\}$ should represent
the trajectory of $x_{r_{du}}(\tilde \mu)$ well. 

$V_{r_{du}}^1$ can be computed using the multi-moment-matching method as in~(\ref{eq:Vidu}) and~(\ref{eq:Vdu}), by choosing expansion points which are different from those used there, i.e.
\begin{equation}
\label{eq:Vidur1}
\textrm{range} (V^{r_{du}}_{\tilde \mu^j})=\mathop{\mathrm{span}}\{\tilde R_0, \tilde R_1,\ldots, \tilde R_q \}_{\tilde \mu^j},
j=1,\ldots,l.
\end{equation}
Finally, 
\begin{equation}
\label{eq:Vdur}
\textrm{range}(V_{r_{du}})=\textrm{orth}\{V^{r_{du}}_{\tilde \mu^1},\ldots,V^{r_{du}}_{\tilde \mu^l}, V_{du}\}
\end{equation}
The $\tilde \mu^j$ in~(\ref{eq:Vidur1}) can be selected by a greedy algorithm searching the maximum of $|\hat x^T_{r_{du}}(\tilde \mu)r_{pr}(\tilde \mu)|$, the first part of $\Delta_2(\tilde \mu)$ associated with $\hat x_{r_{du}}$, and are usually different from $\tilde \mu^i$ used for computing $V_{du}$. 

\subsubsection{Constructing $V_{r_{pr}}$}
From the
state vector of the primal-residual system~(\ref{eq:prim_resisys}), we get
\begin{equation}
\label{eq:xrpr}
\begin{array}{rl}
x_{r_{pr}}(\tilde \mu)&=Q^{-1}(\tilde \mu)r_{pr}(\tilde \mu)\\
&=Q^{-1}(\tilde \mu)B(\mu)-\hat x_{pr}(\tilde \mu)\\
&= Q^{-1}(\tilde \mu)B(\mu)-Vz_{pr}(\tilde \mu),
\end{array}
\end{equation}
where $Q^{-1}(\tilde \mu)B( \mu)$ is exactly the state vector $x(\tilde \mu)$ of the primal system. 


Similarly as constructing $V_{r_{du}}$, we use expansion points different from those used for $V$ to obtain a second projection matrix $V_{r_{pr}}^1$ which is as different as $V$, then the projection matrix 
\begin{equation}
\label{eq:Vrpr}
V_{r_{pr}}:=\textrm{orth}\{V_{r_{pr}}^1, V\}
\end{equation} 
should represent  
the trajectory of $x_{r_{pr}}(\tilde \mu)$ well. 

\subsubsection{Constructing $V_{r_{rpr}}$}
From the
state vector of the primal-residual-residual system~(\ref{eq:prim_resi_resi}), we see that
\begin{equation}
\label{eq:xrrpr}
\begin{array}{rl}
x_{r_{rpr}}(\tilde \mu)&=Q^{-1}(\tilde \mu)r_{r_{pr}}(\tilde \mu)\\
&=Q^{-1}(\tilde \mu)(r_{pr}(\tilde \mu)-Q(\tilde \mu)V_{r_{pr}}z_{r_{pr}}(\tilde \mu))\\
&=Q^{-1}(\tilde \mu)(r_{pr}(\tilde \mu)-V_{r_{pr}}z_{r_{pr}}(\tilde \mu))\\
&=Q^{-1}(\tilde \mu)(B(\mu)-QVz_{pr}(\tilde \mu))-V_{r_{pr}}z_{r_{pr}}(\tilde \mu)\\
&=Q^{-1}(\tilde \mu)B(\mu)-Vz_{pr}(\tilde \mu)-V_{r_{pr}}z_{r_{pr}}(\tilde \mu).
\end{array}
\end{equation}
Taking the same expansion point as in~(\ref{eq:x_expan}), the series expansion of $Q^{-1}(\tilde \mu)B(\mu)$ in the last equation of~(\ref{eq:xrrpr}) gives rise to the projection matrix $V$. Consequently, the subspace for $x_{r_{rpr}}(\tilde \mu)$ is $\textrm{range}(V,V_{r_{pr}})$, which is equivalent with $\textrm{range}(V_{r_{pr}})$, since $V$ is already included in $V_{r_{pr}}$ in ~(\ref{eq:Vrpr}). This is in contradiction with Proposition~\ref{prop:VrrprV} that $V_{r_{rpr}}$ should be different from $V_{r_{pr}}$. Therefore, $Q^{-1}(\tilde \mu)B(\mu)$ in the last equation of~(\ref{eq:xrrpr}) cannot be expanded using the same expansion points as those for both $V$ and $V_{r_{pr}}$. Recall that $V_{r_{rpr}}$ is used to construct the ROM of the primal-residual-residual system~(\ref{eq:prim_resi_resi}) and contributes to the error estimator $\Delta_3^{pr}(\tilde \mu)$. Then the expansion points for series expansion of $Q^{-1}(\tilde \mu)B(\mu)$ in the last equation of~(\ref{eq:xrrpr}) can be iteratively chosen by searching the maximum of $|C( \mu)\hat x_{r_{rpr}}(\tilde \mu)|$, the second part of $\Delta_3^{pr}(\tilde \mu)$, which purely depends on the ROM built by $V_{r_{rpr}}$. Greedy algorithms computing the projection matrices are presented in Section~\ref{sec:alg}.

\section{Comparing the proposed error estimators with the existing ones}
\label{sec:comp}

\subsection{Review of the error estimator in~\cite{morSemZP18}}

State error estimation as well as output error estimation for parametric linear steady systems is proposed in~\cite{morSemZP18} based on randomized residuals. Given the system has only a single input, the output error estimation can be used to estimate the transfer function error in frequency domain. The transfer function error $e_H(\tilde \mu):=H(\tilde \mu)-\hat H(\tilde \mu)$ can be measured using the 2-norm $\|e_H(\tilde \mu)\|_2$. The error estimator is given as
\begin{equation}
\label{eq:est_rand0}
\|e_H(\tilde \mu)\|_2 \approx \frac{1}{K}\left(\sum\limits_{i=1}^{K} \delta_i^2 \right)^{1/2}=:\Delta_r(\tilde \mu), 
\end{equation}
where $\delta_i= ( x_{du}^i (\tilde \mu))^T r_{pr}(\tilde \mu)$, and $ x_{du}^i(\tilde \mu)$ solves the $i$th random dual system,
\begin{equation}
\label{eq:dual_rand}
Q(\tilde \mu)^T x_{du}^i(\tilde \mu)=z_i,  i=1,\ldots,K,
\end{equation}
where $z_i\sim~\mathcal N(0, C^T(\mu)C(\mu))$ is a random vector following the normal distribution with zero mean and covariance matrix $C^T(\mu)C(\mu) \in \mathbb R^{n\times n}$. 
According to Remark 2.6 in~\cite{morSemZP18}, the random dual systems reduce to 
\begin{equation}
\label{eq:dual_rand_SO}
Q(\tilde \mu)^T x_{du}^i(\tilde \mu)=\xi_i C^T(\mu),  i=1,\ldots,K,
\end{equation}
where $\xi_i \sim~\mathcal N(0, 1)$ is a random variable (scalar) with standard normal random distribution. Therefore, $x_{du}^i(\tilde \mu)$ can be obtained by first solving the dual system in~(\ref{eq:dual}) to get $x_{du}(\tilde \mu)$ and then multiplying $x_{du}(\tilde \mu)$ with $\xi_i$, i.e. $x_{du}^i(\tilde \mu)=\xi_i x_{du}(\tilde \mu)$. 

It is stated in~\cite{morSemZP18} (Corollary 2.5) that under certain conditions, $\Delta_r$ is an error estimator of the true error with the probability 
\begin{equation}
\label{eq:est_rand}
\mathbb P\{w^{-1} \Delta_r(\tilde \mu) \leq \|e_H(\tilde \mu)\|_2 \leq w \Delta_r(\tilde \mu), \forall \tilde \mu \in \Xi \}\geq 1-\delta,
\end{equation}
where $w>\sqrt{e}$, $e$ is the Euler number, and $\Xi$ is a finite set of parameter samples, $0<\delta<1$.
Note that the dual system~(\ref{eq:dual}) with large size $n$ needs to be solved at least once for every parameter to obtain $x_{du}^i$, this is still costly. Therefore,
for single output systems, $x_{du}$ is replaced by $\hat x_{du}$, so that only the reduced dual system in~(\ref{eq:dual_redu}) needs to be solved. For multiple output systems, each of the random dual systems in~(\ref{eq:dual_rand}) is first reduced to a small system and then 
$x_{du}^i$ is approximated by the approximate solutions $\hat x_{du}^i$ computed from the reduced random dual systems. Finally, 
we have
\begin{equation}
\label{eq:estr}
\begin{array}{rcl}
\|e(\tilde \mu)\|_2 &\approx& \frac{1}{K}\left(\sum\limits_{i=1}^{K} \delta_i^2 \right)^{1/2}=:\Delta_r(\tilde \mu)\\
&\approx& \frac{1}{K}\left(\sum\limits_{i=1}^{K} \tilde \delta_i^2 \right)^{1/2}=: \tilde \Delta_r(\tilde \mu),
 \end{array}
\end{equation}
where $\tilde \delta_i= (\hat x_{du}^i (\tilde \mu))^T r_{pr}(\tilde \mu)$.

\subsection{Robustness comparison}

\begin{itemize}

\item $\Delta_1(\tilde \mu)$ vs. $\Delta_1^{pr}(\tilde \mu)$: To compute $\Delta_1(\tilde \mu)$, we need reduce both a primal system and a dual system. Whereas, the primal system and the primal-residual system are reduced to obtain $\Delta_1^{pr}(\tilde \mu)$. Although it is not clear which one better estimates the true error theoretically, numerical results nevertheless show obvious superiority of 
$\Delta_1^{pr}(\tilde \mu)$ over $\Delta_1(\tilde \mu)$.

\item $\Delta_1(\tilde \mu)$ vs. $\Delta_2(\tilde \mu)$: it is clear that $\Delta_2(\tilde \mu)$ is an upper bound of $\Delta_1(\tilde \mu)$, though it is not an upper bound of the true error. This means, $\Delta_1(\tilde \mu)$ is more likely to underestimate the true error than $\Delta_2(\tilde \mu)$, if $W_{r_{du}} \neq W_{du}$ due to Proposition~\ref{prop:VdurVdu}.

\item $\Delta_1(\tilde \mu)$ vs. $\Delta_2^{pr}(\tilde \mu)$: analogously, $\Delta_1(\tilde \mu)$ is more likely to underestimate the true error than $\Delta_2^{pr}(\tilde \mu)$, if $W_{r_{pr}} \neq W$ due to Proposition~\ref{prop:VrprV}.

\item $\Delta_1^{pr}(\tilde \mu)$ vs. $\Delta_3(\tilde \mu)$: $\Delta_1^{pr}(\tilde \mu)$ is more likely to underestimate the true error than $\Delta_3(\tilde \mu)$.

\item $\Delta_1^{pr}(\tilde \mu)$ vs. $\Delta_3^{pr}(\tilde \mu)$: $\Delta_1^{pr}(\tilde \mu)$ is more likely to underestimate the true error than $\Delta_3^{pr}(\tilde \mu)$, if $W_{r_{rpr}} \neq W_{r_{pr}}$ due to Proposition~\ref{prop:VrrprV}.

\item $\Delta_2(\tilde \mu)$ vs. $\Delta_2^{pr}(\tilde \mu)$: the only difference between $\Delta_2$ and $\Delta_2^{pr}(\tilde \mu)$ is the difference between their second parts, where the ROM of the dual residual system ($\hat x_{du}(\tilde \mu)$) is used for $\Delta_2(\tilde \mu)$, whereas the ROM of the primal-residual system ($\hat x_{r_{pr}}(\tilde \mu)$) is used for $\Delta_2^{pr}(\tilde \mu)$. They also behave similarly in the numerical experiments.

\item $\Delta_2(\tilde \mu)$ vs. $\Delta_3(\tilde \mu)$: the first  term $|\hat x^T_{r_{du}}(\tilde \mu)r_{pr}(\tilde \mu)|$ of $\Delta_2(\tilde \mu)$ results from the ROM of the primal system and that of the dual system.  The first term $|C(\mu)\hat x_{r_{pr}}(\tilde \mu)|$ of $\Delta_3(\tilde \mu)$ results from reducing the primal system and the primal-residual system. 
As for their second terms: $|\hat x_{du}^T(\tilde \mu)r_{pr}(\tilde \mu)|$ of $\Delta_2(\tilde \mu)$ and $|\hat x_{du}^T(\tilde \mu)r_{r_{pr}}(\tilde \mu)|$ of $\Delta_3(\tilde \mu)$,  $r_{pr}(\tilde \mu)$
is the residual from the ROM of the primal system, but $r_{r_{pr}}(\tilde \mu)$ is the residual 
from the ROM of the primal-residual system. $r_{r_{pr}}(\tilde \mu)$ is the result of two-step model reduction, whereas $r_{pr}$ results from one step of MOR. Numerical results show that $\Delta_2(\tilde \mu)$ is  more robust than $\Delta_3(\tilde \mu)$, when $\Delta_2(\tilde \mu)$ is computed properly, especially for near symmetric systems.

\item $\Delta_3(\tilde \mu)$ vs. $\Delta_3^{pr}(\tilde \mu)$: The only difference between $\Delta_3(\tilde \mu)$ and $\Delta_3^{pr}(\tilde \mu)$ is the difference between their second parts, where $\hat x_{du}(\tilde \mu)$, the quantity computed from the ROM of the dual system is used for $\Delta_3(\tilde \mu)$, whereas, $\hat x_{r_{rpr}}(\tilde \mu)$, the quantity computed from the ROM of the primal-residual-residual system is used for $\Delta_3^{pr}(\tilde \mu)$. Numerical results in the next section show little difference between their effectivities.

\item $\Delta_0(\tilde \mu)$ vs. $\Delta_2(\tilde \mu)$ in~\cite{morFenAB17}: It is shown in~\cite{morFenB19b} that $\Delta_0(\tilde \mu)$ has motivated the derivation of $\Delta_2(\tilde \mu)$ and can be seen as an upper bound of $\Delta_2(\tilde \mu)$. Although $\Delta_0(\tilde \mu)$ is an error bound of the 
transfer function error, it is much more time consuming to compute as compared with $\Delta_2(\tilde \mu)$, since the smallest singular value of a large matrix (of the original model size $n$) needs to be solved for every parameter value in a given training set. $\Delta_2(\tilde \mu)$ avoids this computational issue. Numerical tests on several models in~\cite{morFenB19b} have shown that $\Delta_2(\tilde \mu)$ is much tighter than $\Delta_0(\tilde \mu)$ and behaves as an error bound,  except for very small true errors close to machine precision.

\item $\tilde \Delta_r(\tilde \mu)$ in~\cite{morSemZP18} vs. $\Delta_1(\tilde \mu)$: 
From the proof of Theorem 1, we see that the quantity $|x_{du}^T(\tilde \mu)r_{pr}(\tilde \mu)|$ in~(\ref{error}) is exactly the true error. Using a similar description as in~(\ref{eq:est_rand}),  $|x_{du}^T(\tilde \mu)r_{pr}|$ satisfies
\begin{equation}
\label{eq:est_1p}
\mathbb P\{w^{-1} |x_{du}^T(\tilde \mu)r_{pr}| \leq \|e_H(\tilde \mu)\|_2 \leq w |x_{du}^T(\tilde \mu)r_{pr}(\tilde \mu)|, \forall \tilde \mu \in \Xi, \forall \Xi \in \mathcal D \}=1,
\end{equation}
with $w=1$, which is an exact estimation of the true error not only for any $\tilde \mu$ in a given $\Xi$ as in~(\ref{eq:est_rand}), but also for any $\tilde \mu$ in $\mathcal D$ . Here, $\mathcal D$ is the continuous parameter domain. 
Comparing~(\ref{eq:est_1p}) with~(\ref{eq:est_rand}), we know that $\Delta_r(\tilde \mu)$ in~(\ref{eq:est_rand0}) and (\ref{eq:est_rand}) is an error estimator, whereas $|x_{du}^T(\tilde \mu)r_{pr}(\tilde \mu)|$ in~(\ref{error}) and (\ref{eq:est_1p}) is the true error. Furthermore,   
the error estimator $\Delta_1(\tilde \mu)$ is derived based on $|x_{du}^T(\tilde \mu)r_{pr}(\tilde \mu)|$ by replacing the true dual solution $x_{du}(\tilde \mu)$ in $|x_{du}^T(\tilde \mu)r_{pr}(\tilde \mu)|$ with the approximate dual solution $\hat x_{du}(\tilde \mu)$; whereas $\tilde \Delta_r(\tilde \mu)$ is derived based on $\Delta_r(\tilde \mu)$ in~(\ref{eq:est_rand0}) also by replacing $x_{du}(\tilde \mu)$ in $\Delta_r(\tilde \mu)$ with $\hat x_{du}(\tilde \mu)$. 
In summary, $\Delta_1(\tilde \mu)$ is only a one-step approximation of the true error, whereas, $\tilde \Delta_r(\tilde \mu)$ is a 
two-step approximation of the true error. It is therefore not difficult to see that $\Delta_1(\tilde \mu)$ should be tighter than $\tilde \Delta_r(\tilde \mu)$. Simulation results also show that 
$\tilde \Delta_r(\tilde \mu)$ is often not as tight as $\Delta_1(\tilde \mu)$. From the previous analyses, $\Delta_1(\tilde \mu)$
is less accurate than all the other proposed error estimators, which can also be seen from the numerical results in Section~\ref{sec:numerics}. Therefore, it appears to be unnecessary to compare $\tilde \Delta_r(\tilde \mu)$ with the other estimators.

\end{itemize}

\subsection{Computational complexity comparison}

Computing any of the error estimators discussed in this work needs to construct a ROM of the primal system. 
It is noticed that the projection matrix $V$ used to construct the ROM of the primal system~(\ref{eq:primal}) is the same matrix used to derive the ROM of the original system. Therefore, the ROM of the primal system can be derived {\it for free} in the sense that $V$ is obtained without additional computation. Except for constructing the ROM of the primal system, we list the following additional costs required by different error estimators.

\begin{itemize}
\item Computing $\Delta_0(\tilde \mu)$ involves constructing the ROM of the dual system~(\ref{eq:dual}), and computing the inf-sup constant at each $\tilde \mu$ in the training set $\Xi$.

\item Computing $\Delta_1(\tilde \mu)$ or $\Delta_1^{pr}(\tilde \mu)$ involves constructing the ROM of the dual system or the ROM of the primal-residual system~(\ref{eq:prim_resisys}). 

\item Computing $\Delta_2(\tilde \mu)$, $\Delta_2^{pr}(\tilde \mu)$ or $\Delta_3(\tilde \mu)$ involves constructing the ROM of the dual system~(\ref{eq:dual}), and additionally the ROM of a corresponding residual system needs to be constructed: the ROM of the dual-residual system~(\ref{eq:dual_resisys}) or the ROM of the primal-residual system~(\ref{eq:prim_resisys}). 

\item Computing $\Delta_3^{pr}(\tilde \mu)$ involves constructing the ROMs of two residual systems: the primal-residual system~(\ref{eq:prim_resisys}) and the primal-residual-residual system~(\ref{eq:prim_resi_resi}).

\item Computing $\Delta_r(\tilde \mu)$ involves constructing the ROM of the dual system if the output matrix $C$ is a vector, otherwise, $K$ ROMs of the $K$ random dual systems in~(\ref{eq:dual_rand}) must be constructed.

\end{itemize}

From Subsection~\ref{subsec:construct_ROMs}, we see that to construct the ROMs of the dual system, or any of the residual systems, one only has to solve several linear systems to compute the coefficients in the series expansion of the corresponding solution vector. For interpolatory MOR methods in frequency domain, the cost of constructing the ROM of any of the above mentioned system is equivalent to the cost of constructing the ROM of the original system. This means, in order to compute any of the error estimators, one or two additional ROMs need to be constructed
at each iteration of the greedy algorithm. However, the error bound $\Delta_0(\tilde \mu)$ has the highest computational cost, since computing the inf-sup constant means solving a large eigenvalue problem at each 
$\tilde \mu$ in $\Xi$ per iteration. Furthermore, from the proposed greedy algorithms in the next section, the additional ROMs are constructed simultaneously with the ROM~(\ref{eq:ROM}) of the original system, no separate greedy algorithms are required as in~\cite{morSchWH18}.

\section{Greedy algorithms for constructing the projection matrices}
\label{sec:alg}

The aim of an efficient error estimator is to construct a ROM of the original system with satisfying accuracy and high reliability. In the following, we show algorithms for constructing the ROM of the original system, where an error estimator acts as a guidance for greedy constructing the projection matrix $V$ for the ROM. Again, we use Galerkin Projection to compute the ROM of the original systems and the ROMs of the other systems which are involved in computing the error estimators. To compute any of the proposed error estimators, corresponding projection matrices $V_{du}$, $V_{r_{du}}$, $V_{r_{pr}}$, $V_{r_{rpr}}$ need to be constructed simultaneously with $V$. 

As compared with the algorithms in~\cite{morFenB19b}, we have included the proposed variants of the error estimator and computation of their corresponding projection matrices into the algorithms.
The performance of the proposed error estimators as well as the existing ones are compared in the next section.

We first present the greedy scheme for non-parametric systems in Algorithm~\ref{alg:greedy_nonpara}. The standard moment-matching method~\cite{morBauBF14} is used to compute the projection matrices. 
$\epsilon_{tol}$ is the tolerance for the error of the reduced transfer function. Once the maximal error estimator over the whole sample set $\Xi$ is below the tolerance, the greedy algorithm stops. In every iteration, the $s$ sample corresponding to the maximal error estimator is chosen as the next expansion point $s_i$ (Step 22). Steps 5, 8, 12, 16 and Step 20 orthogonalize the vectors in $V(s_i)$ and $V_{du}(s_i)$, $V_{r_{du}}(s_i^\alpha)$, $V_{r_{pr}}(s_i^\alpha)$, $V_{r_{rpr}}(s_i^\beta)$ against the existing vectors in
$V$ and $V_{du}$, $V_{r_{du}}$, $V_{r_{pr}}$, $V_{r_{rpr}}$, respectively.
In Algorithm~\ref{alg:greedy_nonpara}, some steps are only implemented for certain error estimators, depending on which error estimator is being used. $s_i^\alpha$ is chosen to iteratively construct $V_{r_{du}}$ or $V_{r_{pr}}$, while $s_i^{\beta}$ is chosen to construct $V_{r_{rpr}}$. The choice of the expansion points $s_i^\alpha$ or $s_i^\beta$ depends on the part of the error estimator which is solely decided by the corresponding projection matrices $V_{r_{du}}$, $V_{r_{pr}}$, or $V_{r_{rpr}}$. As for $\Delta_1^{pr}(\tilde \mu)$, since $s_i$ is chosen according to $\Delta_1^{pr}$, $s_i^\alpha$ is chosen according to the norm of $r_{r_{pr}}$ to avoid  
$V_{r_{pr}}$ being identical with $V$ due to Proposition~\ref{prop:VrprV}.
\begin{algorithm}[h]
\caption[]{Greedy ROM construction for non-parametric systems~(\ref{eq:FOM})}.
\label{alg:greedy_nonpara}
\begin{algorithmic}[1] 
\REQUIRE System matrices $E,A,B,C, \epsilon_{tol}$, $\Xi$: a set of samples of $s$ covering the interesting frequency range.
\ENSURE The projection matrix $V$ for constructing the ROM in~(\ref{eq:ROM}).
\STATE $V=[]$, $V_{du}=[]$,$V_{r_{du}}=[]$, $V_{r_{pr}}=[]$, $V_{r_{rpr}}=[]$, set $\epsilon=\epsilon_{tol}+1, q>1$.
\STATE Initial expansion point: $s_i \in \Xi$, for $V, V_{du}$; \ $s_i^\alpha \in \Xi$, for $V_{r_{du}}$(or $V_{r_{pr}}$); \   $s_i^\beta \in \Xi$, for $V_{r_{rpr}}$,  $i=1$. 
\WHILE{$\epsilon>\epsilon_{tol}$}

              \STATE $\mathrm{range}(V(s_i))=\mathrm{span}\{ \tilde B(s_i), \ldots, (\tilde A(s_i))^{q-1}\tilde B(s_i) \}$, where $\tilde A(s)=(sE-A)^{-1}E$, $\tilde{B}(s)=(sE-A)^{-1}B$, and $q \ll n$
\STATE $V=\mathrm{orth}\{V, V(s_i)\}$
\IF{$\Delta(s)=\Delta_1(s)$, or $\Delta_2(s)$, or $\Delta_2^{pr}(s)$, or $\Delta_3(s)$}
           \STATE $\mathrm{range}(V_{du}(s_i))=\mathrm{span}\{\tilde C(s_i), \ldots, (\tilde A_c(s_i))^{q-1}\tilde C(s_i)\}$, where $\tilde{A}_c(s)=(sE-A)^{-T}E^T$, $\tilde{C}(s)=(sE-A)^{-T}C^T$.
\STATE $V_{du}=\mathrm{orth}\{V_{du}, V_{du}(s_i)\}$. 
\ENDIF
\IF { $\Delta(s)=\Delta_2(s)$}
 \STATE $\mathrm{range}(V_{r_{du}}(s_i^\alpha))=\mathrm{span}\{\tilde C(s_i^\alpha), \ldots, (\tilde A_c (s_i^\alpha)^{q-1}\tilde C(s_i^\alpha)\}$. 
\STATE $V_{r_{du}}=\mathrm{orth}\{V_{du}, V_{r_{du}}, V_{r_{du}}(s_i^\alpha)\}$.
\ENDIF 
\IF {$\Delta(s)=\Delta_1^{pr}(s)$, or $\Delta_2^{pr}(s)$, or $\Delta_3(s)$, or $\Delta_3^{pr}(s)$}
\STATE  $\mathrm{range}(V_{r_{pr}}(s_i^\alpha))=\mathrm{span}\{ \tilde B(s_i^\alpha), \ldots, (\tilde A(s_i^\alpha))^{q-1}\tilde B(s_i^\alpha) \}$. 
\STATE $V_{r_{pr}}=\mathrm{orth}\{V, V_{r_{pr}}, V_{r_{pr}}(s_i^\alpha)\}$. 
\ENDIF
\IF {  $\Delta(s)=\Delta_3^{pr}(s)$}
\STATE  $\mathrm{range}(V_{r_{rpr}}(s_i^\beta))=\mathrm{span}\{ \tilde B(s_i^\beta), \ldots, (\tilde A(s_i^\beta))^{q-1}\tilde B(s_i^\beta) \}$.
\STATE $V_{r_{rpr}}=\mathrm{orth}\{V, V_{r_{pr}}, V_{r_{rpr}}, V_{r_{rpr}}(s_i^\beta)\}$.  
\ENDIF
 \STATE  $i=i+1$, $s_i=\textrm{arg} \max\limits_{s \in \Xi} \Delta(s)$.
\IF {$\Delta(s)=\Delta_2(s)$}
\STATE $s^\alpha_i=\textrm{arg} \max\limits_{s \in \Xi} |\hat x^T_{r_{du}}(s)r_{pr}(s)|$.\hfill \%second part of $ \Delta_2(s)$
\ENDIF
\IF {$\Delta(s)=\Delta_2^{pr}(s)$}
\STATE $s^\alpha_i=\textrm{arg} \max\limits_{s \in \Xi} |r_{du}^T(s) \hat x_{r_{pr}}(s)|$. \hfill \%second part of $ \Delta_2^{pr}(s)$
\ENDIF
\IF {$\Delta(s)=\Delta_1^{pr}(s)$}
\STATE $s^\alpha_i=\textrm{arg} \max\limits_{s \in \Xi} \|r_{r_{pr}}(s)\|_2$. \ $r_{r_{pr}(s)}$ is defined in~(\ref{error_var_distance}). 
\ENDIF
\IF {$\Delta(s)=\Delta_3(s)$, or $\Delta_3^{pr}(s)$}
\STATE $s^\alpha_i=\textrm{arg} \max\limits_{s \in \Xi} \Delta_1^{pr}(s)$. \hfill \%first part of $\Delta_3(s)$ or $\Delta_3^{pr}(s)$
\ENDIF
\IF {$\Delta(s)=\Delta_3^{pr}(s)$}
\STATE $s^\beta_i=\textrm{arg} \max\limits_{s \in \Xi} |C\hat x_{r_{rpr}}(s)|$. \quad 
\hfill \%second part of $\Delta_3^{pr}(s)$ 
\ENDIF
\STATE ${\epsilon}=\Delta(s_i)$.
\ENDWHILE
\end{algorithmic}
\end{algorithm}

Algorithm~\ref{alg:greedy_para} shows the adaptive scheme for linear parametric systems.
Algorithm~\ref{alg:greedy_para} is similar with Algorithm~\ref{alg:greedy_nonpara}. Its only difference from Algorithm~\ref{alg:greedy_nonpara} is in computing the projection matrices at a chosen expansion point in Steps 4, 7, 11, 15 and Step 19, where the multi-moment-matching method instead of the moment-matching method is used. 
\begin{algorithm}[h]
\caption[]{Greedy ROM construction for parametric systems~(\ref{eq:FOM})}.
\label{alg:greedy_para}
\begin{algorithmic}[1] 
\REQUIRE System matrices $E(\mu),A(\mu),B(\mu),C(\mu), \epsilon_{tol}$, $\Xi$: a set of samples of $\tilde \mu$ covering the interesting parameter domain.
\ENSURE The projection matrix $V$ for constructing the ROM in~(\ref{eq:ROM}).
\STATE $V=[]$, $V_{du}=[]$,$V_{r_{du}}=[]$, $V_{r_{pr}}=[]$, $V_{r_{rpr}}=[]$, set $\epsilon=\epsilon_{tol}+1$.
\STATE Initial expansion point: $\tilde \mu^i \in \Xi$ for $V, V_{du}$; \ $\tilde \mu^i_\alpha$ for $V_{r_{du}}$(or $V_{r_{pr}}$); \ $\tilde \mu^i_\beta$ for $V_{r_{rpr}}$,  $i=1$. 
\WHILE{$\epsilon>\epsilon_{tol}$}
\STATE compute $V_{\tilde \mu^i}$ following $(\ref{eq:Vi})$.
\STATE $V=\mathrm{orth}\{V, V_{\tilde \mu^i)}\}$.
\IF{$\Delta(\tilde \mu)=\Delta_1(\tilde \mu)$, or $\Delta_2(\tilde \mu)$, or $\Delta_2^{pr}(\tilde \mu)$, or $\Delta_3(\tilde \mu)$}
\STATE compute $V^{du}_{\tilde \mu^i}$ following $(\ref{eq:Vidu})$.
\STATE $V_{du}=\mathrm{orth}\{V_{du}, V^{du}_{\tilde \mu^i}\}$.
\ENDIF
\IF{$\Delta(\tilde \mu)=\Delta_2(\tilde \mu)$}
\STATE  compute $V^{r_{du}}_{\tilde \mu_\alpha^i}$ following $(\ref{eq:Vidur1})$.
\STATE $V_{r_{du}}=\mathrm{orth}\{V_{du}, V_{r_{du}}, V^{r_{du}}_{\tilde \mu_\alpha^i}\}$.
\ENDIF
\IF{$\Delta(\tilde \mu)=\Delta_1^{pr}(\tilde \mu)$, or $\Delta_2^{pr}(\tilde \mu)$, or $\Delta_3(\tilde \mu)$, or $\Delta_3^{pr}(\tilde \mu)$}
\STATE compute $V^{r_{pr}}_{\tilde \mu^i_\alpha}$ following $(\ref{eq:Vi})$.
\STATE $V_{r_{pr}}=\mathrm{orth}\{V, V_{r_{pr}}, V^{r_{pr}}_{\tilde \mu_\alpha^i}\}$.
\ENDIF
\IF{$\Delta(\tilde \mu)=\Delta_3^{pr}(\tilde \mu)$}
\STATE compute $V^{r_{rpr}}_{\tilde \mu^i_\beta}$ following $(\ref{eq:Vi})$.
\STATE $V_{r_{rpr}}=\mathrm{orth}\{V, V_{r_{pr}}, V_{r_{rpr}}, V^{r_{rpr}}_{\tilde \mu_\beta^i}\}$.
\ENDIF
\STATE  $i=i+1$, $\tilde \mu^i=\textrm{arg} \max\limits_{\tilde \mu \in \Xi} \Delta(\tilde \mu)$.
\IF {$\Delta(\tilde \mu)=\Delta_2(\tilde \mu)$}
\STATE $\tilde \mu_\alpha^i=\textrm{arg} \max\limits_{\tilde \mu \in \Xi} |\hat x^T_{r_{du}}(\tilde \mu)r_{pr}(\tilde \mu)|$. \hfill \%second part of $ \Delta_2(\tilde \mu)$
\ENDIF
\IF {$\Delta(s)=\Delta_2^{pr}(\tilde \mu)$}
\STATE $\tilde \mu_\alpha^i=\textrm{arg} \max\limits_{s \in \Xi} |r_{du}^T(\tilde \mu) \hat x_{r_{pr}}(\tilde \mu)|$. \hfill \%second part of $ \Delta_2^{pr}(\tilde \mu)$
\ENDIF
\IF {$\Delta(\tilde \mu)=\Delta_1^{pr}(\tilde \mu)$}
\STATE $\tilde \mu_\alpha^i=\textrm{arg} \max\limits_{\tilde \mu  \in \Xi} \|r_{r_{pr}}(\tilde \mu)\|_2$.      \hfill \% $r_{r_{pr}}(\tilde \mu)$ is defined in~(\ref{error_var_distance}). 
\ENDIF
\IF {$\Delta(\tilde \mu)=\Delta_3(\tilde \mu)$, or $\Delta_3^{pr}(\tilde \mu)$}
\STATE $\tilde \mu_\alpha^i=\textrm{arg} \max\limits_{\tilde \mu \in \Xi} \Delta_1^{pr}(\tilde \mu)$.
\ENDIF
\IF {$\Delta(\tilde \mu)=\Delta_3^{pr}(\tilde \mu)$}
\STATE $\tilde \mu_\beta^i=\textrm{arg} \max\limits_{\tilde \mu \in \Xi} |C(\mu)\hat x_{r_{rpr}}(\tilde \mu)|$.      \hfill \%second part of $\Delta_3^{pr}(\tilde \mu)$
\ENDIF
\STATE ${\epsilon}=\Delta(\tilde \mu^i)$.
\ENDWHILE
\end{algorithmic}
\end{algorithm}

We point out in Remark~\ref{rem:VduV}, Section~\ref{sec:relation} that when a system is almost symmetric, $\Delta_1(\tilde \mu)$ performs badly, which will in turn, affect the behavior of 
$\Delta_2(\tilde \mu)$ and $\Delta_2^{pr}(\tilde \mu)$. From the simulation results in the next section, we will see that, except for the CD player model, $\Delta_1(\tilde \mu)$ is not a good estimator. It is observed that for the RLCtree model, $Q(s)$ is symmetric, and only two elements are different between the input vector $B(\mu)$ and the transpose of the output vector $C(\mu)$. For the MIMO example, the matrix $E$ is symmetric and $B(\mu)=C^T(\mu)$. For the parametric example, the mass matrix is symmetric. The stiffness matrix is unsymmetric, but the maximal magnitude of the elements in the matrix  $T^T(\mu)-T(\mu)$ is around $O(10^{-18})$ for all the parameters. This implicates that $T$ should be symmetric in theory, and the small differences between $T(\mu)$ and its transpose might be caused by numerical errors.  The maximal magnitude of the elements in the damping matrix is also small, $O(10^{-11})$. All the three examples are close to the symmetric case indicated in Remark~\ref{rem:VduV}. 

In the following, we propose two algorithms: Algorithms~\ref{alg:greedy_nonpara_im}-\ref{alg:greedy_para_im}, aiming at improving the behavior of $\Delta_1(\tilde \mu)$, $\Delta_2(\tilde \mu)$ and $\Delta_2^{pr}(\tilde \mu)$ for nearly symmetric systems. Their main difference from Algorithm~\ref{alg:greedy_nonpara} and ~\ref{alg:greedy_para} is that instead of using the same expansion point for $V_{du}$ and $V$, different expansion points ($s_i^\gamma$ or $\tilde \mu_\gamma^i$) are chosen for $V_{du}$ according to a different error criterion which directly depends on $V_{du}$, see Steps 23-28 in Algorithm~\ref{alg:greedy_nonpara_im} and Algorithm~\ref{alg:greedy_para_im}, respectively.
\begin{algorithm}[h]
\caption[]{Improving  $\Delta_1(\tilde \mu)$, $\Delta_2(\tilde \mu)$ and $\Delta_2^{pr}(\tilde \mu)$ for nearly symmetric and non-parametric systems~(\ref{eq:FOM})}.
\label{alg:greedy_nonpara_im}
\begin{algorithmic}[1] 
\REQUIRE System matrices $E,A,B,C, \epsilon_{tol}$, $\Xi$: a set of samples of $s$ covering the interesting frequency range.
\ENSURE The projection matrix $V$ for constructing the ROM in~(\ref{eq:ROM}).
\STATE $V=[]$, $V_{du}=[]$, $V_{r_{du}}=[]$, $V_{r_{pr}}=[]$ set $\epsilon=\epsilon_{tol}+1, q>1$.
\STATE Initial expansion point: $i=1$, $s_i \in \Xi$ for $V$; \ $s_i^\alpha \in \Xi$ for $V_{r_{du}}$(or $V_{r_{pr}}$); \ 
 $s_i^\gamma \in \Xi$ for $V_{du}$.
\WHILE{$\epsilon>\epsilon_{tol}$}
 \STATE $\mathrm{range}(V(s_i))=\mathrm{span}\{ \tilde B(s_i), \ldots, (\tilde A(s_i))^{q-1}\tilde B(s_i) \}$.
\STATE $V=\mathrm{orth}\{V, V(s_i)\}$
\STATE $\mathrm{range}(V_{du}(s_i^\gamma))=\mathrm{span}\{\tilde C(s_i^\gamma), \ldots, (\tilde A_c(s_i^\gamma))^{q-1}\tilde C(s_i^\gamma)\}$.
\STATE $V_{du}=\mathrm{orth}\{V_{du}, V_{du}(s_i^\gamma)\}$. 
\IF { $\Delta(s)=\Delta_2(s)$}
 \STATE $\mathrm{range}(V_{r_{du}}(s_i^\alpha))=\mathrm{span}\{\tilde C(s_i^\alpha), \ldots, (\tilde A_c (s_i^\alpha)^{q-1}\tilde C(s_i^\alpha)\}$. 
\STATE $V_{r_{du}}=\mathrm{orth}\{V_{du}, V_{r_{du}}, V_{r_{du}}(s_i^\alpha)\}$.
\ENDIF 
\IF {$\Delta(s)=\Delta_2^{pr}(s)$}
\STATE  $\mathrm{range}(V_{r_{pr}}(s_i^\alpha))=\mathrm{span}\{ \tilde B(s_i^\alpha), \ldots, (\tilde A(s_i^\alpha))^{q-1}\tilde B(s_i^\alpha) \}$. 
\STATE $V_{r_{pr}}=\mathrm{orth}\{V, V_{r_{pr}}, V_{r_{pr}}(s_i^\alpha)\}$. 
\ENDIF
 \STATE  $i=i+1$, $s_i=\textrm{arg} \max\limits_{s \in \Xi} \Delta(s)$.
\IF {$\Delta(s)=\Delta_2(s)$}
\STATE $s^\alpha_i=\textrm{arg} \max\limits_{s \in \Xi} |\hat x^T_{r_{du}}(s)r_{pr}(s)|$.      \hfill \%second part of $\Delta_2(s)$
\ENDIF
\IF {$\Delta(s)=\Delta_2^{pr}(s)$}
\STATE $s^\alpha_i=\textrm{arg} \max\limits_{s \in \Xi} |r_{du}^T(s)\hat x_{r_{pr}}(s)|$.      \hfill \%second part of $\Delta_2^{pr}(s)$
\ENDIF
\IF {$\Delta(s)=\Delta_1(s)$}
\STATE $s^\gamma_i=\textrm{arg} \max\limits_{s \in \Xi} \|r_{du}(s)\|_2$.
\ENDIF
\IF {$\Delta(s)=\Delta_2(s)$ or $\Delta_2^{pr}$}
\STATE $s^\gamma_i=\textrm{arg} \max\limits_{s \in \Xi} \Delta_1(s)$.      \hfill \%first part of $\Delta_2(s)$ or $\Delta_2^{pr}$
\ENDIF
\STATE ${\epsilon}=\Delta(s_i)$.
\ENDWHILE
\end{algorithmic}
\end{algorithm}
\begin{algorithm}[h]
\caption[]{Improving  $\Delta_1(\tilde \mu)$, $\Delta_2(\tilde \mu)$ and $\Delta_2^{pr}(\tilde \mu)$ for nearly symmetric and parametric systems~(\ref{eq:FOM})}.
\label{alg:greedy_para_im}
\begin{algorithmic}[1] 
\REQUIRE System matrices $E(\mu),A(\mu),B(\mu),C(\mu), \epsilon_{tol}$, $\Xi$: a set of samples of $\tilde \mu$ covering the interesting frequency range.
\ENSURE The projection matrix $V$ for constructing the ROM in~(\ref{eq:ROM}).
\STATE $V=[]$, $V_{du}=[]$, $V_{r_{du}}=[]$,  $V_{r_{pr}}=[]$, set $\epsilon=\epsilon_{tol}+1$.
\STATE Initial expansion point: $\tilde \mu^i \in \Xi$ for $V$;  \ $\tilde \mu^i_\alpha \in \Xi$ for $V_{r_{du}}$(or $V_{r_{pr}}$); \ $\tilde \mu^i_\gamma \in \Xi$ for $V_{du}$; $i=1$. 
\WHILE{$\epsilon>\epsilon_{tol}$}
\STATE compute $V_{\tilde \mu^i}$ following $(\ref{eq:Vi})$.
\STATE $V=\mathrm{orth}\{V, V_{\tilde \mu^i)}\}$.
\STATE compute $V^{du}_{\tilde \mu^i_\gamma}$ following $(\ref{eq:Vidu})$.
\STATE $V_{du}=\mathrm{orth}\{V_{du}, V^{du}_{\tilde \mu^i_\gamma}\}$.
\IF{$\Delta(\tilde \mu)=\Delta_2(\tilde \mu)$}
\STATE  compute $V^{r_{du}}_{\tilde \mu_\alpha^i}$ following $(\ref{eq:Vidur1})$.
\STATE $V_{r_{du}}=\mathrm{orth}\{V_{du}, V_{r_{du}}, V^{r_{du}}_{\tilde \mu_\alpha^i}\}$.
\ENDIF
\IF{$\Delta(\tilde \mu)=\Delta_2^{pr}(\tilde \mu)$}
\STATE compute $V^{r_{pr}}_{\tilde \mu^i_\alpha}$ following $(\ref{eq:Vi})$.
\STATE $V_{r_{pr}}=\mathrm{orth}\{V, V_{r_{pr}}, V^{r_{pr}}_{\tilde \mu_\alpha^i}\}$.
\ENDIF
\STATE  $i=i+1$, $\tilde \mu^i=\textrm{arg} \max\limits_{\tilde \mu \in \Xi} \Delta(\tilde \mu)$.
\IF {$\Delta(\tilde \mu)=\Delta_2(\tilde \mu)$}
\STATE $\tilde \mu_\alpha^i=\textrm{arg} \max\limits_{\tilde \mu \in \Xi} |\hat x^T_{r_{du}}(\tilde \mu)r_{pr}(\tilde \mu)|$.      \hfill \%second part of $\Delta_2(\tilde \mu)$
\ENDIF
\IF {$\Delta(\tilde \mu)=\Delta_2^{pr}(\tilde \mu)$}
\STATE $\tilde \mu_\alpha^i=\textrm{arg} \max\limits_{\tilde \mu \in \Xi} |r^T_{du}(\tilde \mu)\hat x_{r_{pr}}(\tilde \mu)|$.      \hfill \%second part of $\Delta_2^{pr}(\tilde \mu)$
\ENDIF
\IF {$\Delta(\tilde \mu)=\Delta_1(\tilde \mu)$}
\STATE $\tilde \mu_\gamma^i=\textrm{arg} \max\limits_{\tilde \mu \in \Xi}\|r_{du}(\tilde \mu)\|$. 
\ENDIF
\IF {$\Delta(\tilde \mu)=\Delta_2(\tilde \mu)$ or $\Delta_2^{pr}(\tilde \mu)$}
\STATE $\tilde \mu_\gamma^i=\textrm{arg} \max\limits_{\tilde \mu \in \Xi}\Delta_1(\tilde \mu)$. \quad      \hfill \%first part of $\Delta_2(\tilde \mu)$ or $\Delta_2^{pr}(\tilde \mu)$
\ENDIF
\STATE ${\epsilon}=\Delta(\tilde \mu^i)$.
\ENDWHILE
\end{algorithmic}
\end{algorithm}

\clearpage

\section{Simulation results}
\label{sec:numerics}

In this section, we show the performance of the proposed error estimators and the existing ones. Detailed analyses for each of them are presented accordingly. Since the error bound $\Delta_0(\tilde \mu)$ in~\cite{morFenAB17}
has been compared in detail with the error estimator $\Delta_2(\tilde \mu)$ in a recent work~\cite{morFenB19b}, we do not repeat this comparison. Furthermore, since $\Delta_0(\tilde \mu)$ was shown to be less tighter than $\Delta_2(\tilde \mu)$, it will not be compared with other error estimators either, as it will be clear from the results below that 
$\Delta_0(\tilde \mu)$ may not outperform most of the error estimators. 

We use the same four models as in~\cite{morFenB19b} to show the robustness of the error estimators. The first two are non-parametric SISO systems. One is a well-known MOR benchmark example, the model of a CD player (with order $n=120$), the other is a model of an RLC tree circuit with order $n=6,134$. 
The third example is a circuit model with $n=980$. It has 4 inputs and 4 outputs, and no parameters. Both the CD player model and the third multi-input multi-output (MIMO) circuit model are from the SLICOT benchmark collection 
\footnote{URL: http://www.icm.tu-bs.de/NICONET/benchmodred.html}. The last one is the model of a butterfly-shaped micro-gyroscope, available from the MOR benchmark collection\footnote{URL: https://morwiki.mpi-magdeburg.mpg.de/morwiki}. It is a second-order parametric system with $n=17,931$. 

The interesting frequency of the CD player model is $[0, 1\text{~MHz}]$. The interesting frequency of the second and the third models is $[0, 3\text{~GHz}]$. 
The Gyroscope model is a low frequency problem with $f\in [50\text{~Hz}, 250\text{~Hz}]$.

The error tolerance $\epsilon_{tol}$ used in the greedy algorithms, i.e. the error tolerance for the error of the ROM of the original system, is set as $1\times 10^{-3}$ for the first three examples, while for the last example, we set $\epsilon_{tol}=1\times 10^{-7}$, since the transfer function $H(\mu)$ has the smallest magnitude of $2.8\times 10^{-7}$.

For all the non-parametric examples, we use $q=3$ (order of moments matched) in Algorithm~\ref{alg:greedy_nonpara} and Algorithm~\ref{alg:greedy_nonpara_im}. For the parametric model, we use $R_0, R_1$ to generate the matrices $V_{\tilde \mu^i}$, $V^{du}_{\tilde \mu^i}$, $V^{r_{du}}_{\tilde \mu_\alpha^i}$, $V^{r_{pr}}_{\tilde \mu_\alpha^i}$ and $V^{r_{rpr}}_{\tilde \mu_\beta^i}$ in Algorithm~\ref{alg:greedy_para} and Algorithm~\ref{alg:greedy_para_im}.
At each iteration, the maximal error estimator in $\Xi$, is computed, and is used as the error control for the ROM~(\ref{eq:ROM}) of the original system. Therefore, the maximal true error $\epsilon_{\max}=\max\limits_{\mu^i \in \Xi}\epsilon(\mu^i)$ is used for comparison, where $\epsilon(\mu^i)$ is the true error of the ROM evaluated at $\mu^i$, at the current iteration of the algorithm. 

For Algorithms~\ref{alg:greedy_nonpara}-\ref{alg:greedy_para}, the initial expansion point $s_1$ or $\tilde \mu^1$ for computing $V, V_{du}$ is taken as the first sample in $\Xi$, and 
the initial expansion point $s_1^\alpha$ or $\tilde \mu_\alpha^1$ for computing $V_{r_{du}}$, $V_{r_{pr}}$ is taken as the last sample in $\Xi$ to
make the two expansion points different from each other. The expansion point $s_1^\beta$ or $\tilde \mu_\beta^1$ is for $V_{r_{rpr}}$. It is taken as the midpoint in $\Xi$. Algorithms~\ref{alg:greedy_nonpara_im}-\ref{alg:greedy_para_im} are for (nearly) symmetric systems, and the initial expansion points $s_1, \tilde \mu^1 $ for $V$ are different from $s_1^\gamma, \tilde \mu^1_\gamma$ for $V_{du}$. Therefore, $s_1$ or $\tilde \mu^1$ is taken as the first sample in $\Xi$ and $s_1^\gamma$ or $\tilde \mu^1_\gamma$ is taken as the midpoint in $\Xi$. The initial point $s_1^\alpha$ or $\tilde \mu^1_\alpha$ for $V_{r_{du}}$, $V_{r_{pr}}$ is taken as the last point in $\Xi$.
\subsection{The CD player model}
The training set $\Xi$ for this model contains 60 samples of $s$, and then the finally obtained ROM in~(\ref{eq:ROM}) is validated at 600 samples of $s$ covering the whole interesting frequency range. The samples are taken from the interval $[0,1\text{~MHz}]$ using the MATLAB function "\textrm{logspace}". The results of Algorithm~\ref{alg:greedy_nonpara} using different error estimators are shown in Tables~\ref{error_CD player_est1}-\ref{error_CD player_est3}, where the error estimators and the corresponding true errors $\epsilon_{\textrm{max}}$ of the ROMs at each iteration of the Algorithm, are listed. Note that different ROMs are derived by using different error estimators, therefore the true  errors depend on the error estimators and are usually different. This also applies to analogous results listed in the other tables for other examples.

In Table~\ref{error_CD player_est1}, we also show the results for $\tilde \Delta_r(s)$ from~\cite{morSemZP18}, where $K$ in~(\ref{eq:estr}) is taken as $K=20$, which is shown to produce better results than $K=10$~\cite{morSemZP18}. During the greedy iteration, $\tilde \Delta_r(s)$ always underestimates the maximal true error. $\Delta_1^{pr}$ underestimates the true error 
at the first 5 iterations, but then becomes an accurate estimator at the last two iterations. $\Delta_1(s)$ is better than $\tilde \Delta_r(s)$, but is no better than the other estimators.
$\Delta_2(s)$ and its primal version $\Delta_2^{pr}(s)$ behave like error bounds. $\Delta_1^{pr}(s)$, $\Delta_3(s)$ and $\Delta_3^{pr}(s)$ have underestimation only at the first several iterations. In general, once  they bound error from above, they are very tight. 
\begin{table}[h]
\begin{center}
\caption{CD player, $\varepsilon_{tol}=10^{-3}$, $q=3$, $r=44(\Delta_r)$, $r=52 (\Delta_1)$, $r=56 (\Delta_1^{pr})$.}
\label{error_CD player_est1}
\begin{tabular}{|c||c|c||c|c||c|c||} \hline
iteration $i$ & $\varepsilon_{\max} (\tilde \Delta_r)$ & $\Delta_r(s_i)$ & $\varepsilon_{\max} (\Delta_1)$   &  $\Delta_1(s_i)$ & $\varepsilon_{\max} (\Delta_1^{pr})$ &   $\Delta_1^{pr}(s_i)$   \\ \hline    
1      & 61.63 & 21.88   & 40.75    &    $34.93$  &  40.75 &2.56    \\ \hline  
2     &51.98 &   18.46  & $19.34$  &     $33.92$  & 19.34&1.07    \\ \hline
3     &14.49 & 5.14     &  $0.59$   &     $1.47$ &14.48 & 0.64\\  \hline
4   &0.76 &  0.27       & $0.31$     &    $0.26$ &14.45 &5.46     \\ \hline 
5   & 0.11& 0.04      & 0.06   &     $0.11$   &0.26 & 0.26     \\ \hline
6  & 0.0016 & $5.86\times 10^{-4}$&  $0.04$   &     $0.04$  &0.0024 & 0.0024 \\  \hline
7    &--- &---  &  $6.81\times 10^{-4}$  & $7.65\times 10^{-4}$&  $1.28\times 10^{-5}$     &     $1.28\times 10^{-5}$ \\  \hline
  \end{tabular}
\end{center}
\end{table}
\begin{table}[h]
\begin{center}
\caption{CD player, $\varepsilon_{tol}=10^{-3}$, $q=3$, $r=52$.}
\label{error_CD player_est2}
\begin{tabular}{|c|c|c|c|c|} \hline
iteration $i$ & $\varepsilon_{\max} (\Delta_2)$   &    $\Delta_2(s_i)$ & $\varepsilon_{\max} (\Delta_2^{pr})$  & $\Delta_2^{pr}(s_i)$    \\ \hline    
1      & 40.75                                   &    $51$   &  40.75& 46.1   \\ \hline  
2       & $30.16$                &     $35.75$   &19.34  &52.2     \\ \hline
3       &  $0.75$           &     $5.41$  &0.59 & 1.95\\  \hline
4        & $0.32$           &    $0.4$    &0.31 & 0.38  \\ \hline 
5      &  $0.03$          &     $0.03$     &0.06 &  0.19     \\ \hline
6      &  $0.002$          &     $0.002$  &0.04 & 0.04\\  \hline
7      &  $8.28\times 10^{-4}$  &     $8.38\times 10^{-4}$ &$6.82\times 10^{-4}$ &$8.48\times 10^{-4}$ \\  \hline
  \end{tabular}
\end{center}
\end{table}
\begin{table}[h]
\begin{center}
\caption{CD player, $\varepsilon_{tol}=10^{-3}$, $q=3$, $r=52$.}
\label{error_CD player_est3}
\begin{tabular}{|c|c|c|c|c|} \hline
iteration $i$ & $\varepsilon_{\max} (\Delta_3)$   &    $\Delta_3(s_i)$ & $\varepsilon_{\max} (\Delta_3^{pr})$  & $\Delta_3^{pr}(s_i)$    \\ \hline    
1      & 40.75      &    $35.45$   &  40.75& 34.95   \\ \hline  
2       & $19.34$   &   $35.19$   &16.81  &51.76     \\ \hline
3       &  $0.59$   &     $0.84$  &9.1 & 9.1\\  \hline
4        & $0.31$   &    $0.4$    &0.21 & 0.24  \\ \hline 
5      &  $0.05$          &  $0.05$     &0.03 &  0.03     \\ \hline
6      &  $0.002$          &     $0.002$  &0.0016 & 0.0016\\  \hline
7      &  $8.27\times 10^{-4}$  &     $8.27\times 10^{-4}$ &$7.57\times 10^{-4}$ &$7.57\times 10^{-4}$ \\  \hline
  \end{tabular}
\end{center}
\end{table}
We further validate the ROM obtained by the error estimators at samples in $\Xi_{ver}$ including 600 samples randomly taken from $[0,1~\text{MHz}]$, the results are presented in Table~\ref{error__CD player_eff} and plotted in Figures~\ref{fig:CD player_estr1}-\ref{fig:CD player_est3pr}. In Table~\ref{error__CD player_eff}, we compare the effectivity defined as $\textrm{eff}(s):=\Delta(s)/\varepsilon(s)$, the ratio between the given error estimator and its corresponding true error. $\tilde \Delta_r(s)$ still underestimates the true error at most samples. $\Delta_1^{pr}(s)$, $\Delta_2(s)$, $\Delta_3(s)$, $\Delta_3^{pr}(s)$ are equally well, whereas $\Delta_1(s)$ and $\Delta_2^{pr}(s)$ underestimate
the true error too much ($\min(\textrm{eff})<0.1$) at some samples. However, we observe that underestimation happens only at samples with very small true errors $\varepsilon(s)$ being smaller than $10^{-11}$ which may be caused by rounding errors. If we check the error estimators only at true errors larger than  $10^{-11}$, then we obtain the last two columns in the table, which show that except for $\tilde \Delta_r(s)$ the other estimators are tight. 

Figure~\ref{fig:CD player_estr1} further shows the inaccuracy of $\tilde \Delta_r(s)$ validated at the 600 samples in $\Xi_{ver}$. $\Delta_1(s)$ in Figure~\ref{fig:CD player_est1pr} behaves slightly worse than the other proposed estimators, see Figures~\ref{fig:CD player_est2pr}-\ref{fig:CD player_est3pr}. In the following, we will omit the results of $\tilde \Delta_r(\tilde \mu)$ for the other examples, since it is always worse than the others. 
\begin{table}[h]
\begin{center}
\caption{CD player, effectivity of the error estimators.}
\label{error__CD player_eff}
\begin{tabular}{|c||c|c||c|c|} \hline
\multirow{2}{*}{Estimator}  & \multicolumn{2}{|c||}{For all $\varepsilon(s)$} &  \multicolumn{2}{|c|}{For $\varepsilon(s) \geq 10^{-11}$} \\ \cline{2-5}  
& $\min\limits_{s\in \Xi_{ver}}(\textrm{eff})$   &    $\max\limits_{s\in \Xi_{ver}}(\textrm{eff})$ & $\min\limits_{s\in \Xi_{ver}}(\textrm{eff})$   & $\max\limits_{s\in \Xi_{ver}}(\textrm{eff})$  \\ \hline 
$\tilde \Delta_r$  & 0.09 &1.82 &0.26 &0.26 \\ \hline   
$\Delta_1$      & 0.02    &    $80$  & 0.9211&1.1785\\ \hline  
$\Delta_1^{pr}$      & 0.28   &   20.39 &0.9988 &1.0046 \\ \hline
$\Delta_2$       &  $0.12$   &     $17.32$ &0.9987 &1.1653\\  \hline
$\Delta_2^{pr}$         & $0.02$   &    $80$&1.0000 &1.3643  \\ \hline 
$\Delta_3$       &  $0.12$   &     $10.97$ &0.9993 &1.0004\\  \hline
$\Delta_3^{pr}$         & $0.1$   &    $9.13$ &0.9998 &5.31 \\ \hline 
  \end{tabular}
\end{center}
\end{table}
\begin{figure}[h]
\centering
\includegraphics[width=75mm]{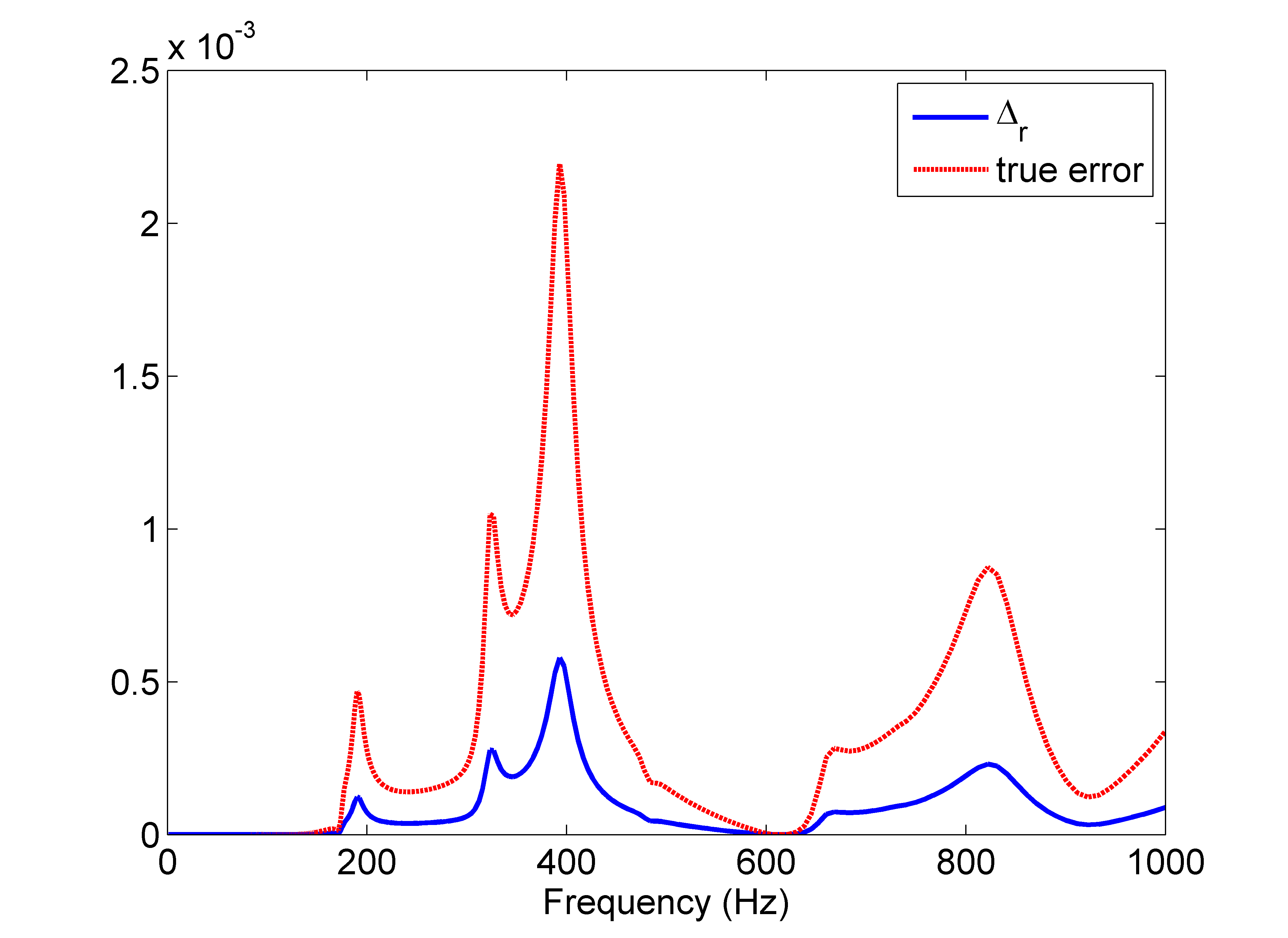}\quad
\includegraphics[width=75mm]{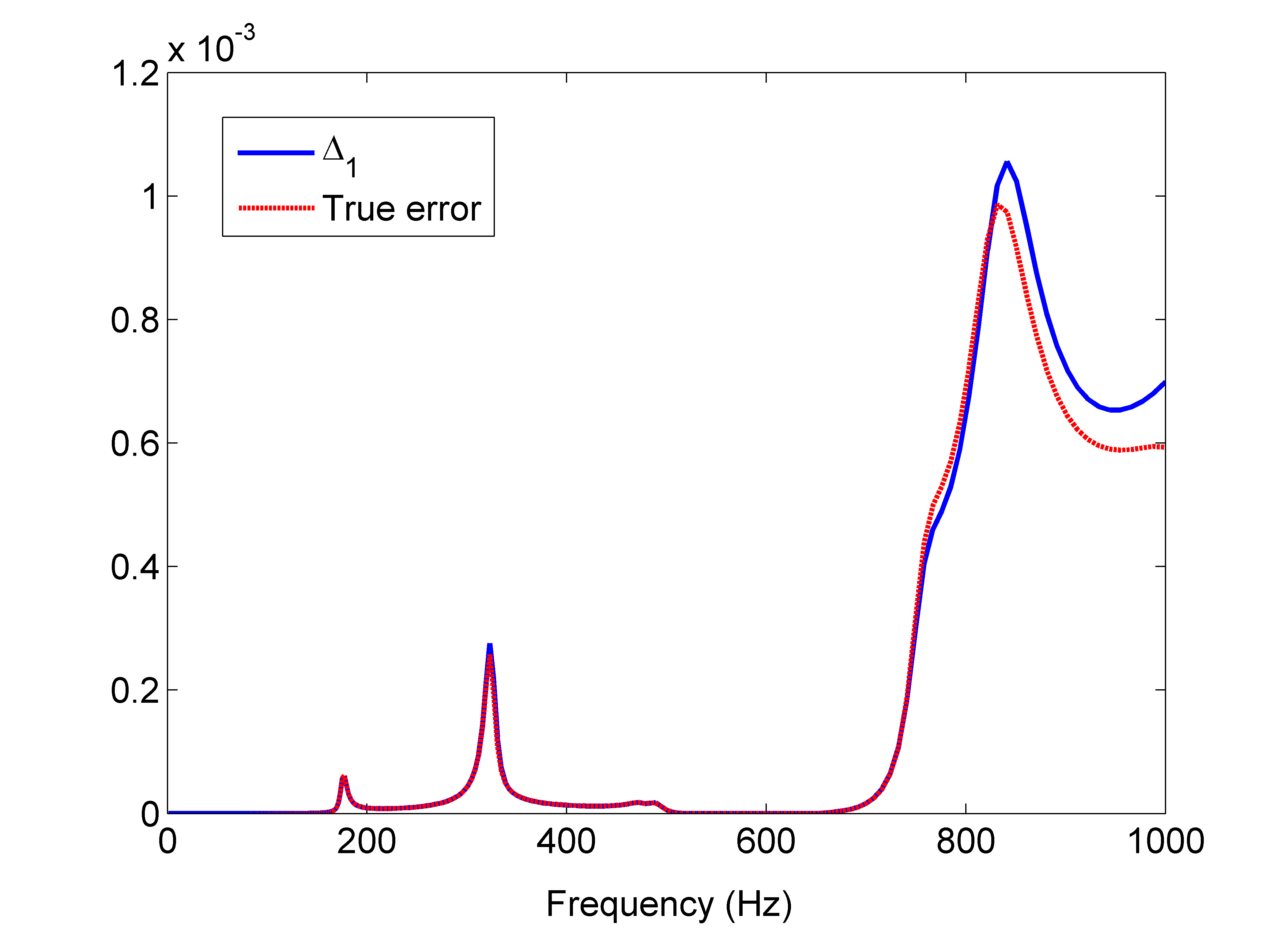} 
\caption{CD player: $\Delta_r(s)$ and $\Delta_1(s)$ vs. the respective true errors at 600 frequency samples.}
\label{fig:CD player_estr1}
\end{figure}
\begin{figure}[h]
\centering
\includegraphics[width=75mm]{CD_player_est1.png}\quad
\includegraphics[width=75mm]{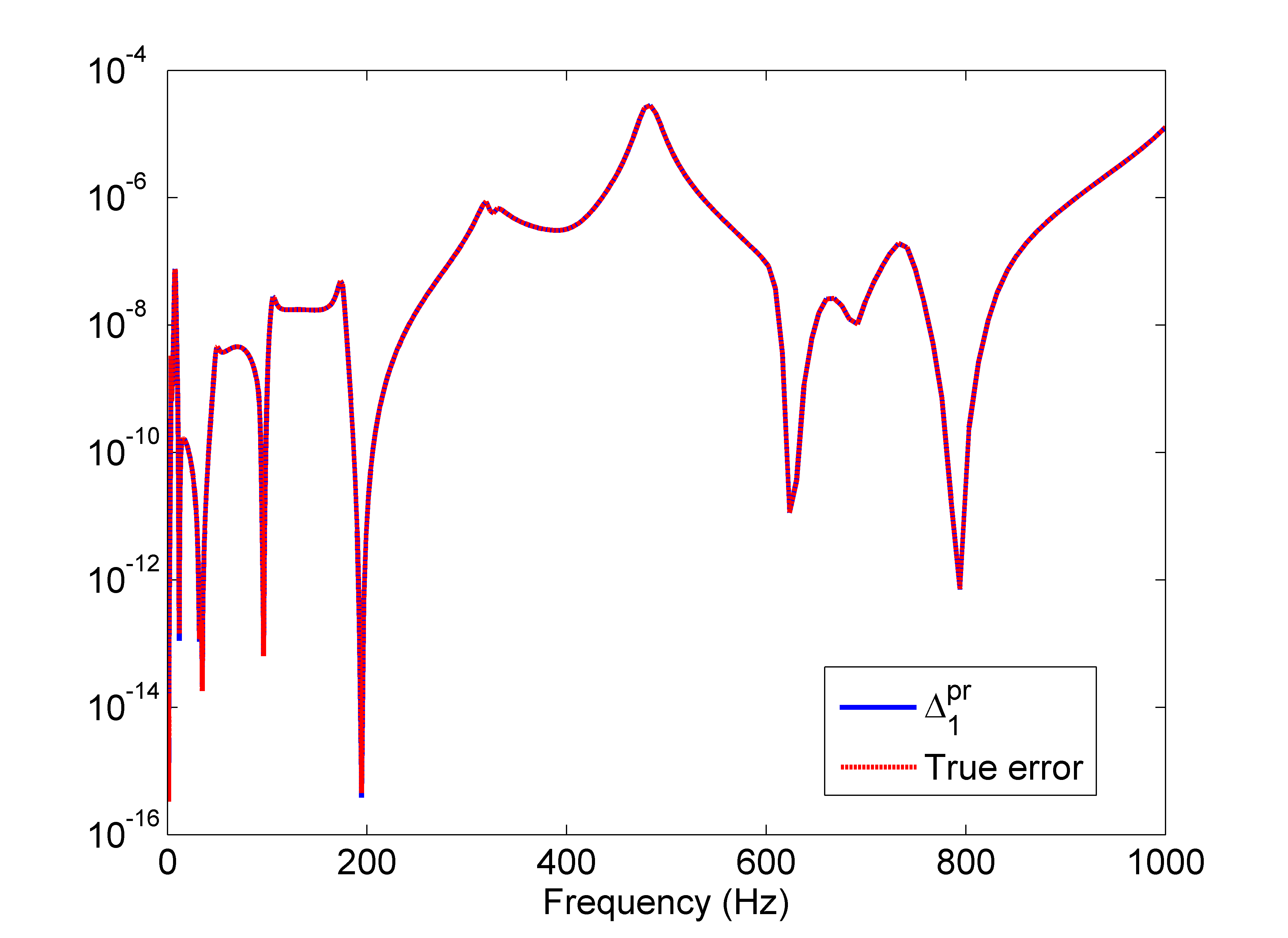} 
\caption{CD player: $\Delta_1(s)$ and $\Delta_1^{pr}(s)$ vs. the respective true errors at 600 frequency samples.}
\label{fig:CD player_est1pr}
\end{figure}
\begin{figure}[h]
\centering
\includegraphics[width=75mm]{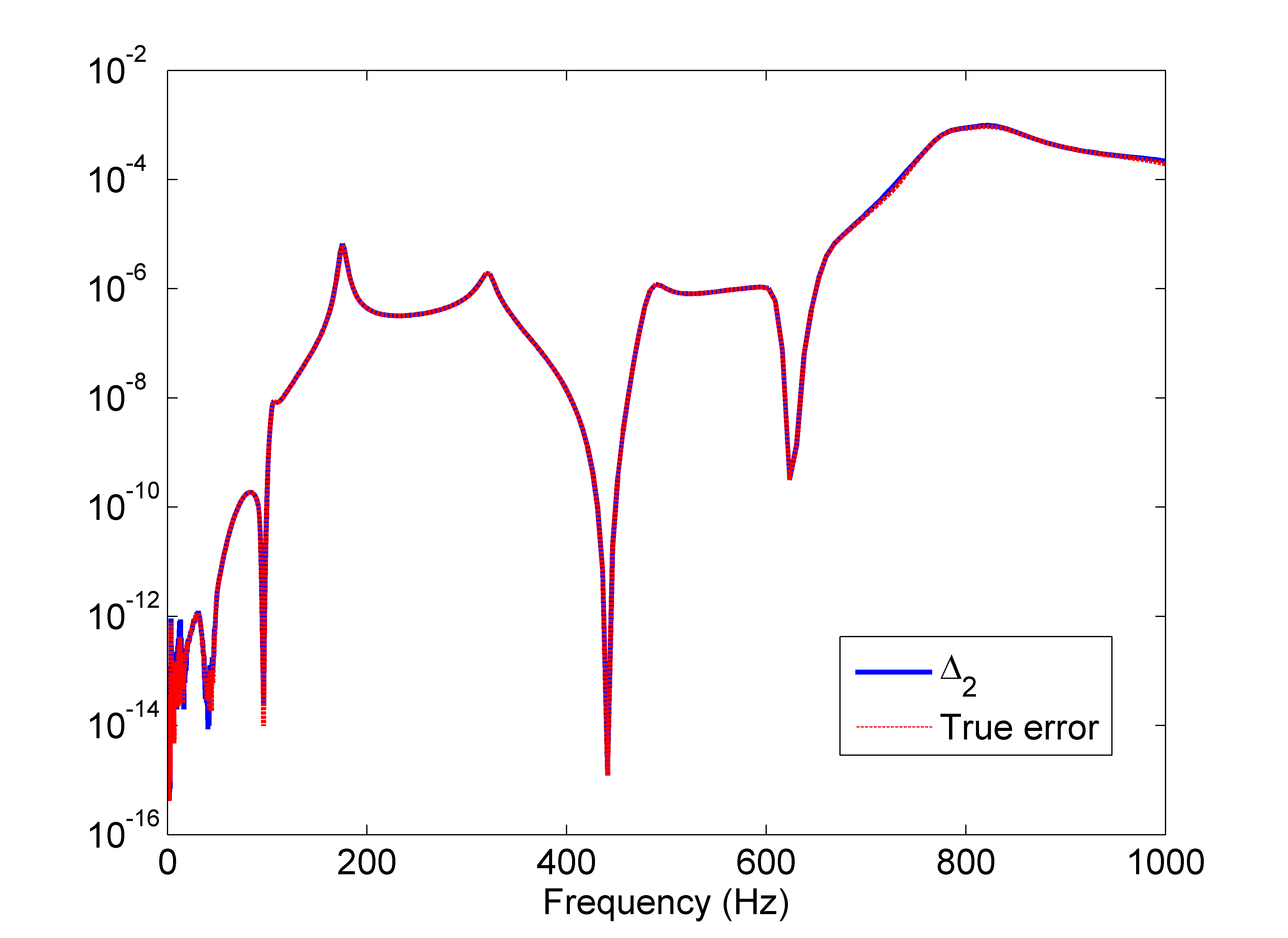}\quad
\includegraphics[width=75mm]{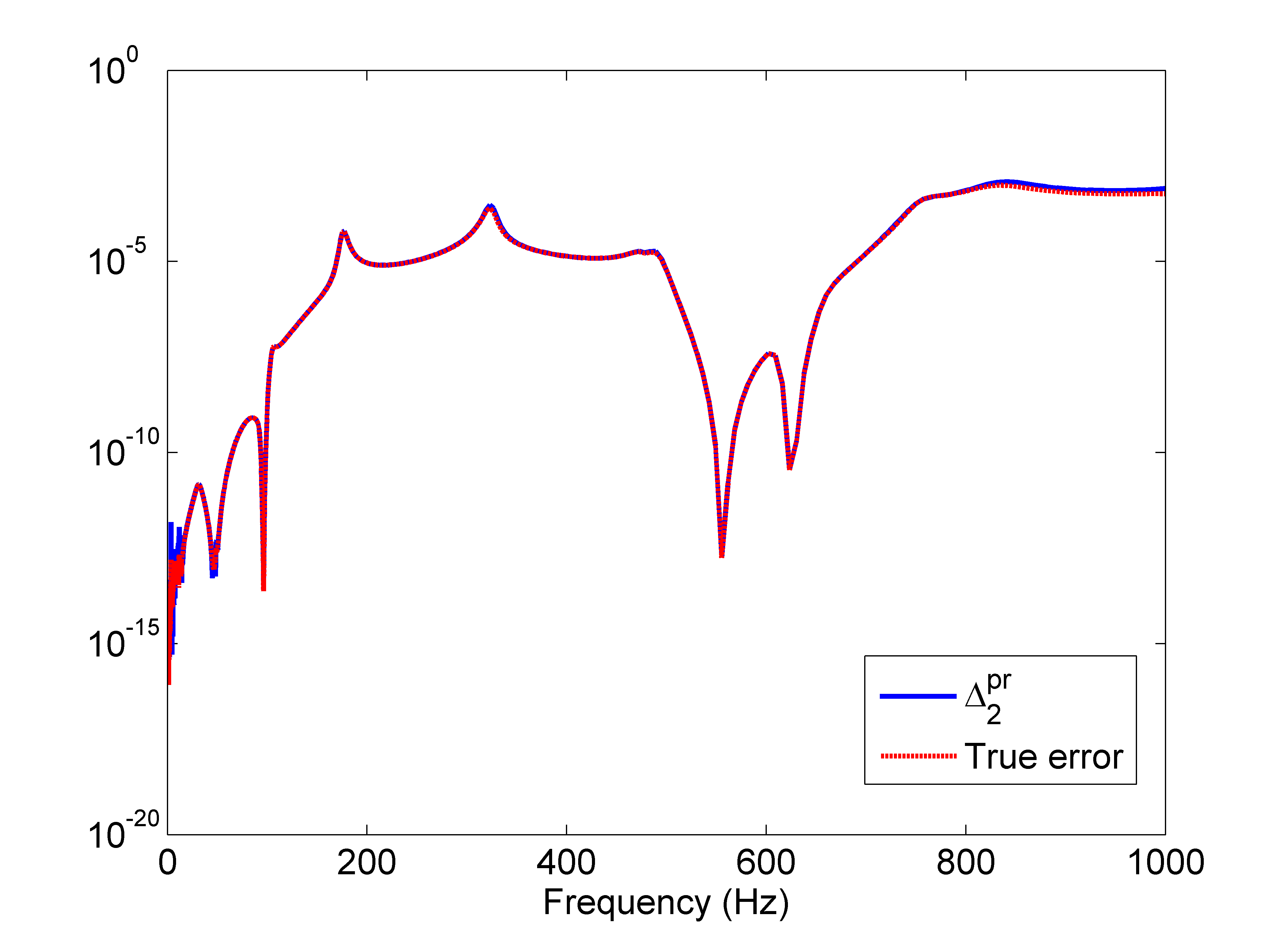}
\caption{CD player: $\Delta_2(s)$ and $\Delta_2^{pr}(s)$ vs. the respective true errors at 600 frequency samples .}
\label{fig:CD player_est2pr}
\end{figure}
\begin{figure}[h]
\centering
\includegraphics[width=75mm]{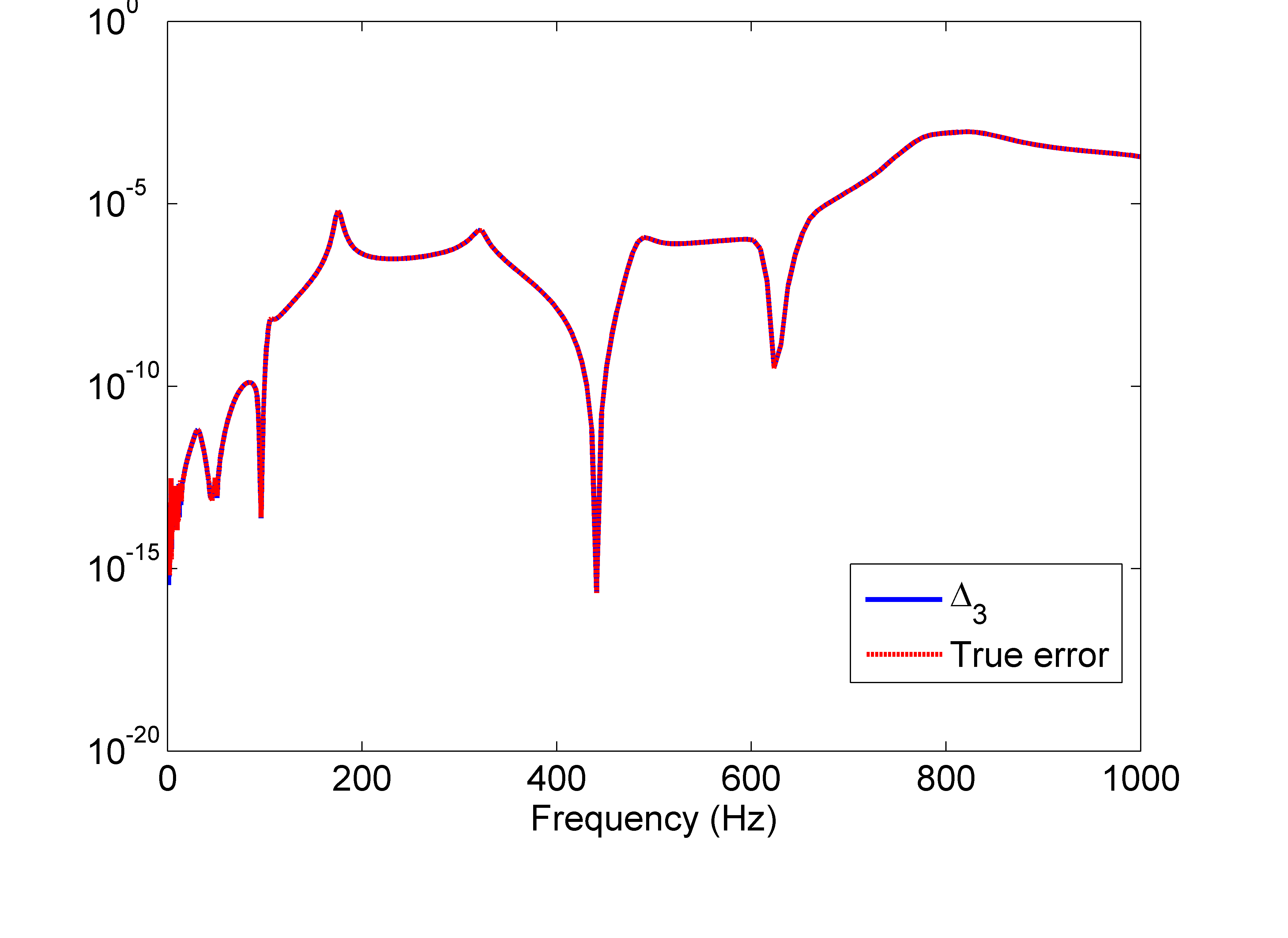}\quad
\includegraphics[width=75mm]{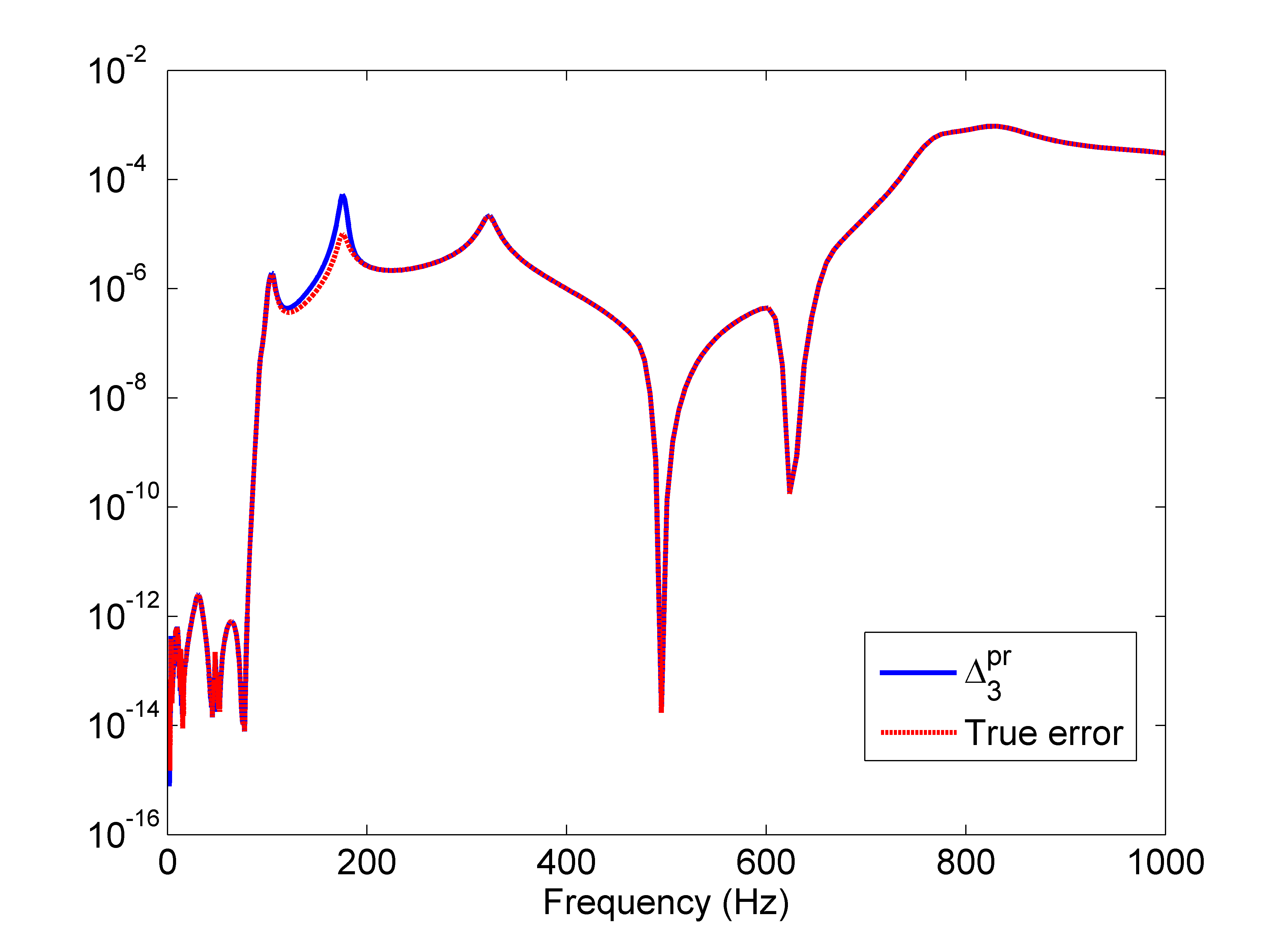}
\caption{CD player: $\Delta_3(s)$ and $\Delta_3^{pr}(s)$ vs. the respective true errors at 600 frequency samples.}
\label{fig:CD player_est3pr}
\end{figure}

\subsection{The RLC tree model}
We use a training set $\Xi$ with 90 frequency samples covering the whole frequency range $[0,3\text{~GHz}]$. The samples $s_i$ are taken using the function $f_i=3\times 10^{i/10}, s_i=2 \pi \jmath, i=1,\ldots, 90$. Here, $\jmath$ is the imaginary unit.  The results of the greedy algorithm using different error estimators are listed in Tables~\ref{error_RLC_est1pr}-\ref{error_RLC_est3pr}. $\Delta_1(s)$ always underestimates the true error, and finally it
makes the greedy algorithm stop before the true error $\varepsilon_{\max}$ is below the tolerance. The other estimators behave like tight upper bounds for the true error in this example, especially $\Delta_3(s)$ and $\Delta_3^{pr}(s)$ which actually measure the true error almost exactly
at the last two iterations.  
\begin{table}[h]
\begin{center}
\caption{RLCtree, $\varepsilon_{tol}=10^{-3}$, $q=3$,  $r=12 (\Delta_1)$, $r=20 (\Delta_1^{pr})$.}
\label{error_RLC_est1pr}
\begin{tabular}{|c||c|c||c|c||} \hline
iteration $i$ & $\varepsilon_{\max} (\Delta_1)$   &  $\Delta_1(s_i)$ & $\varepsilon_{\max} (\Delta_1^{pr})$ &   $\Delta_1^{pr}(s_i)$   \\ \hline    
1       & 0.19    &    $0.01$  &  0.19 &0.22   \\ \hline  
2   & $0.06$  &     $0.006$  &0.02&0.02   \\ \hline
3      &  --- &     --- &$2.54 \times 10^{-6}$& $2.55 \times 10^{-6}$\\  \hline
  \end{tabular}
\end{center}
\end{table}
\begin{table}[h]
\begin{center}
\caption{RLCtree, $\varepsilon_{tol}=10^{-3}$, $q=3$,  $r=20 (\Delta_2)$, $r=19 (\Delta_2^{pr})$.}
\label{error_RLC_est2pr}
\begin{tabular}{|c||c|c||c|c||} \hline
iteration $i$ & $\varepsilon_{\max} (\Delta_2)$   &  $\Delta_2(s_i)$ & $\varepsilon_{\max} (\Delta_2^{pr})$ &   $\Delta_2^{pr}(s_i)$   \\ \hline    
1       & 0.19    &    $0.63$  &  0.19 &0.22   \\ \hline  
2   & $0.02$  &     $0.06$  &0.02&0.05  \\ \hline
3      &  $6.13 \times 10^{-6}$&  $6.45 \times 10^{-6}$&$2.25 \times 10^{-5}$& $1.05 \times 10^{-4}$\\  \hline
  \end{tabular}
\end{center}
\end{table}
\begin{table}[h]
\begin{center}
\caption{RLCtree, $\varepsilon_{tol}=10^{-3}$, $q=3$,  $r=20$.}
\label{error_RLC_est3pr}
\begin{tabular}{|c||c|c||c|c||} \hline
iteration $i$ & $\varepsilon_{\max} (\Delta_3)$   &  $\Delta_3(s_i)$ & $\varepsilon_{\max} (\Delta_3^{pr})$ &   $\Delta_3^{pr}(s_i)$   \\ \hline    
1       & 0.19    &    $0.22$  &  0.19 &0.29   \\ \hline  
2   & $0.02$  &     $0.02$  &0.02&0.02   \\ \hline
3      &  $2.54 \times 10^{-6}$ &   $2.55\times 10^{-6}$ &$2.54 \times 10^{-6}$& $2.54 \times 10^{-6}$\\  \hline
  \end{tabular}
\end{center}
\end{table}
The derived ROMs using different error estimators are validated on a validation set $\Xi_{ver}$ with 900 samples in the interesting frequency range. The effectivity of every error estimator is listed in Table~\ref{tab:RLC_eff}. If consider the overall effectivity, then all the estimators underestimate the true error too much except for $\Delta_2^{pr}(s)$. However, if only consider true errors which are bigger than $10^{-11}$, then $\Delta_1^{pr}(s)$, $\Delta_3(s)$ and $\Delta_3^{pr}(s)$ are the best ones, $\Delta_2(s)$ is also good, $\Delta_2^{pr}(s)$ overestimate the true error more than many others. It is clear that $\Delta_1(s)$ is not a good error estimator any more. Figures~\ref{fig:RLC_est1pr}-\ref{fig:RLC_est3pr} further show the behaviors of the error estimators over the sample set $\Xi_{ver}$ including 900 samples, which are in agreement with the above analysis for the data in Table~\ref{tab:RLC_eff}.

\begin{table}[h]
\begin{center}
\caption{RLCtree, effectivity of the error estimators.}
\label{tab:RLC_eff}
\begin{tabular}{|c||c|c||c|c|} \hline
\multirow{2}{*}{Estimator}  & \multicolumn{2}{|c||}{For all $\varepsilon(s)$} &  \multicolumn{2}{|c|}{For $\varepsilon(s) \geq 10^{-11}$} \\ \cline{2-5}  
& $\min\limits_{s\in \Xi_{ver}}(\textrm{eff})$   &    $\max\limits_{s\in \Xi_{ver}}(\textrm{eff})$ & $\min\limits_{s\in \Xi_{ver}}(\textrm{eff})$   & $\max\limits_{s\in \Xi_{ver}}(\textrm{eff})$  \\ \hline 
$\Delta_1$      & 0.002    &    $285$  & 0.006&132\\ \hline  
$\Delta_1^{pr}$      & 0.002  &   253&0.9001 &1.0826 \\ \hline
$\Delta_2$       &  $0.004$   &     $244$ &0.37 &51\\  \hline
$\Delta_2^{pr}$         & $0.56$   &    $102$&0.68 &102  \\ \hline 
$\Delta_3$       &  $0.008$   &     $258$ &0.9 &1.2337\\  \hline
$\Delta_3^{pr}$         & $0.008$   &    $258$ &0.9 &1.0894 \\ \hline 
  \end{tabular}
\end{center}
\end{table}
\begin{figure}[h]
\centering
\includegraphics[width=75mm]{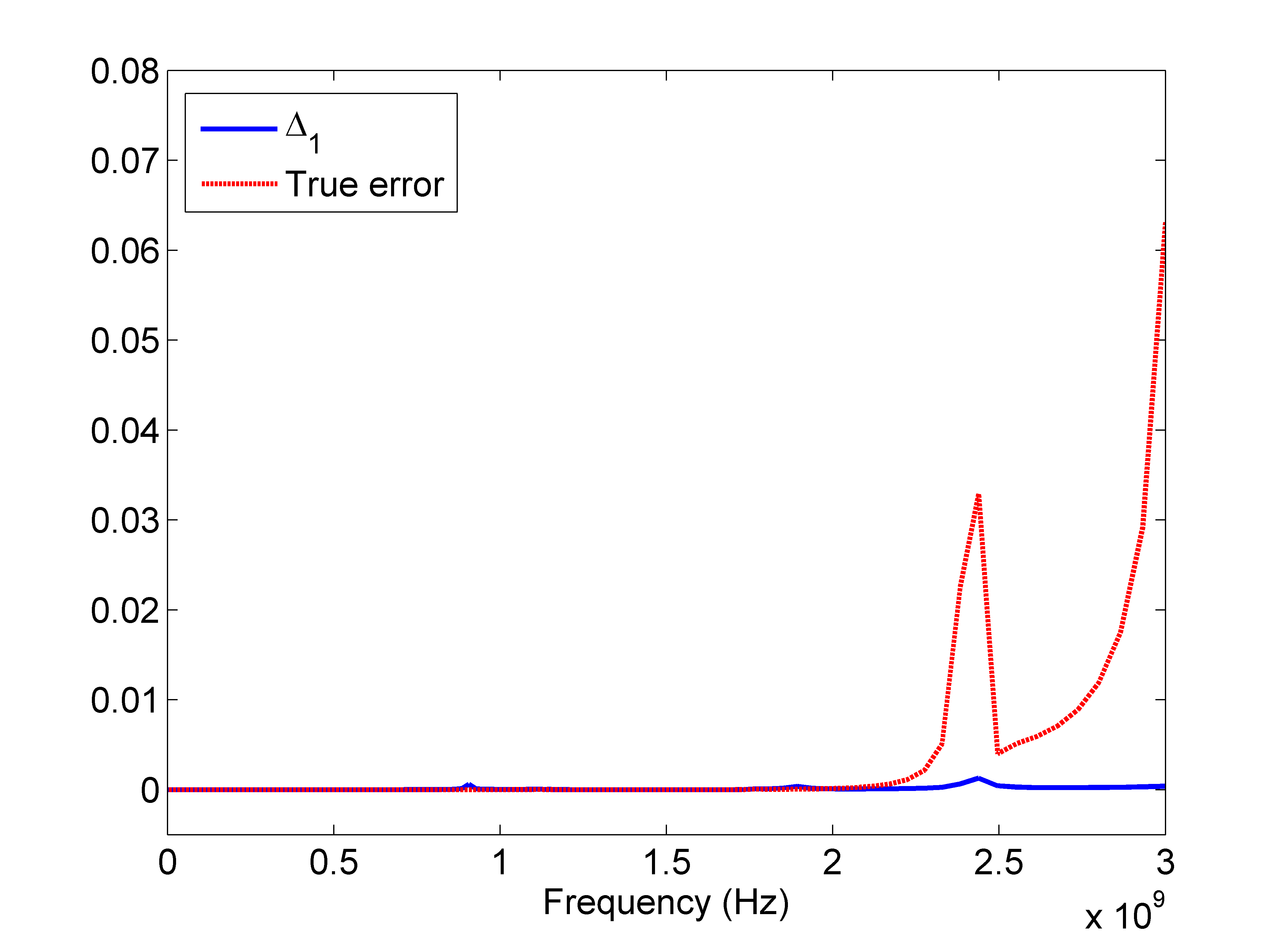}\quad
\includegraphics[width=75mm]{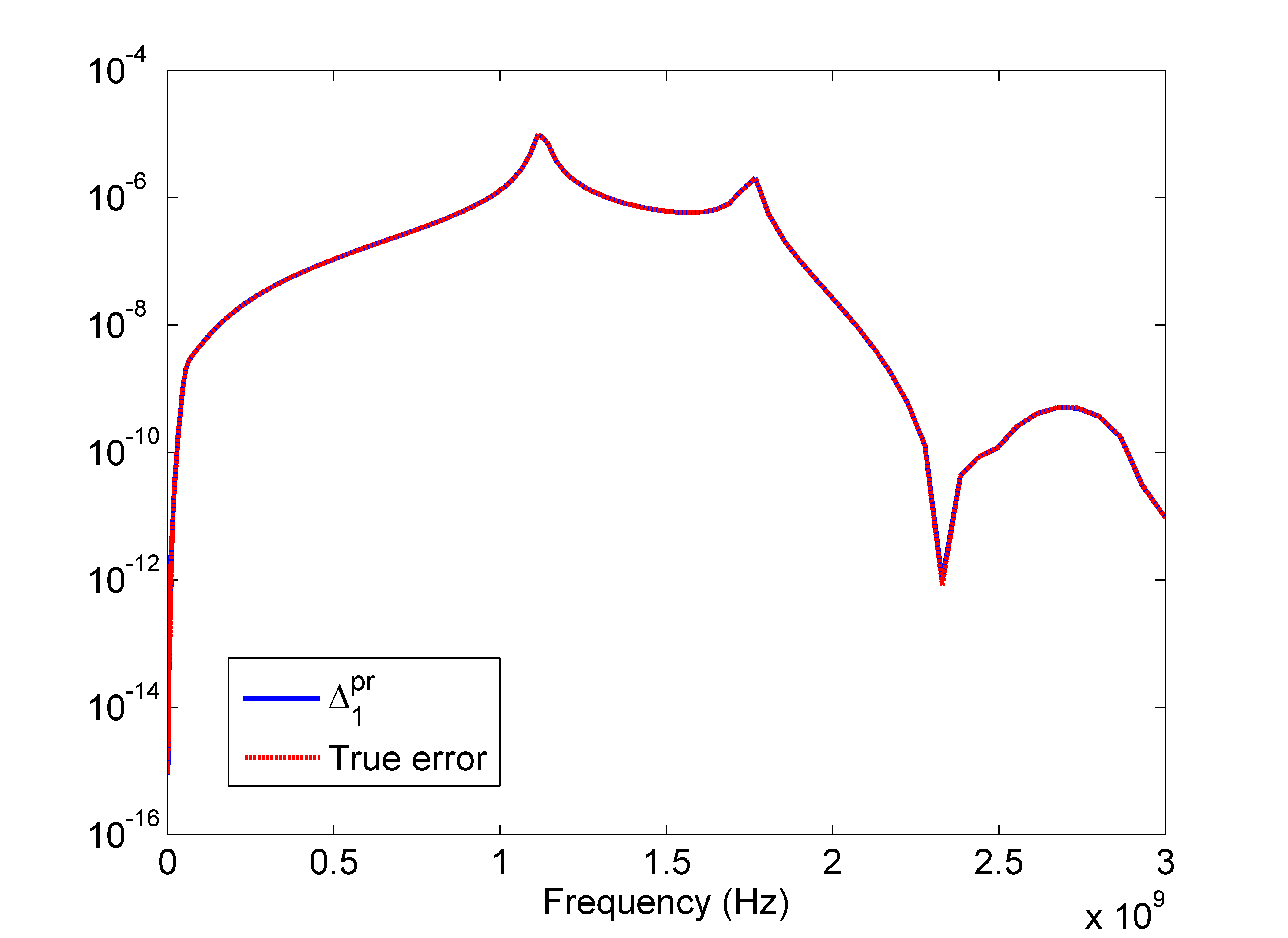} 
\caption{RCLtree: $\Delta_1(s)$ and $\Delta_1^{pr}(s)$ vs. the respective true errors at 900 frequency samples.}
\label{fig:RLC_est1pr}
\end{figure}
\begin{figure}[h]
\centering
\includegraphics[width=75mm]{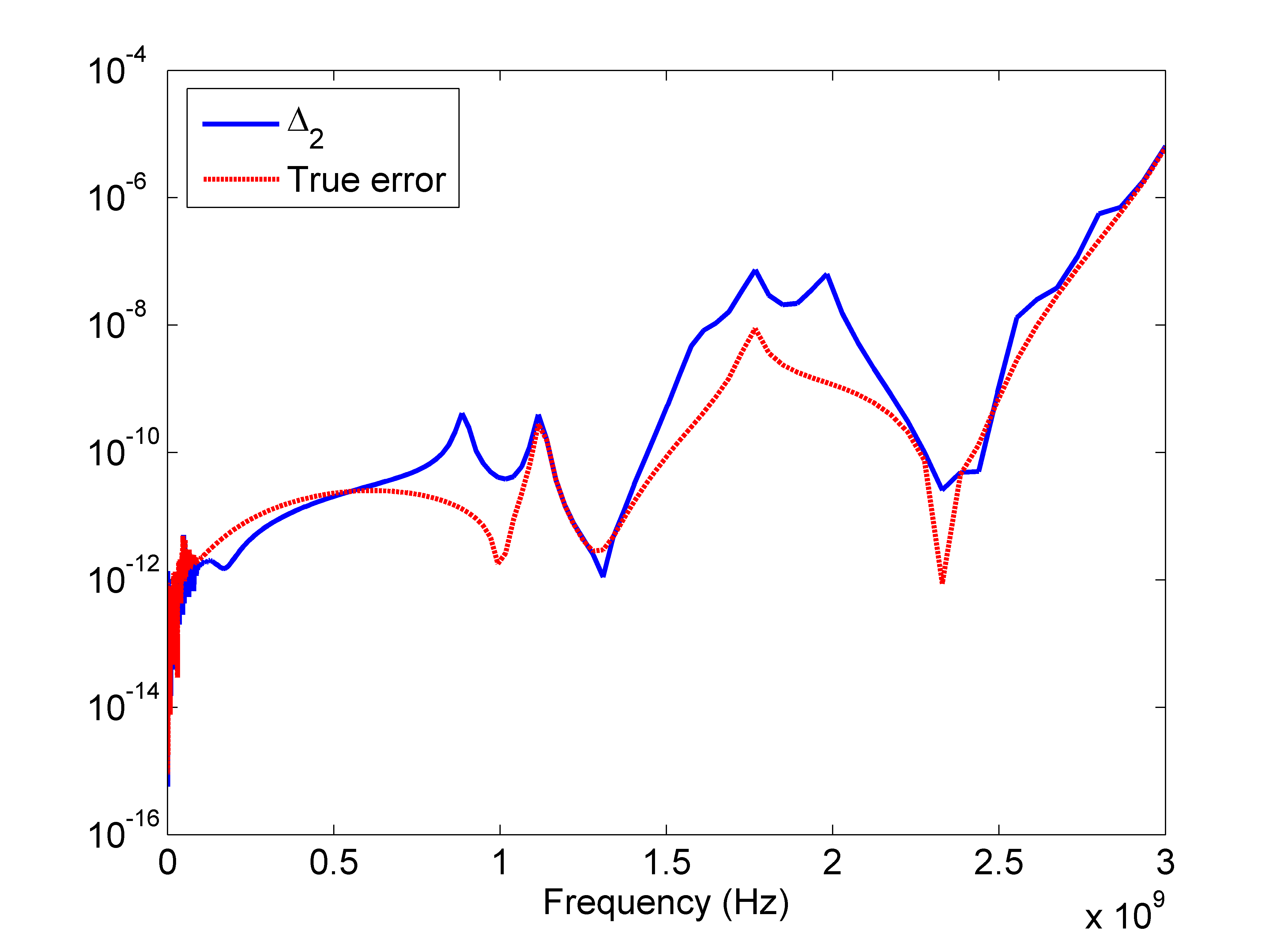}\quad
\includegraphics[width=75mm]{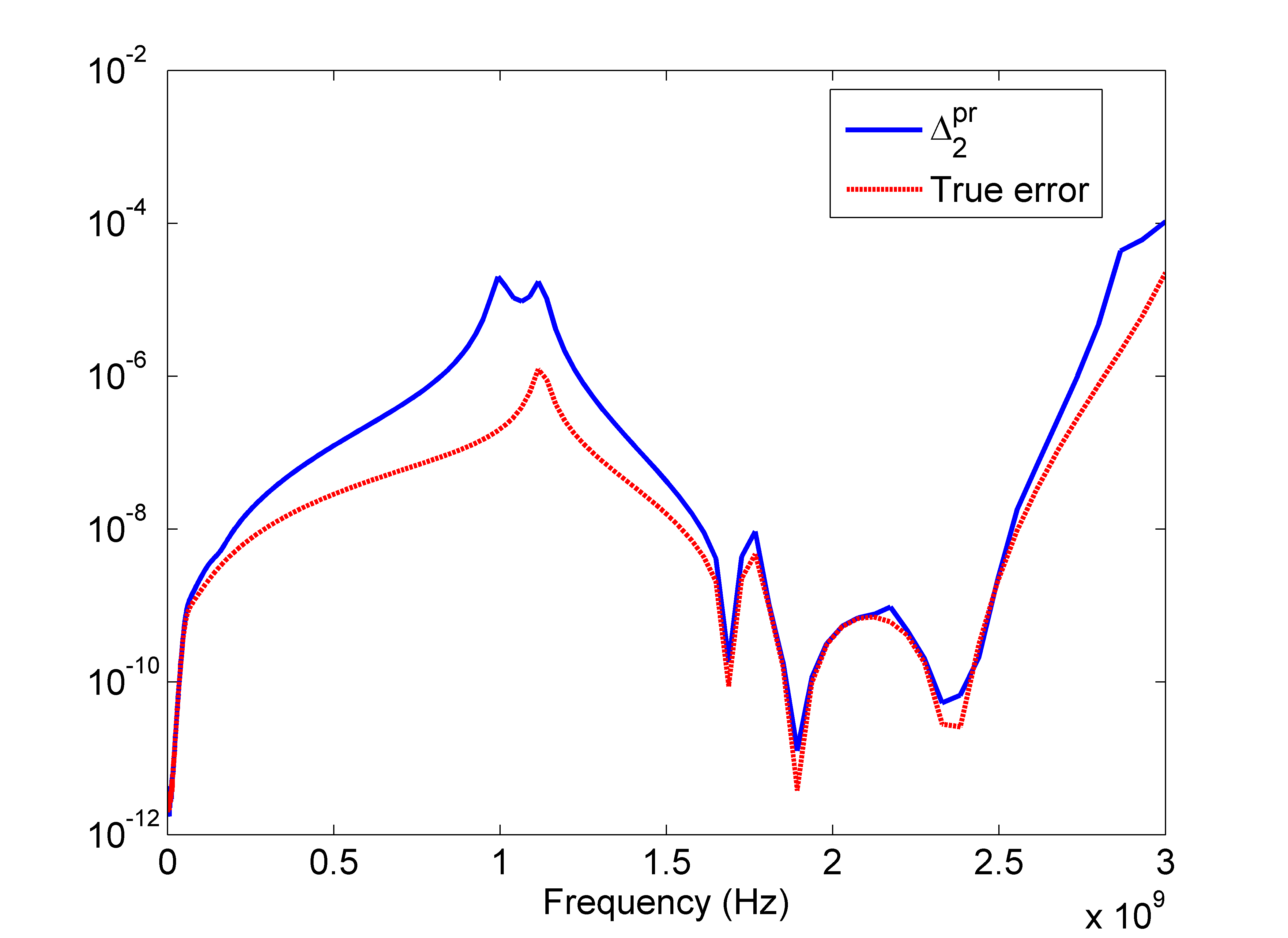}
\caption{RLCtree: $\Delta_2(s)$ and $\Delta_2^{pr}(s)$ vs. the respective true errors at 900 frequency samples .}
\label{fig:RLC_est2pr}
\end{figure}
\begin{figure}[h]
\centering
\includegraphics[width=75mm]{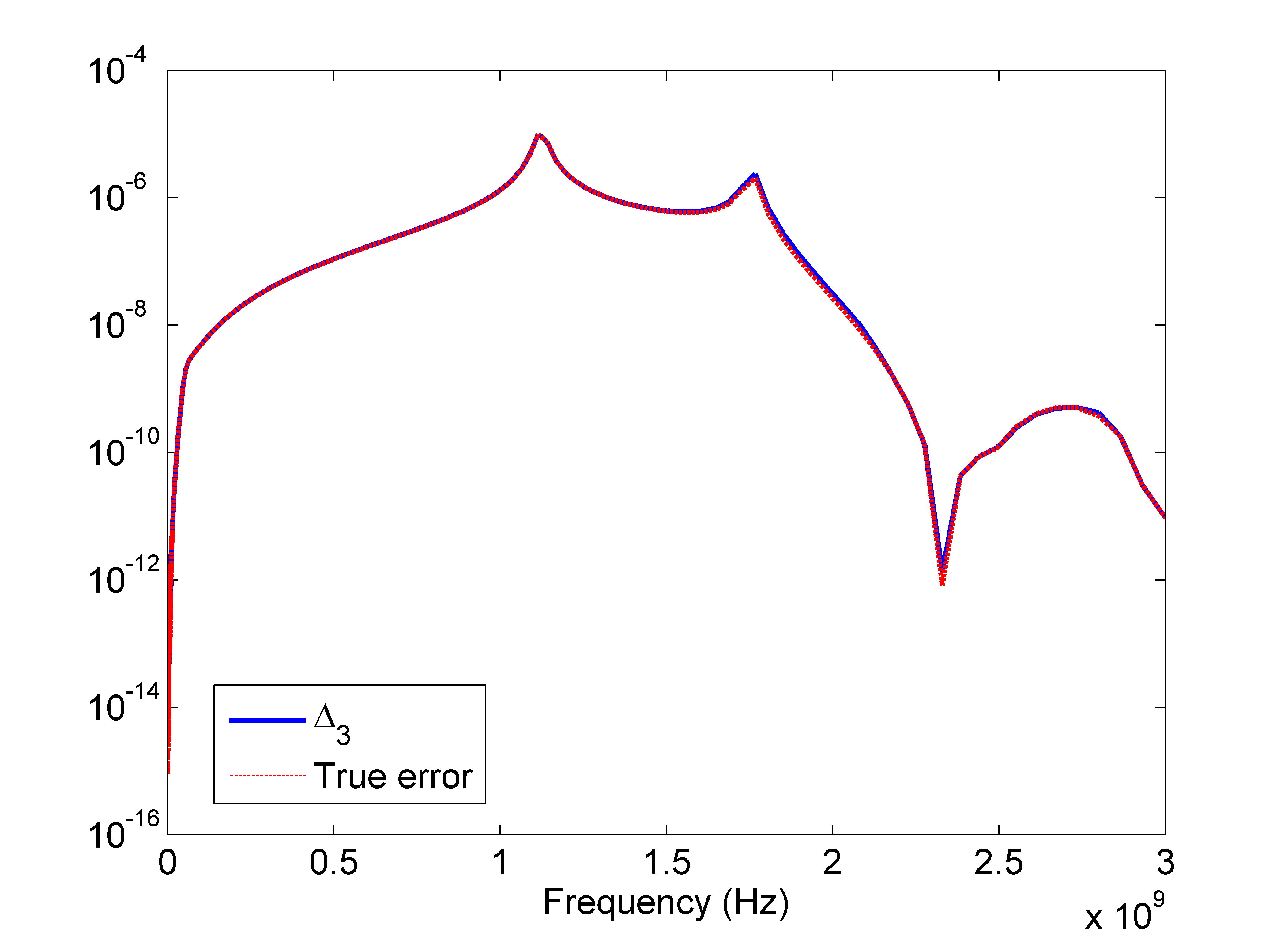}\quad
\includegraphics[width=75mm]{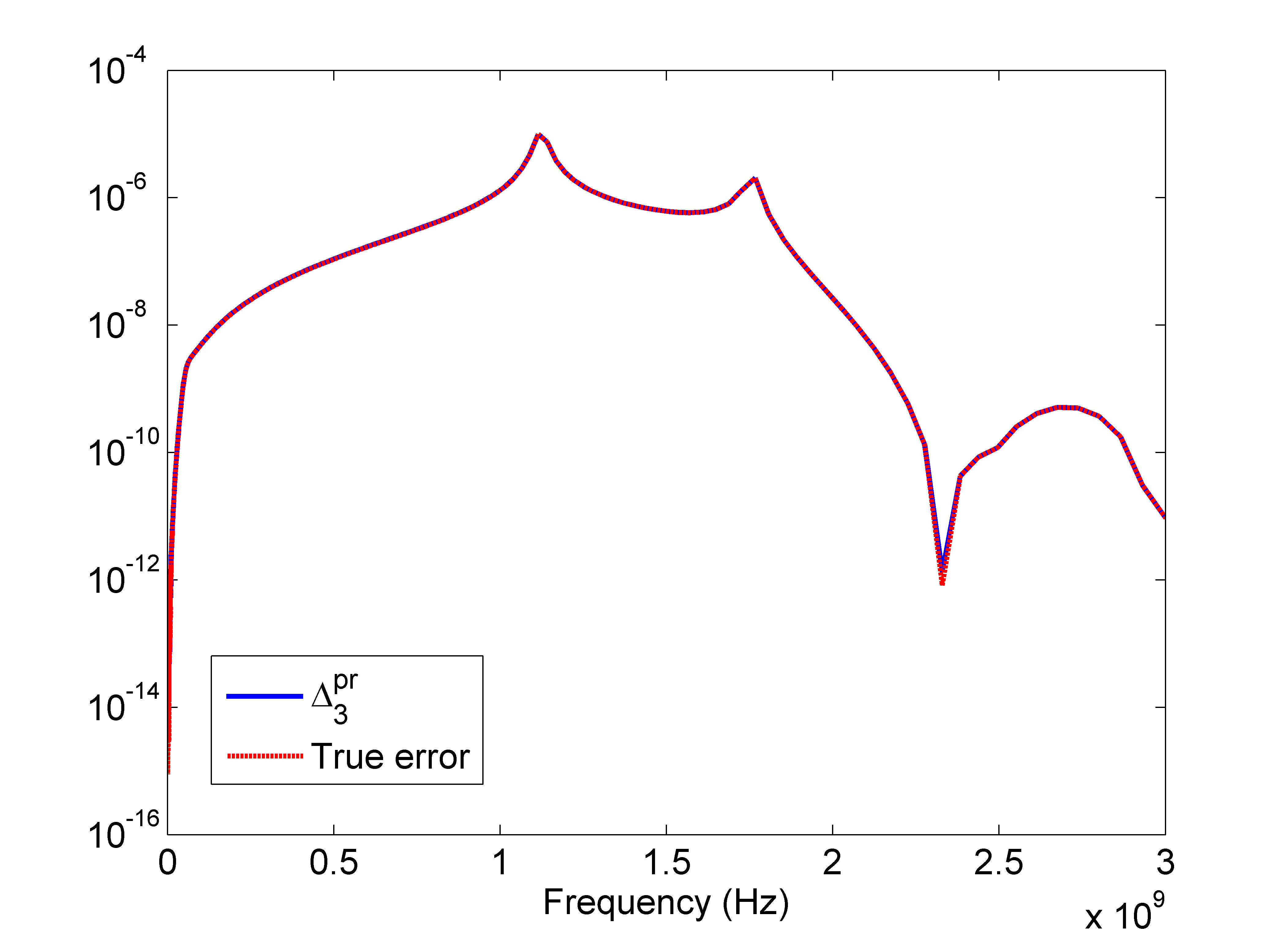}
\caption{RLCtree: $\Delta_3(s)$ and $\Delta_3^{pr}(s)$ vs. the respective true errors at 900 frequency samples.}
\label{fig:RLC_est3pr}
\end{figure}
\subsection{MIMO example}
This example has the same frequency range as the second example, therefore we use the same $\Xi$ as for the RLC tree model. The error estimator is the maximal error estimator defined as
$$\Delta(s)=\max\limits_{ij} \Delta_{ij}(s),$$
where $\Delta_{ij}(s)$ estimates the true error $\epsilon_{ij}(s)=|H_{ij}(s)-\hat H_{ij}(s)|$. Here $H_{ij}(s)$ and $\hat H_{ij}(s)$ are the transfer functions corresponding to the $i$-th input port and $j$-th output port of the original model and the ROM, respectively. 
The true error is the maximal true error $\epsilon(s)=\max\limits_{ij}|\epsilon_{ij}(s)|$, and $\epsilon_{\textrm{max}}=\max\limits_{s\in \Xi} \epsilon(s)$ as defined before. 

The results of Algorithm~\ref{alg:greedy_nonpara} using different error estimators are listed in Tables~\ref{error_MIMO_est1pr}-\ref{error_MIMO_est3pr}. Algorithm~\ref{alg:greedy_nonpara} stops before the true error $\epsilon_{\max}$ is below the tolerance when using $\Delta_1(s)$, whereas $\Delta_1^{pr}(s)$, $\Delta_3(s)$ and $\Delta_3^{pr}(s)$ exactly estimate the true error at each iteration step. $\Delta_2(s)$ and its primal variation  $\Delta_2^{pr}(s)$ produce the same results and make the algorithm converge in 3 iterations. Note that $\Delta_3(s)$ and $\Delta_3^{pr}(s)$  also yield the same results.
\begin{table}[h]
\begin{center}
\caption{MIMO example, $\varepsilon_{tol}=10^{-3}$, $q=3$,  $r=20 (\Delta_1)$, $r=52 (\Delta_1^{pr})$.}
\label{error_MIMO_est1pr}
\begin{tabular}{|c||c|c||c|c||} \hline
iteration $i$ & $\varepsilon_{\max} (\Delta_1)$   &  $\Delta_1(s_i)$ & $\varepsilon_{\max} (\Delta_1^{pr})$ &   $\Delta_1^{pr}(s_i)$   \\ \hline    
1       & 0.28    &    $3.16 \times 10^{-5}$  &  0.28&0.28   \\ \hline  
2      &  --- &     --- &$5.91 \times 10^{-5}$& $5.91 \times 10^{-5}$\\  \hline
  \end{tabular}
\end{center}
\end{table}
\begin{table}[h]
\begin{center}
\caption{MIMO example, $\varepsilon_{tol}=10^{-3}$, $q=3$,  $r=73$.}
\label{error_MIMO_est2pr}
\begin{tabular}{|c||c|c||c|c||} \hline
iteration $i$ & $\varepsilon_{\max} (\Delta_2)$   &  $\Delta_2(s_i)$ & $\varepsilon_{\max} (\Delta_2^{pr})$ &   $\Delta_2^{pr}(s_i)$   \\ \hline    
1       & 0.28    &    0.28  &  0.28&0.28   \\ \hline  
2      & $5.91 \times 10^{-5}$&    $2.3 \times 10^{-3}$&$5.91 \times 10^{-5}$& $2.3 \times 10^{-3}$\\  \hline
3      &  $4.72\times 10^{-8}$&  $1.43 \times 10^{-7}$&$4.72\times 10^{-8}$& $1.43 \times 10^{-7}$\\  \hline
  \end{tabular}
\end{center}
\end{table}
\begin{table}[h]
\begin{center}
\caption{MIMO example, $\varepsilon_{tol}=10^{-3}$, $q=3$,  $r=52$.}
\label{error_MIMO_est3pr}
\begin{tabular}{|c||c|c||c|c||} \hline
iteration $i$ & $\varepsilon_{\max} (\Delta_3)$   &  $\Delta_3(s_i)$ & $\varepsilon_{\max} (\Delta_3^{pr})$ &   $\Delta_3^{pr}(s_i)$   \\ \hline    
1       & 0.28    &    0.28  &  0.28&0.28   \\ \hline  
2      & $5.91 \times 10^{-5}$&    $5.91 \times 10^{-5}$&$5.91 \times 10^{-5}$& $5.91 \times 10^{-5}$\\  \hline
  \end{tabular}
\end{center}
\end{table}

The ROMs constructed by Algorithm~\ref{alg:greedy_nonpara} using the error estimators are further validated over a validation set $\Xi_{ver}$ with 900 samples, respectively. 
Table~\ref{tab:MIMO_eff} lists the effectivity values of the error estimators. Among them, $\Delta_2(s)$ and its primal variation $\Delta_2^{pr}(s)$ are the best ones and have the same effectivity values. $\Delta_1^{pr}(s)$, $\Delta_3(s)$
and $\Delta_3^{pr}(s)$ have similar results and are still good.

Figures~\ref{fig:MNA4_est1pr}-\ref{fig:MNA4_est3pr} plot the error estimators and the corresponding true errors of the ROMs. The waveforms of the error estimators well reflect the data in Table~\ref{tab:MIMO_eff}. It is noticed that the maximal true errors over the validation sample set $\Xi_{ver}$ obtained by $\Delta_1^{pr}(s)$, $\Delta_3(s)$ and $\Delta_3^{pr}(s)$ are still bigger than the error tolerance, though they are exactly reproduced by the error estimators. Since the error estimators accurately measure the maximal true error, the ROMs can be further improved by adding one more expansion point from $\Xi_{ver}$ (rather than $\Xi$) at which the error estimators are maximal. This will certainly incur more computational costs. Therefore, $\Delta_2(s)$ and $\Delta_2^{pr}(s)$ outperform the other ones for this model.

\begin{table}[h]
\begin{center}
\caption{MIMO example, effectivity of the error estimators.}
\label{tab:MIMO_eff}
\begin{tabular}{|c||c|c||c|c|} \hline
\multirow{2}{*}{Estimator}  & \multicolumn{2}{|c||}{For all $\varepsilon(s)$} &  \multicolumn{2}{|c|}{For $\varepsilon(s) \geq 10^{-11}$} \\ \cline{2-5}  
& $\min\limits_{s\in \Xi_{ver}}(\textrm{eff})$   &    $\max\limits_{s\in \Xi_{ver}}(\textrm{eff})$ & $\min\limits_{s\in \Xi_{ver}}(\textrm{eff})$   & $\max\limits_{s\in \Xi_{ver}}(\textrm{eff})$  \\ \hline 
$\Delta_1$      & $ 8.78\times10^{-8}$   &    $2.53$  &  $ 8.78\times10^{-8}$&1.43\\ \hline  
$\Delta_1^{pr}$      & 0.1  &   40&0.2 &26 \\ \hline
$\Delta_2$       &  $0.1$   &     $5$ &0.2 &3.5\\  \hline
$\Delta_2^{pr}$         & $0.1$   &    $5$&0.2&3.5 \\ \hline 
$\Delta_3$       &  $0.1$   &     $25$ &0.2 &21\\  \hline
$\Delta_3^{pr}$         & $0.1$   &    $28$ &0.2 &25\\ \hline 
  \end{tabular}
\end{center}
\end{table}
\begin{figure}[h]
\centering
\includegraphics[width=75mm]{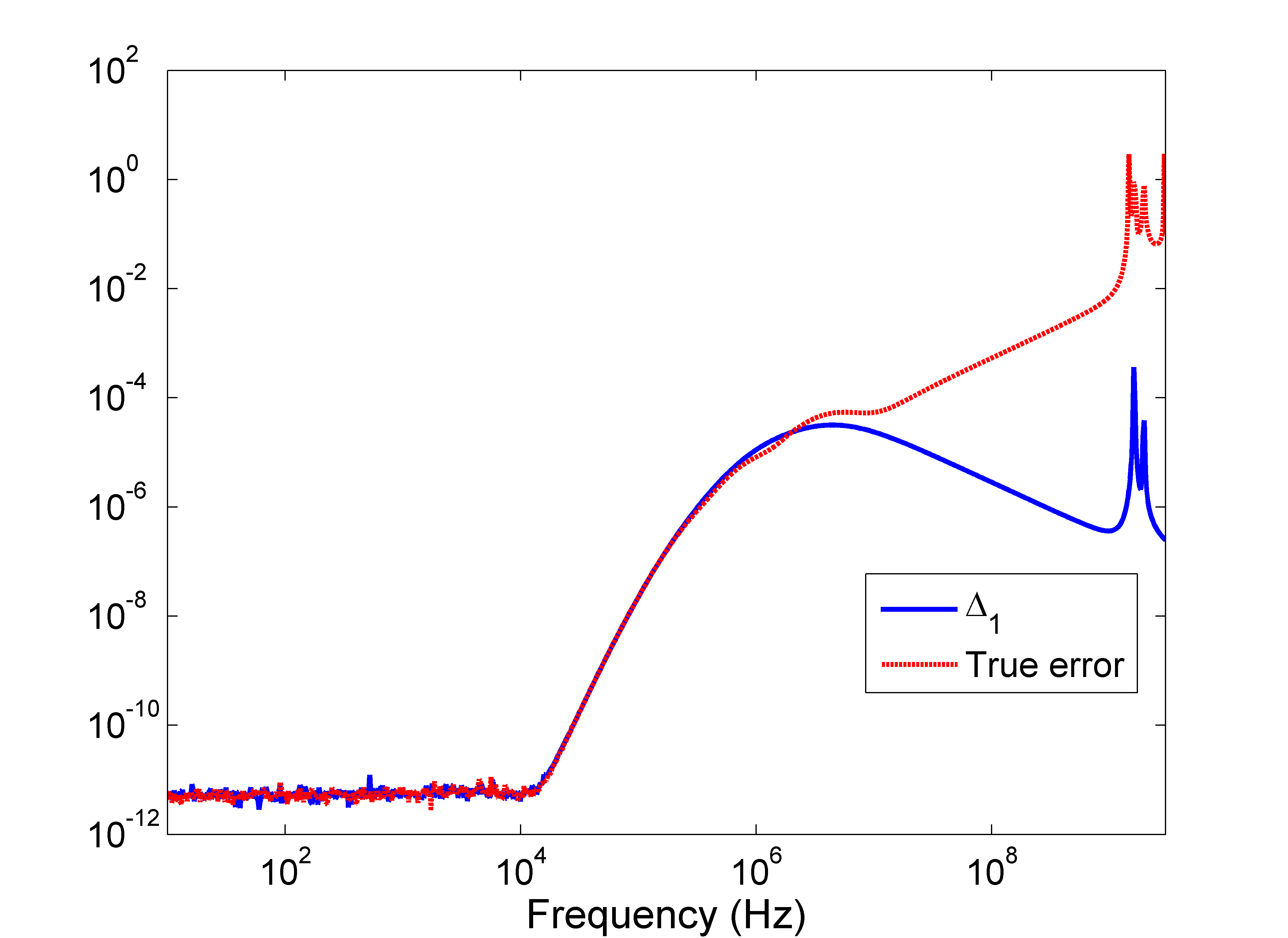}\quad
\includegraphics[width=75mm]{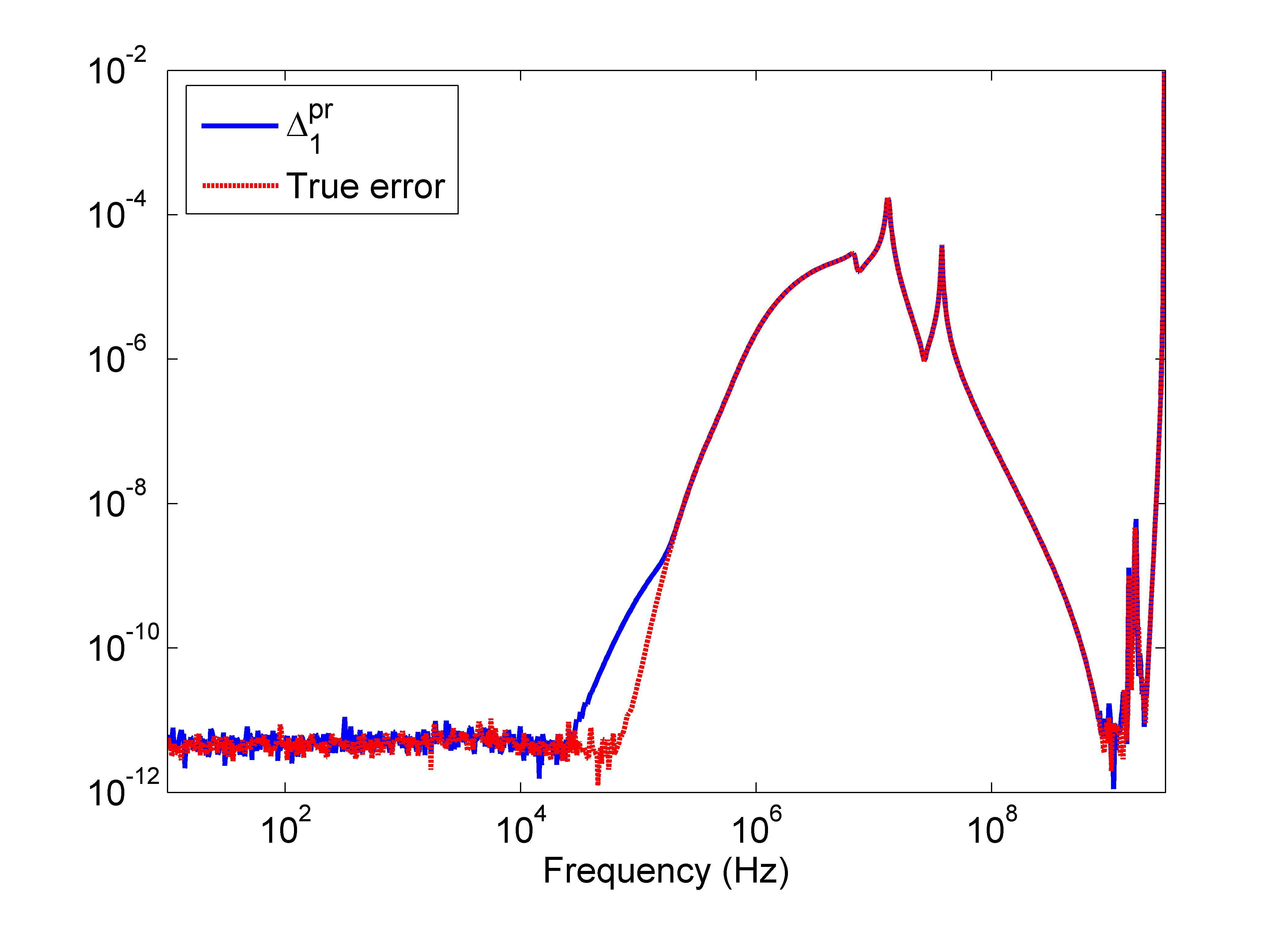} 
\caption{MIMO example: $\Delta_1(s)$ and $\Delta_1^{pr}(s)$ vs. the respective true errors at 900 frequency samples.}
\label{fig:MNA4_est1pr}
\end{figure}
\begin{figure}[h]
\centering
\includegraphics[width=75mm]{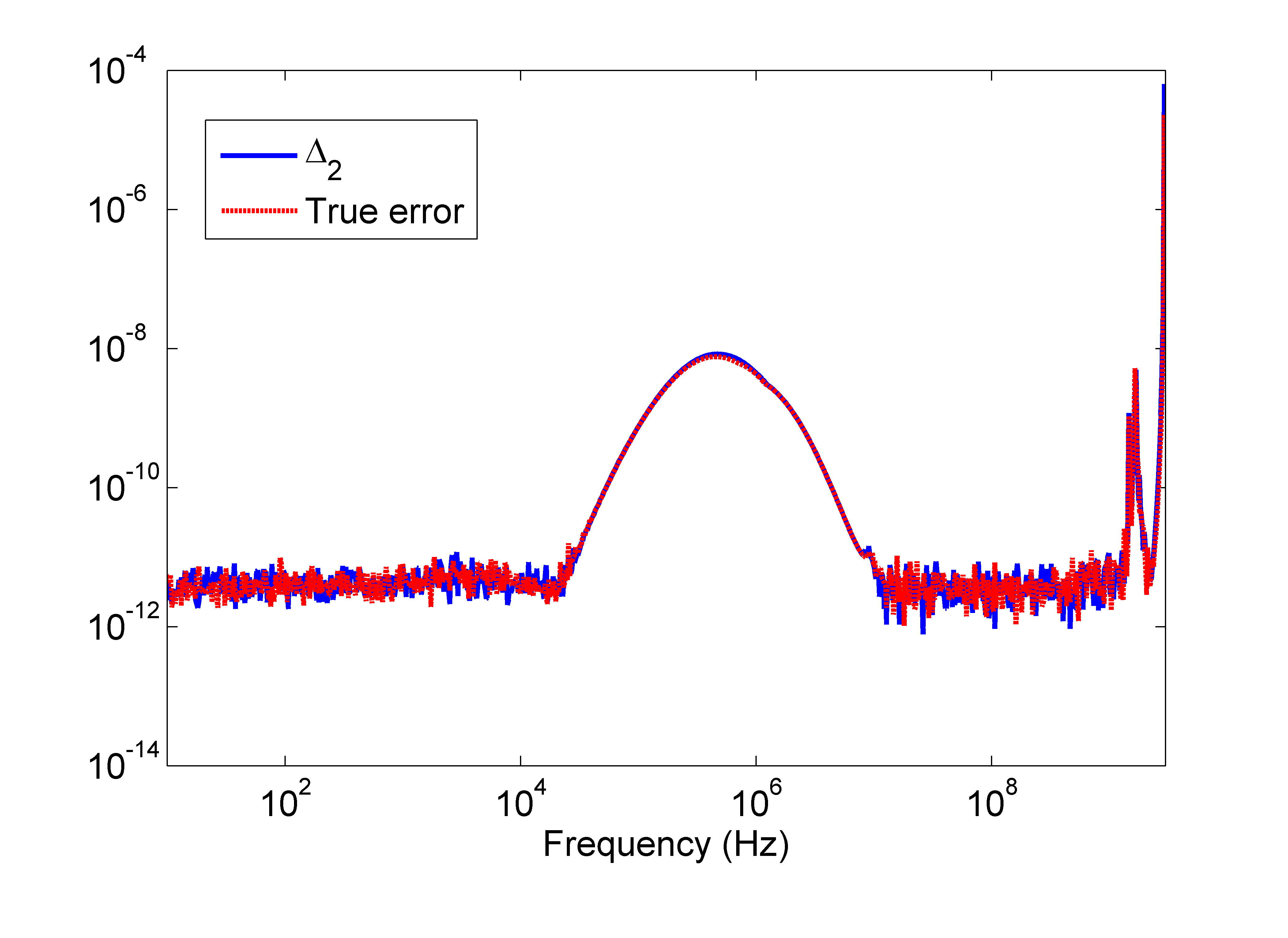}\quad
\includegraphics[width=75mm]{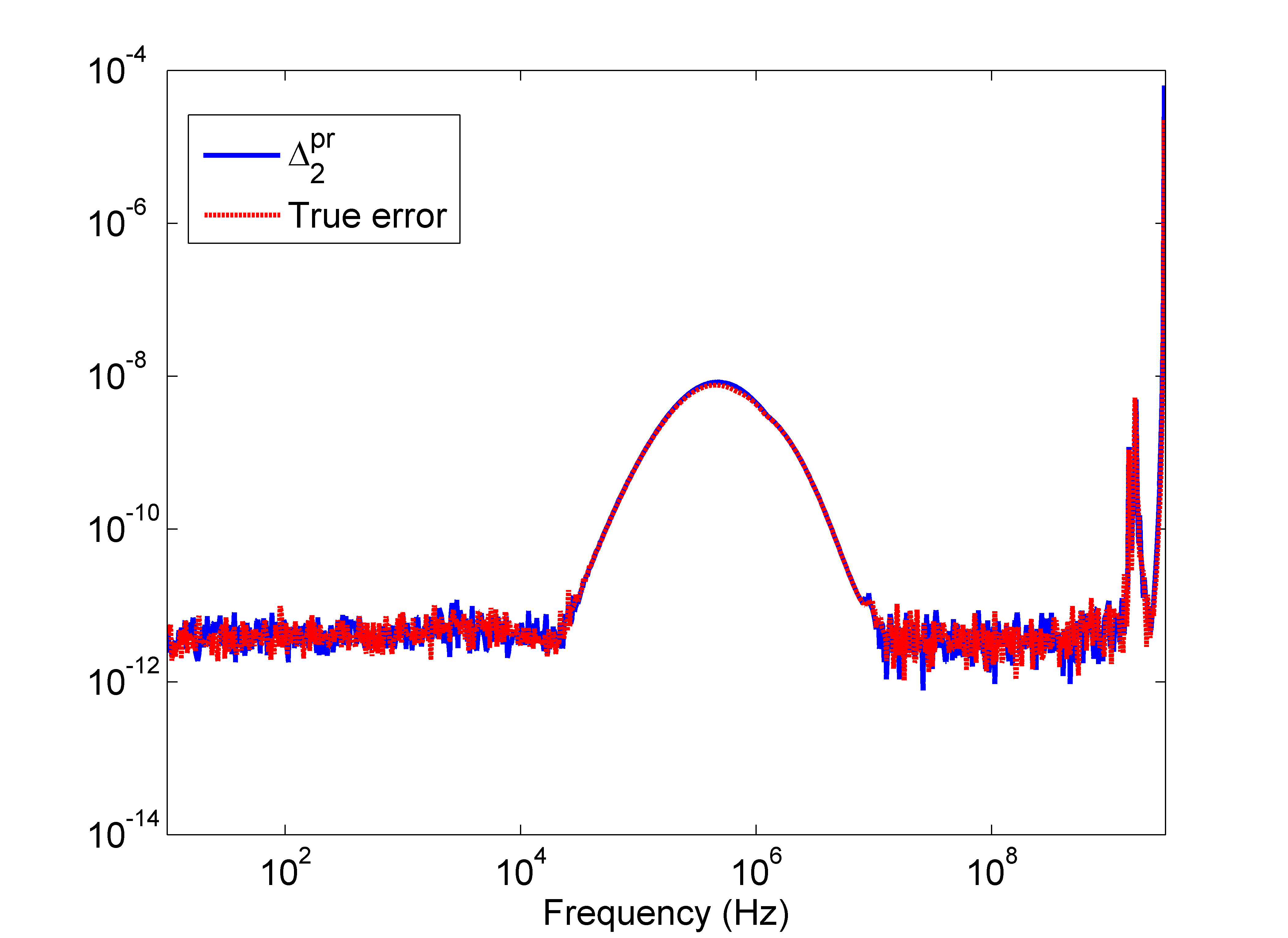}
\caption{MIMO example: $\Delta_2(s)$ and $\Delta_2^{pr}(s)$ vs. the respective true errors at 900 frequency samples .}
\label{fig:MNA4_est2pr}
\end{figure}
\begin{figure}[h]
\centering
\includegraphics[width=75mm]{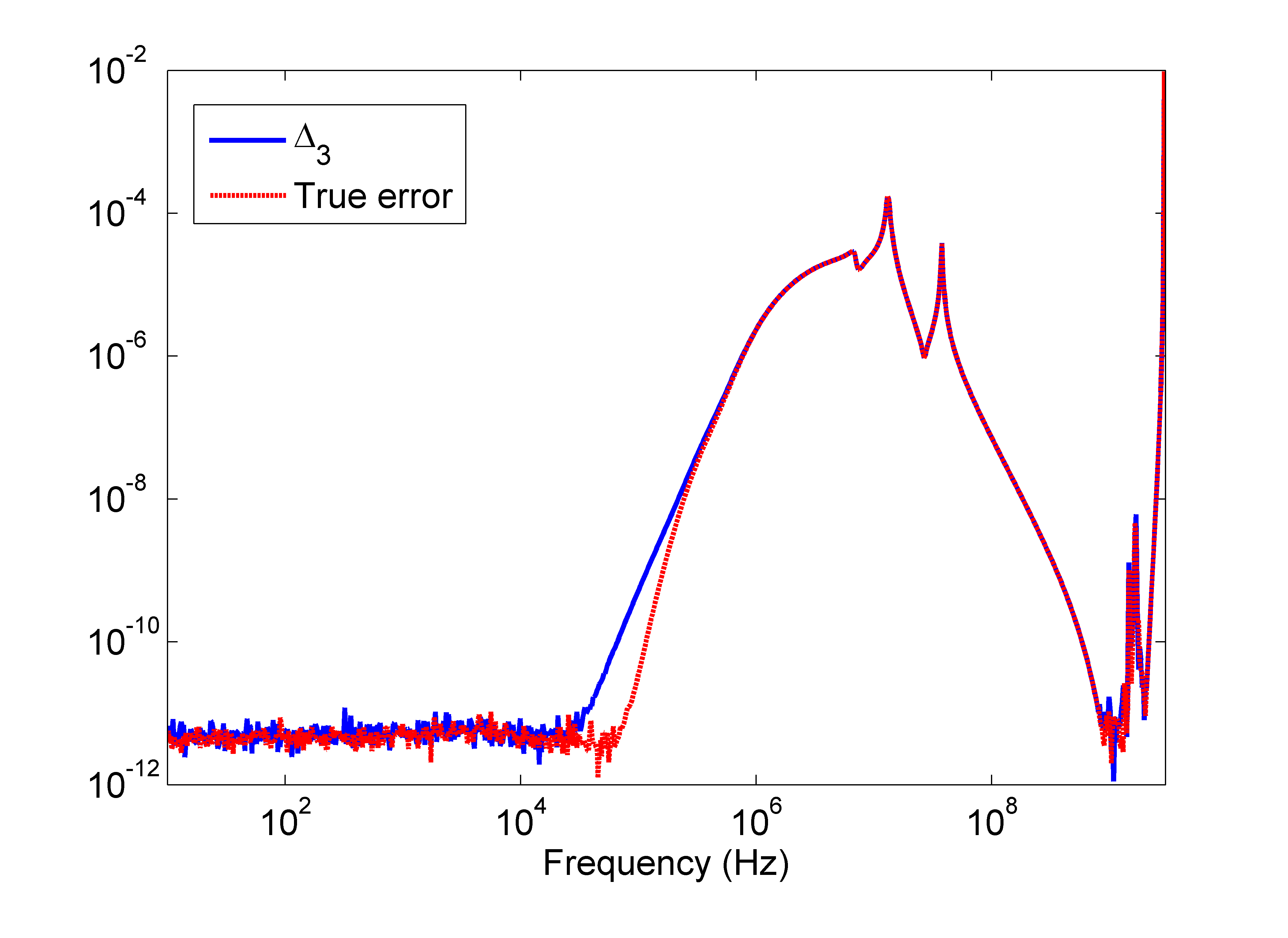}\quad
\includegraphics[width=75mm]{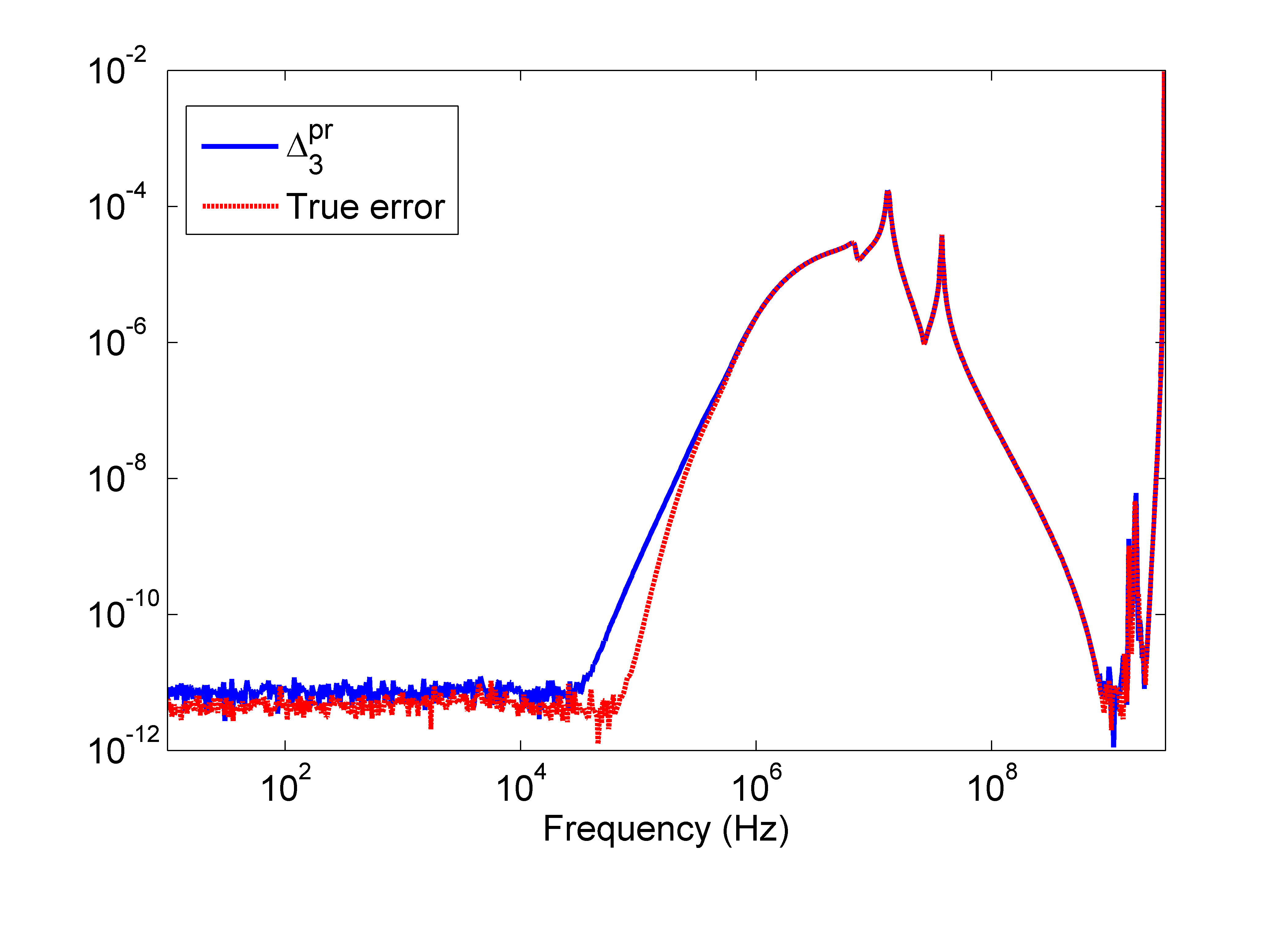}
\caption{MIMO example: $\Delta_3(s)$ and $\Delta_3^{pr}(s)$ vs. the respective true errors at 900 frequency samples.}
\label{fig:MNA4_est3pr}
\end{figure}
\subsection{Parametric example}

The micro-gyroscope model is a second-order parametric system with four parameters, 
\begin{equation*}
\begin{array}{rcl}
M(\mu)\ddot{x}(\mu,t)+D(\mu)\dot{x}(\mu,t)+T(\mu)x(\mu,t)&=&Bu(t),\\
y(\mu,t)&=&Cx(\mu,t).
\end{array}
\end{equation*}
Here, $\mu=(\theta, \alpha, \beta, d)$, $M(\mu)=(M_1+dM_2)$, $T(\mu)=(T_1+\frac{1}{d}T_2+dT_3)$, $D(\mu)=\theta(D_1+dD_2)+\alpha M(d)+\beta T(d) \in \mathbb R^{n\times n}$, $n=17,913$. The parameters are $d, \theta, \alpha, \beta$. $d \in [100\%,200\%]$, the width of the bearing, taken as the percentage of the base value, and $\theta \in [10^{-7}, 10^{-5}]$MHz, the rotation velocity along the x-axis. $\alpha$, $\beta$ define to the proportional damping~\cite{morSalEL06}.

After Laplace transform, the system in frequency domain is
\begin{equation*}
\begin{array}{rcl}
s^2M(\mu)x(\mu, s)+sD(\mu)x(\mu, s)+T(\mu)x&=&Bu_{\mathcal L}(s),\\
y(\mu, s)&=&Cx(\mu, s).
\end{array}
\end{equation*}
The above system can be rewritten into the affine form,
\begin{equation*}
\label{freqgyro}
\begin{array}{rcl}
Q(\tilde \mu)x(\tilde \mu)&=&Bu_{\mathcal L}(\tilde \mu),\\
y(\tilde \mu)&=&Cx(\tilde \mu),
\end{array}
\end{equation*}
where $Q(\tilde \mu)=T_1+\tilde \mu_1M_1+\tilde\mu_2M_2+\tilde\mu_3D_1+\tilde\mu_4D_2+\tilde\mu_5M_1+
\tilde\mu_6M_2+\tilde\mu_7T_1+\tilde\mu_8T_2+\tilde\mu_9T_3+\tilde\mu_{10}T_2+\tilde\mu_{11}T_3$. 
Here $\tilde\mu=(\tilde\mu_1,\ldots,\tilde\mu_{11})^T$ includes the newly generated parameters, $\tilde\mu_1=s^2$, $\tilde\mu_2=s^2d$, $\tilde\mu_3=s\theta$, $\tilde\mu_4=s\theta d$, $\tilde\mu_5=s\alpha$, $\tilde\mu_6=s\alpha d$, $\tilde\mu_7=s\beta$, $\tilde\mu_8={s}\beta/{d}$, $\tilde\mu_9=s\beta d$, $\tilde\mu_{10}=1/d$, $\tilde\mu_{11}=d$. 

For this example, we use 75 random samples (3 for $\theta$, 5 for $s$, 5 for $d$) to set up the training set $\Xi$ with $\beta=0$ and $\alpha=0$.
Afterwards, the ROMs are validated at a validation set $\Xi_{ver}$ including 2500 samples (5 for $\theta$, 10 for $s$, 5 for $d$), with $\beta=10^{-9}$ and $\alpha=0.1$ being nonzero. 

The results of Algorithm~\ref{alg:greedy_para} using different error estimators are listed in Tables~\ref{error_Gyro_est1pr}-\ref{error_Gyro_est3pr}. Except for $\Delta_1(\tilde \mu)$, all the other error estimators tightly estimate the true error at each iteration of the algorithm. The ROMs obtained via the error estimators are further validated at samples in $\Xi_{ver}$, and the effectivity of each is presented in Table~\ref{tab:Gyro_eff}. Again, $\Delta_1(\tilde \mu)$ is the worst. The others perform similarly well. We plot the true error of the ROMs and the corresponding error estimators in Figures~\ref{fig:Gyro_est1pr}-\ref{fig:Gyro_est3pr}. $\Delta_1(\tilde \mu)$ almost always underestimates the true error, while 
$\Delta_1^{pr}(\tilde \mu)$, $\Delta_3(\tilde \mu)$ and $\Delta_3^{pr}(\tilde \mu)$ are almost 
indistinguishable from the true error. 
\begin{table}[h]
\begin{center}
\caption{Gyroscope, $\varepsilon_{tol}=10^{-3}$, $q=3$,  $r=84 (\Delta_1)$, $r=94 (\Delta_1^{pr})$.}
\label{error_Gyro_est1pr}
\begin{tabular}{|c||c|c||c|c||} \hline
iteration $i$ & $\varepsilon_{\max} (\Delta_1)$   &  $\Delta_1(\tilde \mu^i)$ & $\varepsilon_{\max} (\Delta_1^{pr})$ &   $\Delta_1^{pr}(\tilde \mu^i$)   \\ \hline    
1       & 0.028    &    $0.04 $  &  0.028&0.025   \\ \hline  
2      &  0.006 &    0.007&   0.001& $0.006$\\  \hline
3      & 0.004 &    $3.2\times 10^{-4}$&0.003 & 0.003\\  \hline
4     &  $4\times 10^{-5}$&    $5.18\times 10^{-4}$& $3.85\times 10^{-4}$ & $3.78 \times 10^{-4}$\\  \hline
5      &  $3.34 \times 10^{-6}$&    $2.99 \times 10^{-5}$& $1.69 \times 10^{-6}$ & $1.69 \times 10^{-6}$\\  \hline
6     & $2.95 \times 10^{-7}$ &    $3.88 \times 10^{-7}$& $ 3.48 \times 10^{-7}$ & $ 3.47 \times 10^{-7}$\\  \hline
7      & $7.91 \times 10^{-8}$&    $8.03 \times 10^{-8}$&$1.39 \times 10^{-7}$ & $1.45 \times 10^{-7}$\\  \hline
8      & ---&   --- &    $8.49 \times 10^{-8}$  & $8.44 \times 10^{-8}$\\  \hline
  \end{tabular}
\end{center}
\end{table}
\begin{table}[h]
\begin{center}
\caption{Gyroscope, $\varepsilon_{tol}=10^{-3}$, $q=3$,  $r=86 (\Delta_2)$, $r=80 (\Delta_2^{pr})$.}
\label{error_Gyro_est2pr}
\begin{tabular}{|c||c|c||c|c||} \hline
iteration $i$ & $\varepsilon_{\max} (\Delta_2)$   &  $\Delta_2(\tilde \mu^i)$ & $\varepsilon_{\max} (\Delta_2^{pr})$ &   $\Delta_2^{pr}(\tilde \mu^i)$   \\ \hline    
1       & $4.53\times 10^{-4}$  &    $0.002 $  &  0.002&0.004   \\ \hline  
2      &  $4.15\times 10^{-4}$&    $6.16\times 10^{-4}$&   $4.14\times 10^{-4}$& $5.83\times 10^{-4}$\\  \hline
3      & $1.71\times 10^{-5}$ &   $8.53\times 10^{-5}$& $1.61\times 10^{-4}$ & $2.69\times 10^{-4}$\\  \hline
4     &  $8.77\times 10^{-6}$&    $8.22\times 10^{-6}$& $9.7\times 10^{-5}$ & $1.57 \times 10^{-4}$\\  \hline
5      &  $1.44 \times 10^{-6}$&    $1.07 \times 10^{-6}$& $9.80 \times 10^{-7}$ & $9.81 \times 10^{-7}$\\  \hline
6     & $3.09 \times 10^{-8}$ &    $3.41 \times 10^{-8}$& $ 1.89 \times 10^{-7}$ & $ 2.06 \times 10^{-7}$\\  \hline
7      & ---&   ---&$7.21 \times 10^{-8}$ & $8.14 \times 10^{-8}$\\  \hline
  \end{tabular}
\end{center}
\end{table}
\begin{table}[h]
\begin{center}
\caption{Gyroscope, $\varepsilon_{tol}=10^{-3}$, $q=3$,  $r=73 (\Delta_3)$, $r=83 (\Delta_3^{pr})$.}
\label{error_Gyro_est3pr}
\begin{tabular}{|c||c|c||c|c||} \hline
iteration $i$ & $\varepsilon_{\max} (\Delta_3)$   &  $\Delta_3(\tilde \mu^i)$ & $\varepsilon_{\max} (\Delta_3^{pr})$ &   $\Delta_3^{pr}(\tilde \mu^i)$   \\ \hline    
1       & 0.009    &    $0.005$  &  $5.42\times 10^{-4}$&0.002   \\ \hline  
2      &  0.009 &    0.005&   $5.60\times 10^{-4}$& $5.26\times 10^{-4}$\\  \hline
3      & $8.85\times 10^{-5}$&    $8.85\times 10^{-5}$&$9.35\times 10^{-5}$ &$6.59\times 10^{-4}$\\  \hline
4     &  $2.20\times 10^{-4}$&    $2.20\times 10^{-4}$& $5.36\times 10^{-6}$ & $5.36 \times 10^{-6}$\\  \hline
5      &  $1.78 \times 10^{-6}$&    $1.48 \times 10^{-6}$& $1.31 \times 10^{-6}$ & $1.30 \times 10^{-6}$\\  \hline
6     & $8.56 \times 10^{-8}$ &    $8.51 \times 10^{-8}$& $ 5.78 \times 10^{-7}$ & $ 5.78 \times 10^{-7}$\\  \hline
7      &---&   ---&$5.60\times 10^{-8}$& $5.59 \times 10^{-8}$\\  \hline
  \end{tabular}
\end{center}
\end{table}
\begin{table}[h]
\begin{center}
\caption{Gyroscope, effectivity of the error estimators.}
\label{tab:Gyro_eff}
\begin{tabular}{|c||c|c||c|c|} \hline
\multirow{2}{*}{Estimator}  & \multicolumn{2}{|c||}{For all $\varepsilon(s)$} &  \multicolumn{2}{|c|}{For $\varepsilon(s) \geq 10^{-11}$} \\ \cline{2-5}  
& $\min\limits_{s\in \Xi_{ver}}(\textrm{eff})$   &    $\max\limits_{s\in \Xi_{ver}}(\textrm{eff})$ & $\min\limits_{s\in \Xi_{ver}}(\textrm{eff})$   & $\max\limits_{s\in \Xi_{ver}}(\textrm{eff})$  \\ \hline 
$\Delta_1$      & 0.025   &    8.87  &  0.025&8.87\\ \hline  
$\Delta_1^{pr}$      & 0.2  &   3.65&0.2 &3.65 \\ \hline
$\Delta_2$       &  $0.38$   &    15 &0.38 &15\\  \hline
$\Delta_2^{pr}$         & $0.2$   & 3.68  &0.2& 3.68 \\ \hline 
$\Delta_3$       &  $0.34$   &     9.34 &0.34 &9.34\\  \hline
$\Delta_3^{pr}$   & 0.5   & 2 &0.5 &2\\ \hline 
  \end{tabular}
\end{center}
\end{table}
\begin{figure}[h]
\centering
\includegraphics[width=75mm]{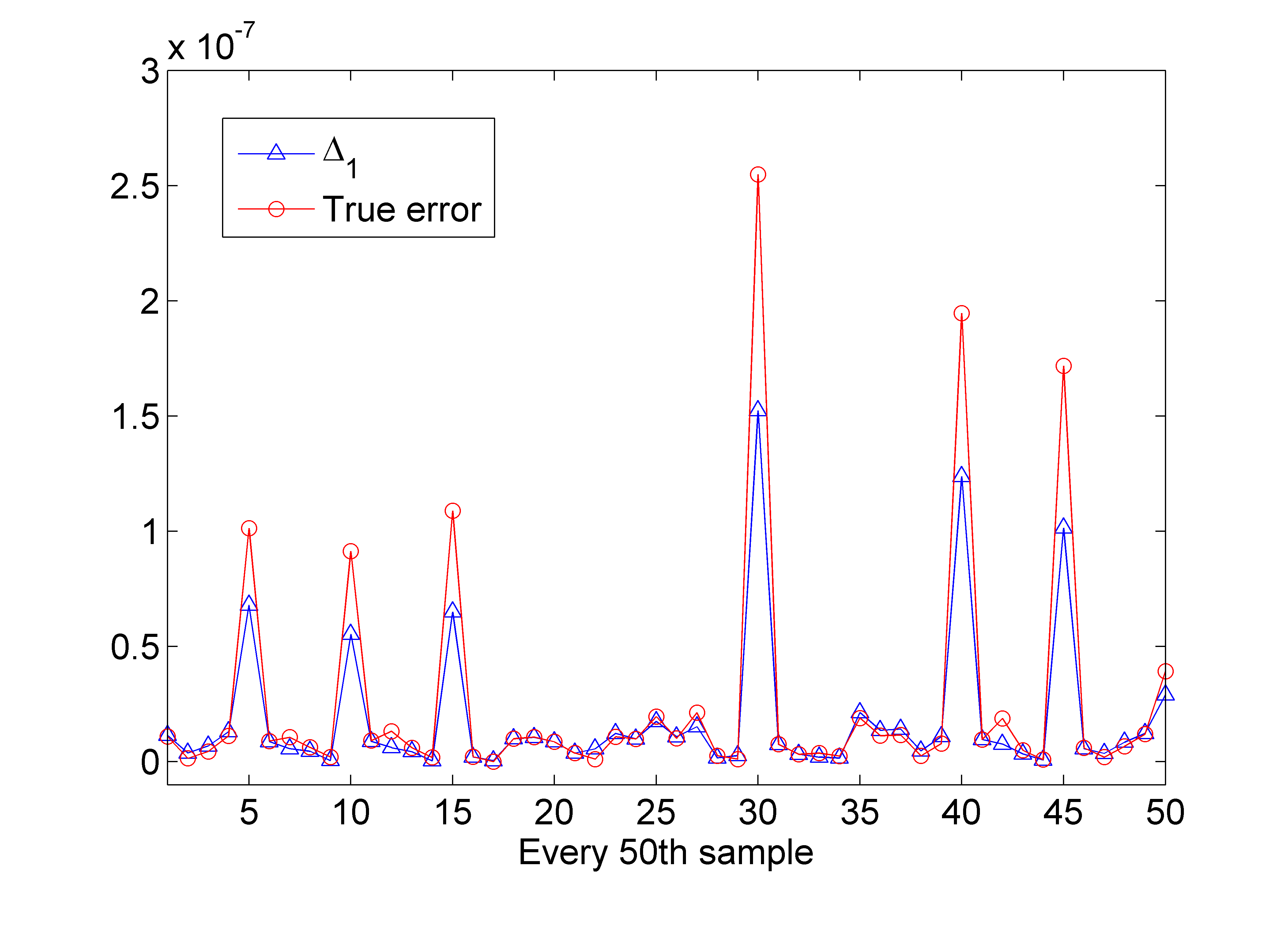}\quad
\includegraphics[width=75mm]{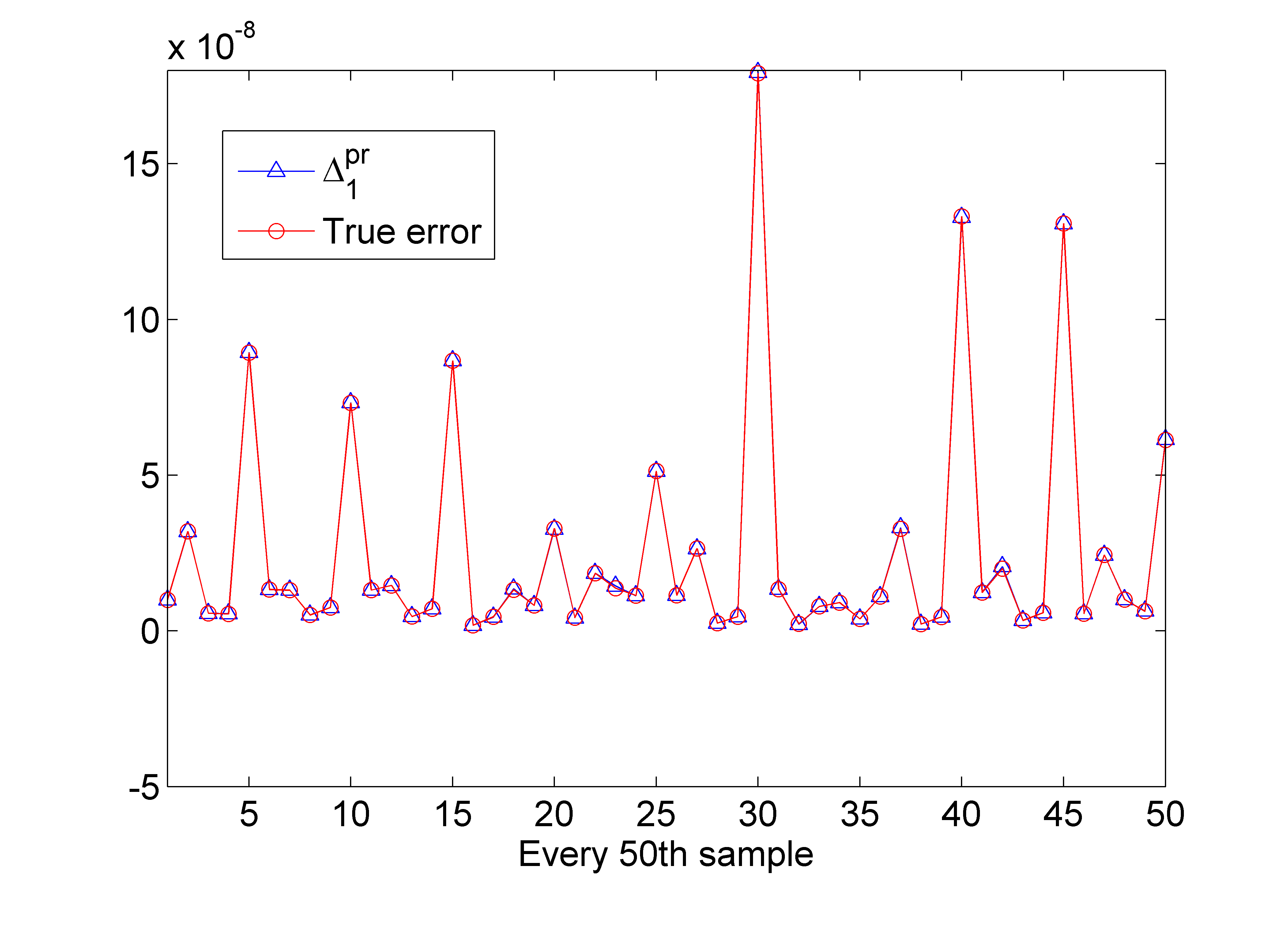} 
\caption{Gyroscope: $\Delta_1(\tilde \mu)$ and $\Delta_1^{pr}(\tilde \mu)$ vs. the respective true errors at 2500 parameter samples.}
\label{fig:Gyro_est1pr}
\end{figure}
\begin{figure}[h]
\centering
\includegraphics[width=75mm]{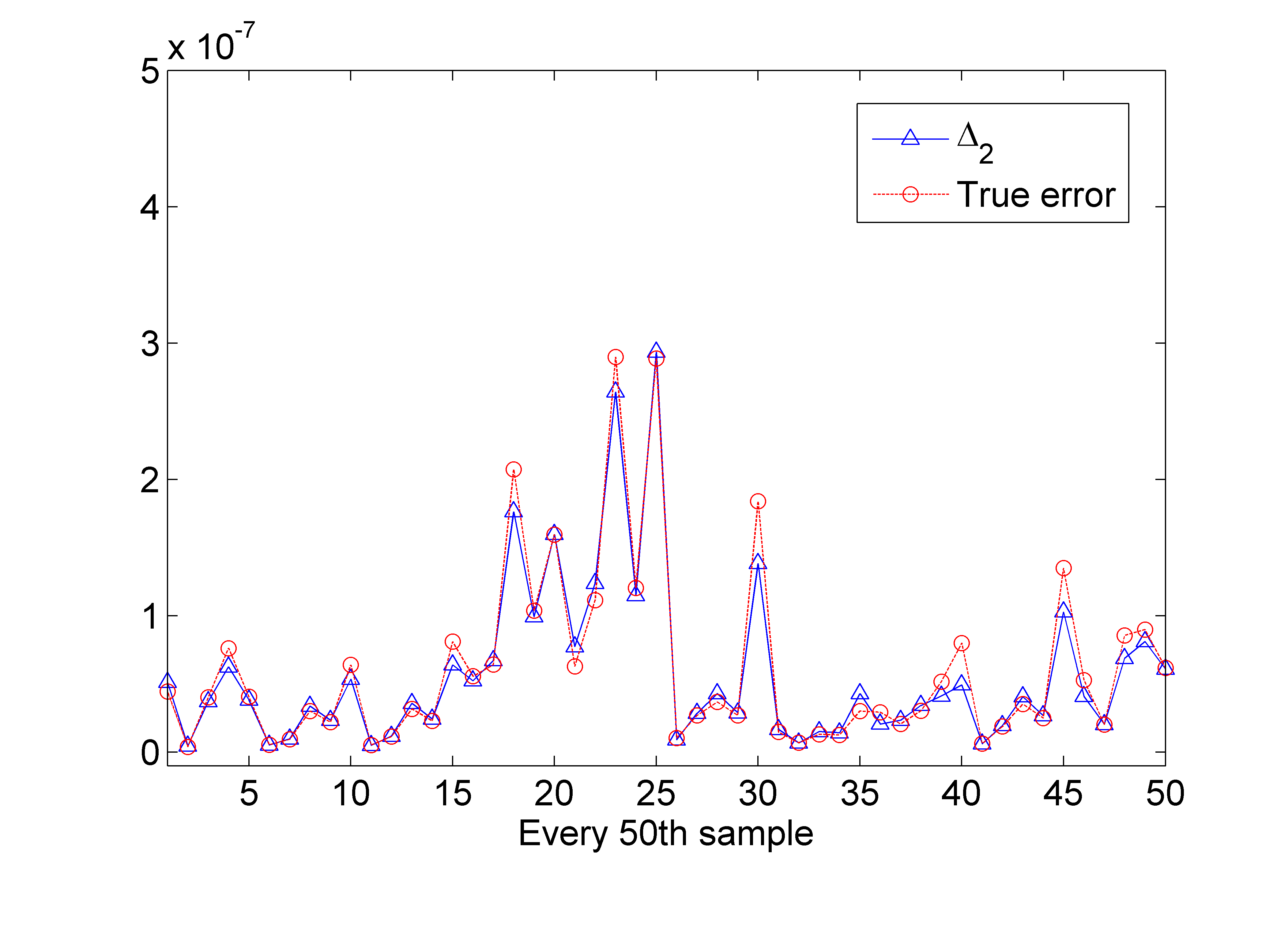}\quad
\includegraphics[width=75mm]{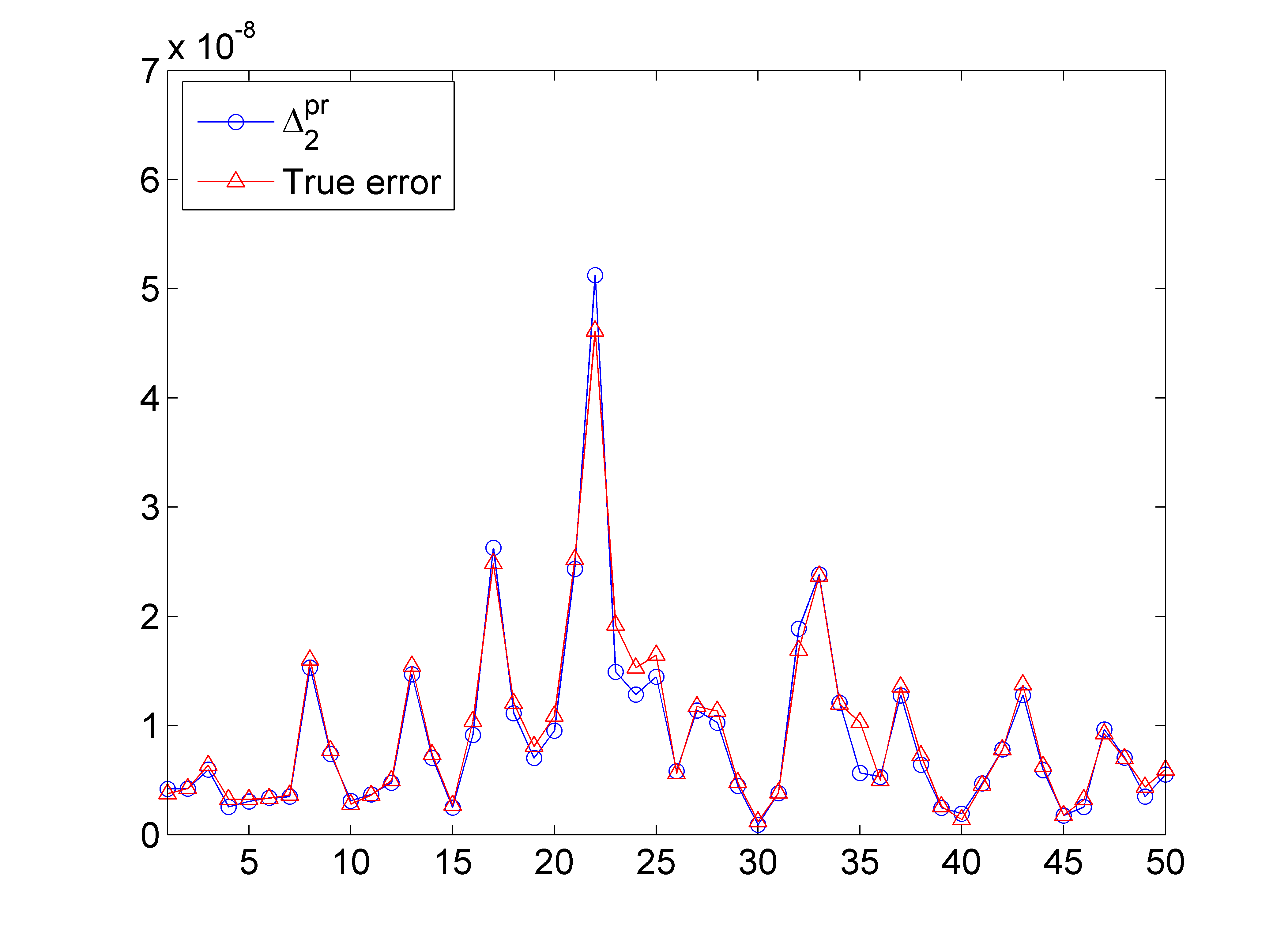}
\caption{Gyroscope: $\Delta_2(\tilde \mu)$ and $\Delta_2^{pr}(\tilde \mu)$ vs. the respective true errors at 2500 parameter samples.}
\label{fig:Gyro_est2pr}
\end{figure}
\begin{figure}[h]
\centering
\includegraphics[width=75mm]{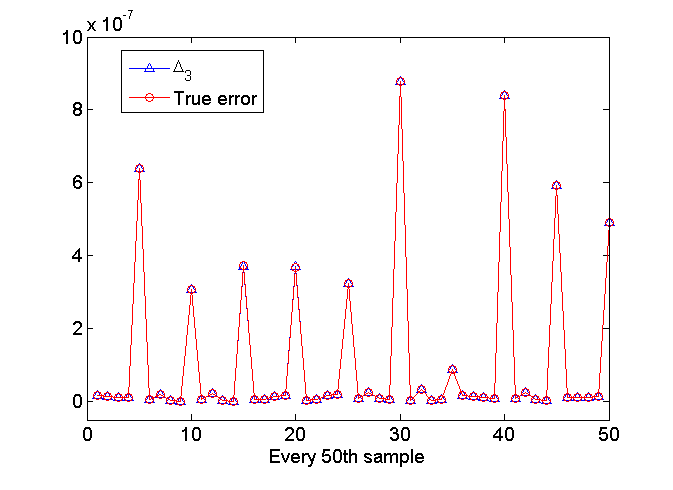}\quad
\includegraphics[width=75mm]{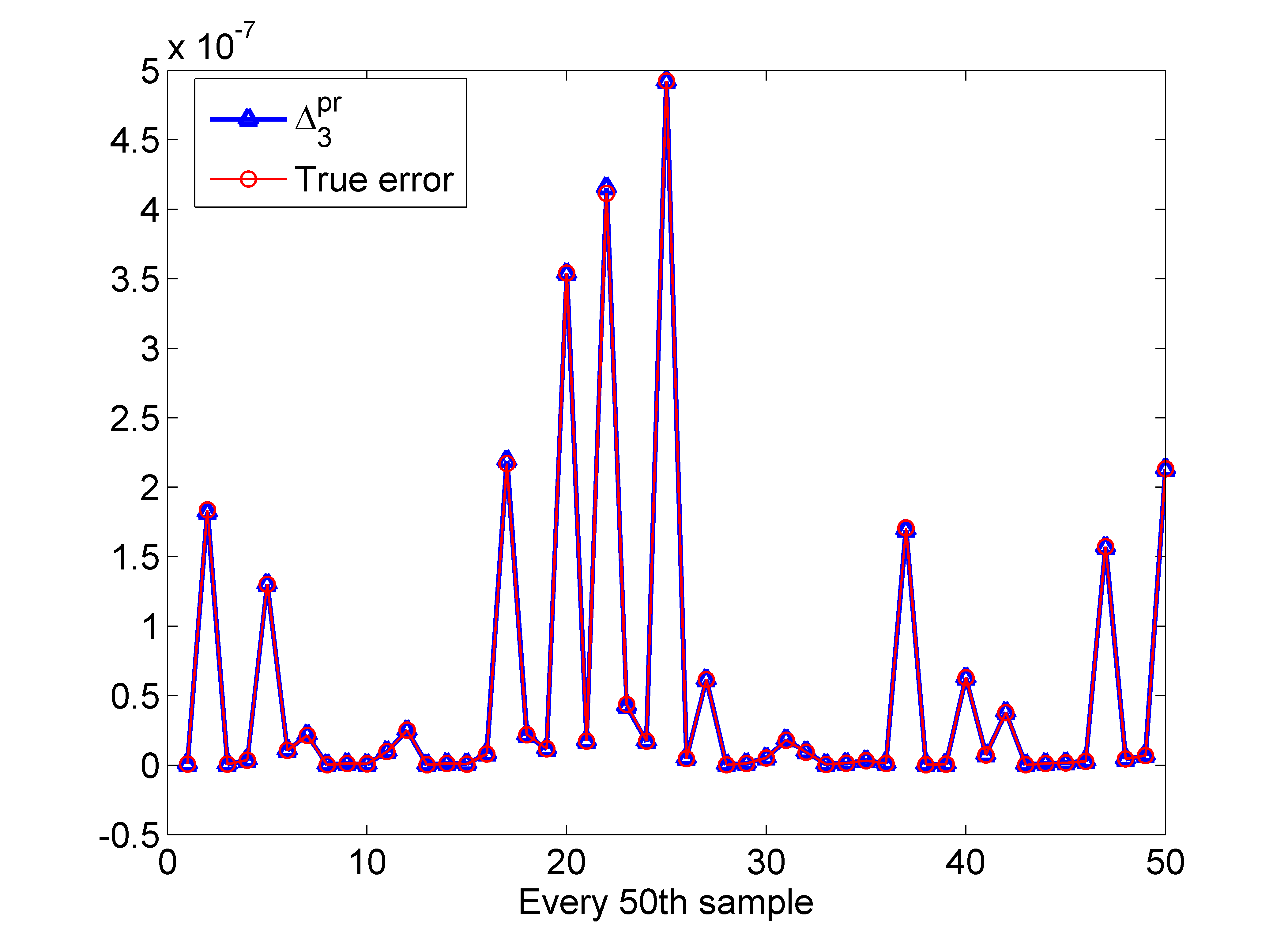}
\caption{Gyroscope: $\Delta_3(\tilde \mu)$ and $\Delta_3^{pr}(\tilde \mu)$ vs. the respective true errors at 2500 parameter samples.}
\label{fig:Gyro_est3pr}
\end{figure}
\subsection{Performances of $\Delta_1(\tilde \mu)$, $\Delta_2(\tilde \mu)$ and $\Delta_2^{pr}(\tilde \mu)$ using Algorithms~\ref{alg:greedy_nonpara_im}-\ref{alg:greedy_para_im}}

In this subsection, we show the results of Algorithms~\ref{alg:greedy_nonpara_im}-\ref{alg:greedy_para_im} for symmetric systems, where the expansion points for $V_{du}$ are selected differently from those for $V$. The results are listed in Tables~\ref{tab:RLC_eff_im}-\ref{tab:Gyro_eff_im}. 
\begin{table}[h]
\begin{center}
\caption{Algorithm~\ref{alg:greedy_nonpara_im}: RLCtree, effectivity of the error estimators.}
\label{tab:RLC_eff_im}
\begin{tabular}{|c||c|c||c|c|} \hline
\multirow{2}{*}{Estimator}  & \multicolumn{2}{|c||}{For all $\varepsilon(s)$} &  \multicolumn{2}{|c|}{For $\varepsilon(s) \geq 10^{-11}$} \\ \cline{2-5}  
& $\min\limits_{s\in \Xi_{ver}}(\textrm{eff})$   &    $\max\limits_{s\in \Xi_{ver}}(\textrm{eff})$ & $\min\limits_{s\in \Xi_{ver}}(\textrm{eff})$   & $\max\limits_{s\in \Xi_{ver}}(\textrm{eff})$  \\ \hline 
$\Delta_1$      & $3.4488 \times 10^{-4}$   &    $38$  & 0.05&6.5\\ \hline  
$\Delta_2$       &  $0.01$   &     $25$ &0.7 &25\\  \hline
$\Delta_2^{pr}$         & $0.004$   &    $244$&1 &25  \\ \hline 
  \end{tabular}
\end{center}
\end{table}
\begin{table}[h]
\begin{center}
\caption{Algorithm~\ref{alg:greedy_nonpara_im}: MIMO example, effectivity of the error estimators.}
\label{tab:MIMO_eff_im}
\begin{tabular}{|c||c|c||c|c|} \hline
\multirow{2}{*}{Estimator}  & \multicolumn{2}{|c||}{For all $\varepsilon(s)$} &  \multicolumn{2}{|c|}{For $\varepsilon(s) \geq 10^{-11}$} \\ \cline{2-5}  
& $\min\limits_{s\in \Xi_{ver}}(\textrm{eff})$   &    $\max\limits_{s\in \Xi_{ver}}(\textrm{eff})$ & $\min\limits_{s\in \Xi_{ver}}(\textrm{eff})$   & $\max\limits_{s\in \Xi_{ver}}(\textrm{eff})$  \\ \hline 
$\Delta_1$      & 0.14   &    46  &  0.14&46\\ \hline  
$\Delta_2$       &  $0.2$   &    15 &0.1 &9\\  \hline
$\Delta_2^{pr}$         & $0.32$   & 164  &0.32& 75\\ \hline 
  \end{tabular}
\end{center}
\end{table}
\begin{table}[h]
\begin{center}
\caption{Algorithm~\ref{alg:greedy_para_im}: Gyroscope, effectivity of the error estimators.}
\label{tab:Gyro_eff_im}
\begin{tabular}{|c||c|c||c|c|} \hline
\multirow{2}{*}{Estimator}  & \multicolumn{2}{|c||}{For all $\varepsilon(s)$} &  \multicolumn{2}{|c|}{For $\varepsilon(s) \geq 10^{-11}$} \\ \cline{2-5}  
& $\min\limits_{s\in \Xi_{ver}}(\textrm{eff})$   &    $\max\limits_{s\in \Xi_{ver}}(\textrm{eff})$ & $\min\limits_{s\in \Xi_{ver}}(\textrm{eff})$   & $\max\limits_{s\in \Xi_{ver}}(\textrm{eff})$  \\ \hline 
$\Delta_1$      & 0.096   &    28 &  0.096&28\\ \hline  
$\Delta_2$       &  $0.35$   &    11 &0.35 &11\\  \hline
$\Delta_2^{pr}$         & $0.22$   & 3.68  &0.22& 3.68 \\ \hline 
 \end{tabular}
\end{center}
\end{table}
Comparing Tables~\ref{tab:RLC_eff_im},~\ref{tab:MIMO_eff_im},~\ref{tab:Gyro_eff_im} with
Tables~\ref{tab:RLC_eff},~\ref{tab:MIMO_eff},~\ref{tab:Gyro_eff}, respectively, we see that the performance of $\Delta_1(\tilde \mu)$ is improved in general, those of $\Delta_2(\tilde \mu)$, and $\Delta_2^{pr}(\tilde \mu)$ are only partially improved. The performance of $\Delta_2(\tilde \mu))$ is improved, especially for the RLC tree example. However, the performance of $\Delta_2^{pr}(\tilde \mu)$ does not become uniformly better, especially for the MIMO example. Although $\Delta_1(s)$ behaves better when using Algorithm~\ref{alg:greedy_nonpara_im} and~\ref{alg:greedy_para_im}, it is still worse than its upper bound $\Delta_2(\tilde \mu)$ or $\Delta_2^{pr}(\tilde \mu)$.

\clearpage

\section{Conclusions}
We propose some a posteriori error estimators for the transfer function error of ROMs that are obtained by any (Petrov-)Galerkin-type MOR method.  
Detailed simulation comparison demonstrates the performance of each. It is clear that either $\Delta_r(\tilde \mu)$ or $\Delta_1(\tilde \mu)$ is not a good error estimator for all the examples and therefore is not recommended as a reliable error estimator. All others 
perform similarly, especially the primal version of $\Delta_1(\tilde \mu)$: $\Delta_1^{pr}(\tilde \mu)$ behaves unexpectedly well and is almost as good as its bounds $\Delta_3(\tilde \mu)$ and $\Delta_3^{pr}(\tilde \mu)$ for all the examples. Among the robust error estimators $\Delta_2(\tilde \mu)$, $\Delta_2^{pr}(\tilde \mu)$, $\Delta_1^{pr}(\tilde \mu)$, $\Delta_3(\tilde \mu)$ and $\Delta_3^{pr}(\tilde \mu)$, the estimator $\Delta_1^{pr}$ needs the least computational cost, since only two ROMs (constructed by $V, V_{r_{pr}}$) need to be computed. For nearly symmetric systems, $\Delta_2(\tilde \mu)$ and its variant $\Delta_2^{pr}(\tilde \mu)$ are not really improved for all the examples when choosing different expansion points for $V$ and $V_{du}$, i.e., when using Algorithms~\ref{alg:greedy_nonpara_im} and~\ref{alg:greedy_para_im}. As future work, more theoretical analysis and numerical simulations might be explored to further explain the numerical behaviors of the proposed error estimators. 

 \section*{Acknowledgment}
Part of this material is based upon work supported by the National
  Science Foundation under Grant No. DMS-1439786 and by the Simons
  Foundation Grant No. 50736 while Feng and Benner were in residence at the
  Institute for Computational and Experimental Research in Mathematics
  in Providence, RI, during the "Model and
  dimension reduction in uncertain and dynamic systems" program.

\bibliographystyle{abbrv}
\bibliography{mor,refs}

\end{document}